%% file: kato.tex
\documentclass[12pt,a4paper,leqno]{article}
\usepackage[francais]{babel}
\usepackage[T1]{fontenc}
\usepackage{latexsym}
\usepackage{amsmath}
\usepackage{amssymb, amscd, amsthm, mathrsfs}
\usepackage[all]{xy}
\usepackage[dvips]{graphicx}
\usepackage{makeidx}
\newtheorem{theo}{{Th\'eor\`eme}}[section]
\newtheorem{coro}[theo]{{Corollaire}}
\newtheorem{lemma}[theo]{{Lemme}}
\newtheorem{prop}[theo]{Proposition}

\theoremstyle{remark}
\newtheorem{remark}[theo]{\textbf{Remarque}}

\theoremstyle{definition}
\newtheorem{defn}[theo]{D\'efinition}
\newtheorem{example}[theo]{Exemple}
\newtheorem{recette}[theo]{\textbf{Recette}}

\newcommand{\ra}{\rightarrow}

\newcommand{\ol}{\overline}
\newcommand{\immouv}[1][r]
   {\ar@{}[#1] |*[o][F]{\hbox{%
         \vrule width 1.5mm height 0pt depth 0pt%
         \vrule width 0pt height .75mm depth .75mm%
         }}
         \ar@{^{(}->}[#1]}

\newcommand{\cA}{\mathcal{A}}
\newcommand{\cB}{\mathcal{B}}
\newcommand{\cC}{\mathcal{C}}

\newcommand{\cE}{\mathcal{E}}
\newcommand{\cF}{\mathcal{F}}
\newcommand{\cG}{\mathcal{G}}
\newcommand{\cH}{\mathcal{H}}

\newcommand{\cK}{\mathcal{K}}

\newcommand{\cM}{\mathcal{M}}

\newcommand{\cO}{\mathcal{O}}
\newcommand{\cP}{\mathcal{P}}
\newcommand{\cR}{\mathcal{R}}
\newcommand{\cS}{\mathcal{S}}
\newcommand{\cU}{\mathcal{U}}

\newcommand{\A}{\mathbb A}
\newcommand{\C}{\mathbb C}
\newcommand{\B}{\mathbb B}

\newcommand{\K}{\mathbb K}
\newcommand{\N}{\mathbb N}

\newcommand{\bP}{\mathbb{P}}
\newcommand{\Q}{\mathbb Q}
\newcommand{\R}{\mathbb R}

\newcommand{\Z}{\mathbb Z}

\newcommand{\bM}{\mathbf M}
\newcommand{\bR}{\mathbf R}

\newcommand{\fD}{\mathfrak D}

\newcommand{\fF}{\mathfrak F}

\newcommand{\fK}{\mathfrak K}

\newcommand{\rA}{\mathrm A}
\newcommand{\rB}{\mathrm B}
\newcommand{\rC}{\mathrm C}
\newcommand{\rD}{\mathrm D}
\newcommand{\rE}{\mathrm E}

\newcommand{\rH}{\mathrm H}

\newcommand{\rP}{\mathrm P}

\newcommand{\rT}{\mathrm T}

\renewcommand{\ker}{\mathrm {Ker} }

\renewcommand{\Re}{\mathrm Re}

\newcommand{\Dir}{\mathbf{Dir}}

\DeclareMathOperator{\Res}{\mathrm Res}
\DeclareMathOperator{\GL}{\mathrm GL}
\DeclareMathOperator{\SL}{\mathrm SL}
\DeclareMathOperator{\LC}{\mathrm LC}

\DeclareMathOperator{\Fil}{\mathrm{Fil}}

\DeclareMathOperator{\Frac}{\mathrm Frac}
\DeclareMathOperator{\Gal}{\mathrm Gal}

\DeclareMathOperator{\Sym}{\mathrm{Sym}}
\DeclareMathOperator{\tr}{\mathrm Tr}

\DeclareMathOperator{\plim}{\varprojlim}

\newcommand{\alg}{\mathrm{alg}}
\newcommand{\an}{\mathrm{an}}
\newcommand{\con}{\mathrm{cong}}

\newcommand{\cycl}{\mathrm{cycl}}
\newcommand{\dR}{\mathrm{dR}}

\newcommand{\spe}{\mathrm{sp}}

\newcommand{\Inf}{\mathrm{inf}}

\newcommand{\res}{\mathrm{res}}

\newcommand{\Iw}{\mathrm{Iw}}
\newcommand{\Exp}{\mathrm{Exp}}

\newcommand{\z}{\zeta}

\newcommand{\eps}{\epsilon}

\begin{document}

\title{Le syst\`eme d'Euler de Kato }
\author{Shanwen \textsc{WANG} }
\date{}
\maketitle

\input{notation.tex}
\input{systeme.tex}
\input{anneaux.tex}
\input{cohomologie.tex}
\input{Loi.tex}

--------------------------------------\\

\author{Shanwen \textsc{WANG}\\ Dipartimento di Matematica Pura e Applicata, \\ Via Trieste 63, \\ 35121, Padova, Italy \\ Email: swang@math.unipd.it }
\end{document}

%% file: notation.tex
\begin{abstract}
Ce texte est consacr\'e au syst\`eme d'Euler de Kato,
construit \`a partir des unit\'es modulaires, et \`a son
image par l'application exponentielle duale (loi de r\'eciprocit\'e
explicite de Kato).  La pr\'esentation que nous en donnons
est sensiblement diff\'erente de la pr\'esentation originelle
de Kato.
\end{abstract}
\def\abstractname{Abstract}
\begin{abstract}
This article is devoted to Kato's Euler system, which is constructed from modular unites, and to its image by the dual exponential map (so called Kato's reciprocity law). The presentation in this article is different from Kato's oringinal one, and the dual exponential map in this article is a modification of Colmez's construction in his Bourbaki talk. 
\end{abstract}
\tableofcontents

\section{ Notations et Introduction }

 \subsection{Notations}
On note $\overline\Q$ la cl\^oture alg\'ebrique de $\Q$ dans $\C$,
et on fixe, pour tout nombre premier~$p$, une cl\^oture
alg\'ebrique $\overline\Q_p$ de $\Q_p$, ainsi qu'un plongement de
$\overline\Q$ dans $\overline\Q_p$.

Si $N\in\N$, on note $\zeta_N$ la racine $N$-i\`eme
$e^{2i\pi/N}\in\overline\Q$ de l'unit\'e, et on note
$\Q^{\rm cycl}$ l'extension cyclotomique de $\Q$,
r\'eunion des $\Q(\zeta_N)$, pour $N\geq 1$, ainsi que $\Q^{\rm cycl}_p$ l'extension cyclotomique de $\Q_p$, r\'eunion de $\Q_p(\z_N)$, pour $N\geq 1$.
 
 \subsubsection{Objets ad\'eliques}
 Soient $\cP$ l'ensemble des premiers de $\Z$ et $\hat{\Z}$ le compl\'et\'e profini de $\Z$, alors
$\hat{\Z}=\prod_{p\in\cP}\Z_p$. Soit $\Q\otimes\hat{\Z}$ l'anneau
des ad\`eles de $\Q$ ( le produit restreint des $\Q_p$ par rapport
aux sous-anneaux $\Z_p$ de $\Q_p$ ). Quel que soit
$x\in\Q\otimes\hat{\Z}$, on note $x_p$ ( resp. $x^{]p[}$ ) la
composante de $x$ en $p$ ( resp. en dehors de $p$ ). Notons
$\hat{\Z}^{]p[}=\prod_{l\neq p}\Z_l$. On a donc
$\hat{\Z}=\Z_p\times\hat{\Z}^{]p[}$. Cela induit les d\'ecompositions
suivantes: pour tout  $d\geq 1$,
\[
\bM_d(\Q\otimes\hat{\Z})=\bM_d(\Q_p)\times\bM_d(\Q\otimes\hat{\Z}^{]p[})
\text{ et }
\GL_d(\Q\otimes\hat{\Z})=\GL_d(\Q_p)\times\GL_d(\Q\otimes\hat{\Z}^{]p[}).\]
On d\'efinit les sous-ensembles suivants de $\Q\otimes\hat{\Z}$ et
$\bM_2(\Q\otimes\hat{\Z})$:

\begin{align*}
\hat{\Z}^{(p)}=\Z_p^{*}\times\hat{\Z}^{]p[} &\text{ et
}
\bM_{2}(\hat{\Z})^{(p)}=\GL_2(\Z_p)\times\bM_2(\hat{\Z}^{]p[}), \\
(\Q\otimes\hat{\Z})^{(p)}=\Z_p^{*}\times(\Q\otimes\hat{\Z}^{]p[})
&\text{ et }
\bM_{2}(\Q\otimes\hat{\Z})^{(p)}=\GL_2(\Z_p)\times\bM_2(\Q\otimes\hat{\Z}^{]p[}).
\end{align*}

\subsubsection{Actions de groupes}
Soient $X$ un espace topologique localement profini, $V$ un
$\Z$-module. On note $\LC_c(X,V)$ le module des fonctions localement
constantes sur  $X$ \`a valeurs dans $V$ dont le support est compact
dans $X$. On note $\fD_{\alg}(X,V)$ l'ensemble des distributions alg\'ebriques sur
$X$ \`a valeurs dans $V$, c'est \`a dire des applications
$\Z$-lin\'eaires de $\LC_c(X,\Z)$ \`a valeurs dans $V$. On note $\int_X\phi\mu$ la valeur de $\mu$
sur $\phi$, o\`u $\mu\in\fD_{\alg}(X,V)$ et $\phi\in \LC_c(X,\Z)$.

Soit $G$ un groupe localement profini, agissant contin\^ument \`a
droite sur $X$ et $V$ (c.-\`a-d. quels que soient $g_1,g_2\in G, x\in
X,$ on a $(x*g_1)*g_2=x*(g_1g_2)$). On munit $\LC_c(X,\Z)$ et
$\fD_{\alg}(X,V)$ d'actions de $G$ \`a droite comme suit:

si $g\in G, x\in X,\phi\in\LC_c(X,\Z), \mu\in\fD_{\alg}(X,V),$ alors
\begin{equation}\label{actiondis} (\phi*g)(x)=\phi(x*g^{-1}) \text{ et } \int_{X}\phi(\mu*g)=\bigl(\int_{X}(\phi*g^{-1})\mu\bigr)*g.
\end{equation}

\subsubsection{Formes modulaires}
Soient $A$ un sous-anneau de $\C$ et $\Gamma$ un sous-groupe
d'indice fini de $\SL_2(\Z)$. On note $\cM_k(\Gamma,\C)$ le
$\C$-espace vectoriel des formes modulaires de poids $k$ pour
$\Gamma$. On note aussi $\cM_{k}(\Gamma,A)$ le sous $A$-module de
$\cM_k(\Gamma,\C)$ des formes modulaires dont le $q$-d\'eveloppement
est \`a coefficients dans $A$. On pose
$\cM(\Gamma,A)=\oplus_{k=0}^{+\infty}\cM_k(\Gamma,A)$. Et on note
$\cM_k(A)$ ( resp. $\cM(A)$ ) la r\'eunion des $\cM_k(\Gamma,A)$
(resp. $\cM(\Gamma,A)$), o\`u $\Gamma$ d\'ecrit tous les
sous-groupes d'indice fini de $\SL_2(\Z)$. On peut munir l'alg\`ebre
$\cM(\C)$ d'une action de $\GL_2(\Q)_{+}=\{\gamma\in\GL_2(\Q)|\det\gamma>0\}$ de la fa\c{c}on suivante:
\begin{equation}\label{etoi}
f*\gamma=(\det\gamma)^{1-k}f_{|_k}\gamma, \text{ pour }
f\in\cM_k(\C) \text{ et } \gamma\in\GL_2(\Q)_{+},
\end{equation}
o\`u $f_{|_k}\gamma$ est l'action modulaire usuelle de $\GL_2(\R)_{+}$ (voir section $\S\ref{formes}$ la formule (\ref{slash})).


\begin{defn}Soient $N\geq 1$ et $\Gamma_N=\{\bigl(\begin{smallmatrix}a & b\\ c & d\end{smallmatrix}\bigr)\in\SL_2(\Z), \bigl(\begin{smallmatrix}a & b\\ c & d\end{smallmatrix}\bigr)\equiv \bigl(\begin{smallmatrix}1 & 0\\ 0 & 1\end{smallmatrix}\bigr)[N]\}$. Le groupe $\Gamma_N$ est un sous-groupe de $\SL_2(\Z)$ d'indice fini.
On dit qu'un sous-groupe $\Gamma$ de $\SL_2(\Z)$ est un sous-groupe
de congruence s'il contient $\Gamma_N$ pour un certain $N\geq 1$.
\end{defn}
\begin{example}
Les sous-groupes
$\Gamma_0(N)=\{\gamma\in\SL_2(\Z)|\gamma\equiv(\begin{smallmatrix}*&*\\0&*\end{smallmatrix})\mod
N\}$ et
$\Gamma_1(N)=\{\gamma\in\SL_2(\Z)|\gamma\equiv(\begin{smallmatrix}1&*\\0&1\end{smallmatrix})\mod
N\}$ sont des sous-groupes de congruences.
\end{example}

On d\'efinit de m\^eme:
\[\cM^{\con}_k(A)=\bigcup\limits_{\substack{\Gamma \text { sous-groupe de congruence } }}\cM_k(\Gamma,A)\text{ et }
\cM^{\con}(A)=\oplus_{k=0}^{+\infty}\cM_k^{\con}(A).\]

Soit $K$ un sous-corps de $\C$ et soit $\ol{K}$ la cl\^oture alg\'ebrique de $K$. On note $\Pi_K$ le groupe des
automorphismes de $\cM(\ol{K})$ sur $\cM(\SL_2(\Z),K)$; c'est un
groupe profini.
On note
$\Pi_{\Q}^{'}$ le groupe des automorphismes
de $\cM(\overline\Q)$ engendr\'e par $\Pi_{\Q}$ et
$\GL_2(\Q)_{+}$.
Plus g\'en\'eralement, si $S\subset\cP$ est fini, on note $\Pi_{\Q}^{(S)}$ le sous-groupe de $\Pi_{\Q}^{'}$ engendr\'e par $\Pi_{\Q}$ et $\GL_2(\Z^{(S)})_{+}$, o\`u $\Z^{(S)}$ est le sous-anneau de $\Q$ obtenu en inversant tous les nombres premiers qui n'appartiennent pas \`a $S$.

\subsubsection{Objets $p$-adiques}
Soit $q$ une variable.
On note $\fK^+=\Q_p\{\frac{q}{p}\}$ l'anneau des fonctions analytiques
sur le disque ferm\'e $\{q\in\C_p:v_p(q)\geq 1\}$,
que l'on munit de la valuation spectrale $v_p$
(i.e. $v_p(f)=\inf_{v_p(q)\geq 1}v_p(f(q))$ ).
On note $\fK$ le compl\'et\'e du corps des fractions de $\fK^+$.
On fixe une cl\^oture alg\'ebrique $\ol{\fK}$ de $\fK$ et
on note $\ol{\fK}^+$ la cl\^oture int\'egrale de $\fK^+$ dans $\ol{\fK}$.
On note $\cG_\fK$ le groupe galois de $\ol{\fK}$ sur $\fK$.

On choisit un syst\`eme compatible $(q_M)_{M\geq 1}$
de racines $M$-i\`emes de $q$ dans $\ol{\fK}^+$
(i.e. $q_{NM}^N=q_M$, pour tous
$N,M\geq 1$.)
On note $F_M=\Q_p(\z_{M})$, $F_{Mp^{\infty}}=\cup_{n\geq 1}F_{Mp^{n}}$ et $F_{\infty}=\cup_{M\geq 1}F_M$.
Soit $\fK_M=\fK[q_{M},\z_M]$ ; c'est une extension galoisienne de $\fK$.
On note $\fK_{\infty}$ la r\'eunion des $\fK_M$ pour tous $M\geq 1$,
 $\fK_{\infty}^+$ la cl\^oture int\'egrale de $\fK^+$ dans $\fK_{\infty}$, $P_{\Q_p}$ le groupe galois de $\ol{\Q}_p\fK_{\infty}$ sur $\fK$, ainsi que $P^{\cycl}_{\Q_p}$ le groupe galois de $\fK_{\infty}$ sur $\fK$.

L'application qui \`a une forme modulaire associe son $q$-d\'eveloppement, nous fournit une inclusion de $\cM(\ol{\Q})$ dans $\ol{\Q}_p\fK_{\infty}^+$ et un morphisme $P_{\Q_p}\ra\Pi_{\Q}$ car $P_{\Q_p}$
pr\'eserve l'espace $\cM(\ol{\Q})$. Ceci induit un morphisme $\cG_\fK\ra \Pi_{\Q}$ et un morphisme "de localisation" $\rH^i(\Pi_{\Q}, W)\ra \rH^i(\cG_\fK,W)$ pour tout $\Pi_{\Q}$-module $W$ et tout $i\in\N$ (c.f. $\S\ref{corps}$).

Enfin, on note $\cK^+=\Q_p[[q]]$ le compl\'et\'e $q$-adique de $\fK^+$,  
$\cK_M^+$ le compl\'et\'e $q$-adique
de $\fK_M^+=\fK^+[\z_M,q_M]$ pour $M\geq 1$ un entier, $\cK^+_{\infty}$ la r\'eunion des $\cK^+_M$ pour tout $M\geq 1$, ainsi que  $\cK_{Mp^{\infty}}^+$ la r\'eunion des $\cK^+_{Mp^n}$ pour tout $n\geq 1$.
\subsection{Introduction}
\subsubsection{Fonctions $L$ $p$-adiques de formes modulaires}
Soit $N\geq 1$ et soit $\eps$ un caract\`ere de Dirichlet modulo $N$. 
Soit $f(\tau)=\sum_{n=1}^{+\infty}a_nq^{n}\in S_k(\Gamma_0(N),\eps)$ une forme primitive de poids $k\geq 2$ avec $q=e^{2i\pi \tau}$.
 Soient $\alpha,\beta$ les racines du polyn\^ome $X^2-a_pX+\eps(p)p^{k-1}$. Soit $v_p(\alpha)<k-1$, on pose $f_{\alpha}(\tau)=f(\tau)-\beta f(p\tau)$. C'est une forme de niveau $Np$, propre pour tous les $T_l$, normalis\'ee, et avec la valeur propre $\alpha$ pour $U_p$. Soit $\sum_{n=1}^{+\infty}b_nq^{n/N}$ le $q$-d\'eveloppement de $f_{\alpha}$.  Comme $v_p(\alpha)<k-1$ ( en particulier $\alpha\neq 0$ ), on peut prolonger $n\mapsto b_n$ en une fonction sur $\Z[\frac{1}{p}]$ en for\c{c}ant l'\'equation fonctionelle $b_{np}=\alpha b_n$.

Soit $\phi\in\LC_c(\Q_p, \ol{\Q})$ une fonction localement constante \`a support compact dans $\Q_p$,
\`a valeurs dans $\ol{\Q}$. On d\'efinit la fonction $L$
complexe de la forme modulaire $f$ associ\'ee \`a $\phi$ et $\alpha$,
par la formule
\[L(f_{\alpha},\phi,s)=\sum_{n\in\Z[\frac{1}{p}]}\phi(n)b_nn^{-s}.\]
La s\'erie converge pour $\Re(s)>\frac{k+1}{2}$ et la fonction
$L(f_{\alpha},\phi,s)$ admet un prolongement analytique
\`a tout le plan complexe.
De plus, il existe des nombres complexes non nuls $\Omega_f^{+}, \Omega_f^{-}$
permettant de rendre alg\'ebriques les valeurs sp\'eciales $L(f_{\alpha},\phi,j)$ de $L(f_{\alpha},\phi,s)$,
 pour $1\leq j\leq k-1$. Plus pr\'ecis\'ement,
si $\phi\in \LC_c(\Q_p,\ol{\Q})$, alors
\[\frac{\Gamma(j)}{(-2i\pi)^j}L(f,\phi,j)\in \begin{cases}\Q(f_\alpha, \z_N)\cdot\Omega_f^{+}, &\text{ si } 1\leq j\leq k-1, \phi(-x)=(-1)^j\phi(x) \\
\displaystyle \Q(f_{\alpha},\z_N)\cdot\Omega_f^{-}, &\text{ si } 1\leq j\leq
k-1, \phi(-x)=(-1)^{j+1}\phi(x).
\end{cases} \]
On pose \[\tilde{L}(f_{\alpha},\phi,j)=\frac{1}{2}\frac{\Gamma(j)}{(-2i\pi)^j}
(\frac{L(f_{\alpha},\phi+(-1)^j\phi\circ(-1),j)}{\Omega_f^{+}}
+\frac{L(f_{\alpha},\phi-(-1)^j\phi\circ(-1),j)}{\Omega_f^{-}}),\]
o\`u $\phi\circ(-1)(x)=\phi(-x)$, ce qui est \`a valeurs dans $\ol{\Q}$
 et permet de le consid\'erer comme un nombre $p$-adique.

On d\'efinit une transform\'ee de Fourier de $ \LC_c(\Q_p,\ol{\Q}_p)$ dans $\LC_c(\Q_p,\ol{\Q}_p)$ par la formule
\[\hat{\phi}(x)=\int_{\Q_p}\phi(y)e^{-2i\pi xy}\,dy=
p^{-m}\sum_{y\mod p^m}\phi(y)e^{-2i\pi xy},\]
o\`u $m\in\N$ assez grand.

Rappelons que la transform\'ee d'Amice nous donne un isomorphisme entre l'alg\`ebre des distributions sur $\Z_p$ \`a valeurs dans une extension finie $L$ de $\Q_p$ et 
l'anneau $\cR^+_L$ des fonctions localement analytiques sur la boule ouverte $v_p(T)>0$ \`a coefficients dans $L$:
\[\cA: \fD(\Z_p, L)\cong \cR^+_L; \mu\mapsto \cA_\mu(T)=\int_{\Z_p}(1+T)^x\mu.\]
Une distribution $\mu$ sur $\Z_p$ est d'ordre $k$ si sa transform\'ee d'Amice $\cA_\mu(T)=\sum_{n=0}^{+\infty} a_nT^n$ est d'ordre $k$ (i.e. 
la suite de terme g\'en\'erale $\{ v_p(a_n)+k\log_pn\}$ est minor\'ee.)
    
\begin{theo}\label{distri}Si $v_p(\alpha)<k-1$, il existe une unique distribution $\mu_{f,\alpha}$ d'ordre $v_p(\alpha)$ sur $\Z_p$, telle que l'on ait
\[\int_{\Z_p}\phi(x)x^{j-1}\mu_{f,\alpha}=\tilde{L}(f_\alpha, \hat{\phi}, j),\]
quels que soient $\phi\in\LC(\Z_p,\ol{\Q}_p)$, et $1\leq j\leq k-1$. De plus, quel que soit $\phi$ une fonction localement analytique sur $\Z_p$ \`a valeurs dans $\ol{\Q}_p$, on a $\int_{p\Z_p}\phi(p^{-1}x)\mu_{f,\alpha}=\alpha^{-1}\int_{\Z_p}\phi(x)\mu_{f,\alpha}$.
\end{theo}
Ce th\'eor\`eme signifie l'existence de l'interpolation $p$-adique des valeurs sp\'eciales de la fonction $L$ complexe de $f_\alpha$. En particulier, on peut construire 
la fonction $L$ $p$-adique de $f$ associ\'ee \`a $\alpha$ et \`a un caract\`ere localement analytique $\phi:\Z_p^*\ra \ol{\Q}_p$ ,
\[L_{p,\alpha}(f,\phi, s)=\int_{\Z_p^*}\phi(x)\langle x\rangle^{s-1}\mu_{f,\alpha}, \text{ o\`u } \langle x\rangle^{s-1}=\exp((s-1)\log \langle x\rangle) \text{ et } s\in \Z_p .\]

La d\'emonstration du th\'eor\`eme $\ref{distri}$ a \'et\'e donn\'ee par plusieurs personnes
via des m\'ethodes tr\`es diff\'erentes.
Celle de Kato $\cite{KK}$ repose sur la construction d'un
syst\`eme d'Euler (via la $K$-th\'eorie) et sur une loi de
r\'eciprocit\'e explicite r\'esultant d'un calcul d\'elicat
dans les anneaux de Fontaine,
qui permet de montrer que la machine \`a fonctions L $p$-adiques
de Perrin-Riou~\cite{Per} fournit naturellement la distribution $\mu_{f,\alpha}$
quand on l'applique au syst\`eme d'Euler de Kato.  Colmez a esquiss\'e dans~$\cite{PC1}$ une variante de la m\'ethode de Kato,
et ce texte est consacr\'e \`a v\'erifier que cette esquisse,
convenablement modifi\'ee, conduit bien au r\'esultat de Kato.

\subsubsection{Le syst\`eme d'Euler de Kato}
En bref, un syst\`eme d'Euler est une collection de classes de cohomologie v\'erifiant une relation de distribution. On construit le syst\`eme d'Euler de Kato
comme suit:

\`A partir des unit\'es de Siegel, on construit une distribution
alg\'ebrique $z_{\mathbf{Siegel}}$ sur $(\Q\otimes\hat{\Z})^2-(0,0)$ \`a
valeurs dans $\Q\otimes(\cM(\bar{\Q})[\frac{1}{\Delta}])^{*}$, o\`u $\Delta=q\prod_{n\geq 1}(1-q^n)^{24}$ est la forme modulaire de poids $12$. Cette
distribution $z_{\mathbf{Siegel}}$ est invariante sous l'action du groupe
$\Pi_{\Q}^{'}$.
La th\'eorie de Kummer $p$-adique nous fournit un \'el\'ement
\[z_{\mathbf{Siegel}}^{(p)}\in \rH^1(\Pi^{'}_{\Q},\fD_{\alg}((\Q\otimes\hat{\Z})^{2}-(0,0),\Q_p(1))).\]
Par cup-produit et
restriction \`a $\Pi_{\Q}^{(p)}\subset\Pi'_\Q$
et $\bM_2(\Q\otimes\hat{\Z})^{(p)}\subset((\Q\otimes\hat{\Z})^{2}-(0,0))^2$,
 on obtient une distribution alg\'ebrique:
\[z_{\mathbf{Kato}}\in\rH^2(\Pi^{(p)}_{\Q},\fD_{\alg}(\bM_2(\Q\otimes\hat{\Z})^{(p)},\Q_p(2))).\]
En modifiant $z_{\mathbf{Kato}}$ par un op\'erateur $(c^2-\langle c,1\rangle)(d^2-\langle1,d\rangle)$
(c.f. $\S \ref{gcd}$ ) qui fait dispara\^\i tre les d\'enominateurs,
 on obtient une distribution alg\'ebrique \`a valeurs dans $\Z_p(2)$
(que l'on peut donc voir comme une mesure), et une torsion \`a
la Soul\'e nous fournit enfin un \'el\'ement
\[z_{\mathbf{Kato},c,d}(k,j)\in\rH^2(\Pi^{(p)}_{\Q},\fD_{\alg}(\bM_2(\Q\otimes\hat{\Z})^{(p)},V_{k,j})),\]
o\`u $V_{k,j}=\Sym^{k-2}V_p\otimes\Q_p(2-j)$, o\`u $V_p$ est la repr\'esentation standard de dimension
$2$ de $\GL_2(\Z_p)$.

\subsubsection{La loi de r\'eciprocit\'e explicite de Kato}
La loi de r\'eciprocit\'e explicite de Kato consiste \`a
relier l'\'el\'ement $z_{\mathbf{Kato},c,d}(k,j)$, qui vit dans la cohomologie
du groupe $\Pi^{(p)}_{\Q}$, \`a une distribution
construite \`a partir de produits de deux s\'eries d'Eisenstein
(le produit scalaire de Petersson d'une forme primitive avec
un tel produit fait appara\^itre les valeurs sp\'eciales de la fonction $L$ de $f$,
et c'est cela qui permettrait de construire la fonction $L$ $p$-adique).
Ceci se fait en plusieurs \'etapes:

$\bullet$ On commence par ``localiser'' notre classe de cohomologie
\`a $\cG_{\fK}$ et \`a
\'etendre les coefficients de
$V_{k,j}$ \`a $\B_{\dR}^+(\bar\fK^+)\otimes V_{k,j}$, o\`u
$\B_{\dR}^+(\bar\fK^+)$ est un \'enorme anneau de Fontaine.

$\bullet$ On constate que l'image de $z_{\mathbf{Kato},c,d}(k,j)$ sous l'application "de localisation" \[\rH^2(\Pi^{(p)}_{\Q},\fD_{\alg}(\bM_2(\Q\otimes\hat{\Z})^{(p)},V_{k,j}))
\ra\rH^2(\cG_{\fK},\fD_{\alg}(\bM_2(\Q\otimes\hat{\Z})^{(p)},\B_{\dR}^+(\ol{\fK}^+)\otimes V_{k,j}))\]
est l'inflation d'un $2$-cocycle sur $P^{\cycl}_{\Q_p}$
\`a valeurs dans $\fD_{\alg}(\bM_2(\Q\otimes\hat{\Z})^{(p)},
\B_{\dR}^+(\fK_{\infty}^+)\otimes V_{k,j})$.
Les m\'ethodes de descente presque \'etale de Tate~\cite{Ta} et
Sen~\cite{Sen}, revisit\'ees par Faltings~\cite{Fa}
(c.f.~aussi Andreatta-Iovita~\cite{AI}) permettraient de montrer
que c'est toujours le cas, mais nous donnons une preuve directe
pour l'\'el\'ement de Kato
(c.f.~la construction dans $\S \ref{constrcocycle}$).

$\bullet$ On construit
une application exponentielle duale (c.f.~$\S~\ref{constructiondeloi}$):
\[\exp^{*}_{\mathbf{Kato}}:\rH^2(P_{\Q_p},\fD_{\alg}(\bM_2(\Q\otimes\hat{\Z})^{(p)},\B_{\dR}^+(\fK_{\infty}^+)\otimes V_{k,j}))\ra\rH^0(P_{\Q_p},\fD_{\alg}(\bM_2(\Q\otimes\hat{\Z})^{(p)},\cK_{\infty}^+));\]
et on calcule l'image de $z_{kato,c,d}(k,j)$.
On obtient finalement le r\'esultat fondamental suivant:
\begin{theo}\label{theo}Si $k\geq 2$, $1\leq j\leq k-1$, et $c,d\in\Z_p^{*}$,
on a:
\[\exp^{*}_{\mathbf{Kato}}(z_{\mathbf{Kato},c,d}(k,j))=z_{\mathbf{Eis},c,d}^{(p)}(k,j),\]
o\`u $z_{\mathbf{Eis},c,d}^{(p)}(k,j)\in
\rH^0(\cG_{\fK},\fD_{\alg}(\bM_2(\Q\otimes\hat{\Z})^{(p)},\fK_{\infty}^+))$
est la localis\'ee d'une distribution
$z_{\mathbf{Eis},c,d}(k,j)$ sur $\bM_2(\Q\otimes\hat{\Z})$, \`a valeurs dans
$\cM_k^{\con}(\Q_p^{\cycl})\subset \fK_\infty^+$, fix\'ee
sous l'action de $\Pi_{\Q}$.
\end{theo}

\subsubsection{La cohomologie de $\rP_m$}
Soit $M\geq 1$ tel que $v_p(M)=m\geq v_p(2p)$.
On note $\fK_{Mp^{\infty}}=\cup_{n\geq 1}\fK_{Mp^n}$.
La d\'efinition de l'application
$\exp^*_{\mathbf{Kato}}$ et le calcul de l'image de $z_{\mathbf{Kato},c,d}(k,j)$ reposent
sur une description explicite de la cohomologie
du groupe de Galois $P_{\fK_M}$ de l'extension
$\fK_{Mp^{\infty}}/\fK_M$; c'est un groupe analytique $p$-adique de rang $2$,
isomorphe \`a
\[\rP_m=\{(\begin{smallmatrix}a& b\\ c& d\end{smallmatrix})\in\GL_2(\Z_p): a=1, c=0, b\in p^m\Z_p, d\in 1+p^m\Z_p \}.\]
Soit $V$ une repr\'esentation analytique de $\rP_m$.
On d\'emontre le tr\`es utile r\'esultat suivant, o\`u l'on note
$(u,v)$ l'\'el\'ement $(\begin{smallmatrix}1& u\\ 0& e^v\end{smallmatrix})$
de $\rP_m$:
\begin{theo}[Prop.$\ref{analytic}$]
Soit $V$ une repr\'esentation analytique de $\rP_m$. Alors
\begin{itemize}
\item[$(i)$]Tout \'el\'ement de $ \rH^2(\rP_m,V) $ est repr\'esentable par un  $2$-cocycle analytique;
\item[$(ii)$]On a $ \rH^2(\rP_m,V)\cong V/(\partial_1,\partial_2-1)$,
et l'image d'un $2$-cocycle analytique
\[((u,v),(x,y))\ra c_{(u,v),(x,y)}=\sum_{i+j+k+l\geq 2}c_{i,j,k,l}u^iv^jx^ky^l,\]
par cet isomorphisme, est celle de
$\delta^{(2)}(c_{(u,v),(x,y)})=c_{1,0,0,1}-c_{0,1,1,0}$.
\end{itemize}
\end{theo}
\subsubsection{Survol de la m\'ethode de Kato et sa variante de Colmez}
Nous donnerons un survol vague de la d\'emonstration  du th\'eor\`eme $\ref{distri}$ pour expliquer le r\^ole du th\'eor\`eme $\ref{theo}$ dans la m\'ethode de Kato. 
Pour plus de d\'etails, on renvoie le lecteur  \`a $\cite{KK}$ et  $\cite{PC1}$.

$\bullet$ Si $k\geq 2 $ et $j\in\Z$, Kato $\cite{KK}$ a construit
un syst\`eme coh\'erent $c_M(k,j)\in \rH^2_{\mathbf{et}}(Y(M), W_{k,j})$, o\`u $W_{k,j}$ est le syst\`eme 
local sur la courbe modulaire $Y(M)$ de niveau $M$ dans la cohomologie du lequel on d\'ecoupe 
les repr\'esentations $p$-adiques associ\'ees aux formes modulaires de poids $k$ et de niveau $M$ d'apr\`es Deligne $\cite{De}$. 
Ensuite, il dispose d'une application exponentielle duale:
\begin{equation}
\exp^{*}_{\mathbf{Kato}}: \rH^2_{\mathbf{et}}(Y(M), W_{k,j})\ra \rH^2_{\mathbf{et}}(Y(M), \B_{\dR}^{+}(\ol{\K})\otimes W_{k,j}) \ra \K_M,
\end{equation}
avec  $\K_M=\K[\z_M,q_M]$ et $\ol{\K}$ la cl\^oture alg\'ebrique de $\K$, o\`u  $ \K$ est le compl\'et\'e $p$-adique du corps des fractions de $\Z_p[[q]][q^{-1}]$. 
Dans son expos\'e de Bourbaki $\cite{PC1}$,  Colmez a reformul\'e cette application par la m\'ethode de Tate-Sen-Colmez pour le corps $\ol{\K}$ et il a \'esquiss\'e une preuve 
du th\'eor\`eme $\ref{theo}$. Sa m\'ethode est tr\`es belle; en revanche, le $2$-cocycle qu'il a explicit\'e (c.f.~$\S~\ref{constrcocycle}$) n'appartient pas \`a son groupe de cohomologie. Cela nous oblige \`a \'elarger le corps $\K$ et \`a reconstruire l'application $\exp_{\mathbf{Kato}}^{*}$ pour l'anneau $\fK^{+}$ suivant l'id\'ee de Colmez. 

D'autre part, si $V$ est une repr\'esentation de de Rham de $\cG_{\Q_p[\z_M]}$, on a l'application exponentielle duale de Bloch-Kato:
\[\exp^{*}_{\mathbf{BK}}: \rH^1(\cG_{\Q_p[\z_M]}, \B_{\dR}^{+}(\ol{\Q_p})\otimes V   )\cong \rH^0(\cG_{\Q_p[\z_M]},  \B_{\dR}^{+}(\ol{\Q_p})\otimes V  )=: D_{\dR}^{+}(V),\]
dont l'inverse est donn\'e par le cup-produit avec $\log \chi_{cycl}\in \rH^1(\cG_{\Q_p[\z_M]}, \Q_p)$.
En plus, Kato $\cite{KK}$ a v\'erifi\'e que le diagramme suivant est commutatif:
\[\xymatrix{
&\rH^2_{\mathbf{et}}(Y(M), W_{k,j})\ar[d] &\\
&\rH^1(\cG_{\Q_p[\z_M]}, \B_{\dR}^{+}(\ol{\Q}_p)\otimes \rH^1_{\mathbf{et}}(Y(M)_{\ol{\Q}}  ,W_{k,j}  ) )\ar[dr]^{\exp^{*}_{\mathbf{Kato}}}\ar[dl]^{\exp^{*}_{\mathbf{BK}}}&\\
D_{\dR}^+( \rH^1_{\mathbf{et}}( Y(M)_{\ol{\Q}}  ,W_{k,j}  ))\ar[r]& \cM_k(\Gamma_M, \ol{\Q})\otimes\Q_p\ar[r] & \K_M 
}\]
Ceci lui permet d'identifier les deux applications exponentielles. En revanche, je ne sais pas comment comparer $\exp^{*}_{\mathbf{BK}}$ et la nouvelle $\exp^{*}_{\mathbf{Kato}}$.

On identifie $\exp^{*}_{\mathbf{BK}}$ et $\exp^{*}_{\mathbf{Kato}}$ et on projette le syst\`eme d'Euler de Kato sur la composante \`a la forme primitive $f$.  Ceci demande de calculer le produit scalaire de Pertersson de $f$ avec un produit de deux s\'eries d'Eisenstein, ce qui se fait au moyen de la m\'ethode de Rankin et le r\'esultat fait intervenir les valeurs sp\'eciales des fonction L attach\'ees \`a $f$ et ses tordues des caract\`eres de Dirichlet.

$\bullet$Soit $L$ une extension finie de $\Q_p$. On note $\cR_L$ (resp. $\cE_L^{\dag}$) l'anneau des $f(T)=\sum\limits_{n\in \Z}a_nT^n$ des fonctions analytiques (resp. analytiques born\'ees pour la valuation $v_p(f)=\inf v_p(a_n)$)  sur une cournonne du type $0<v_p(T)\leq r$, o\`u $r$ d\'epend du $f$ et $a_n$ sont des \'el\'ements de $L$. 
On les munit d'un frobenius $\varphi$ et d'une action de $\Gamma=\Gal(\Q_p(\z_{p^{\infty}})/\Q_p)$ via les formules:
\[ \varphi(\sum_{k\in \Z}a_kT^k)=\sum_{k\in \Z} a_k((1+T)^p-1)^k \text{ et } \gamma(\sum_{k\in \Z}a_kT^k)=\sum_{k\in \Z} a_k((1+T)^{\chi_{\cycl}(\gamma)}-1)^k. \]
Ces actions commutent entre elles.

Un $(\varphi,\Gamma)$-module \'etale sur $\cE_L^{\dag}$ est un $\cE_L^{\dag}$-espace vectoriel $D$ de dimension finie $d$ muni d'actions semi-lin\'eaires de $\varphi$ et $\Gamma$ commutant entre elles, verifiant que la matrice de $\varphi$ sous  une certaine base de $D$ appartient \`a $\GL_d(\cO_{\cE^{\dag}_L})$ avec $\cO_{\cE^{\dag}_L}=\{f\in\cE_L^{\dag} |  v_p(f)=\inf v_p(a_k)\geq 0\}$. Sur un $(\varphi,\Gamma)$-module \'etale sur $\cE_L^{\dag}$, il existe un op\'erateur $\psi$ communtant avec l'action de $\Gamma$, qui est l'inverse \`a gauche de $\varphi$. 

La th\'eorie d'Iwasawa et la th\'eorie de $(\varphi,\Gamma)$-modules (c.f. $\cite{CC3}$) nous fournissent un isomorphisme:
\[\Exp^{*}: \rH^1_{\Iw}(\Q_p, V)=\rH^1(\cG_{\Q_p}, \fD_0(\Z_p^{*},V))\cong (D^\dag(V))^{\psi=1},\] 
 avec $V$ une $L$-repr\'esentation de $\cG_{\Q_p}$ et $\fD_0(\Z_p^{*},V)$ l'alg\`ebre des mesures sur $\Z_p^{*}$ \`a valeurs dans $V$.
 
 Soit $V$ une $L$-repr\'esentation cristalline de $\cG_{\Q_p}$. Une analyse d\'elicate ( c.f. section $4.4$ dans $\cite{PC1}$ ) de la structure de $ (D^{\dag}(V))^{\psi=1}$ fournit une candidate pour la distribution voulu dans le th\'eor\`eme $\ref{distri}$.

$\bullet $ Il reste \`a v\'erifier la propri\'et\'e d'interpolation, mais c'est une cons\'equence non-triviale  du th\'eor\`eme suivant (c.f. $\cite{CC3}$):
\begin{theo}Si $V$ est une $L$-repr\'esentation de de Rham de $\cG_{\Q_p}$, si $n\in \N$ assez grand, et si $\mu \in \rH^1_{\Iw}(\Q_p, V)$,  alors $\varphi^{-n}(\Exp^{*}(\mu))$ converge dans $\B_{\dR}^{+}(\ol{\Q}_p)\otimes V$, et on a 
\[ p^{-n}\varphi^{-n}(\Exp^{*}(\mu))=\sum\limits_{n\in\Z}\exp^{*}_{\mathbf{BK}}(\int_{1+p^n\Z_p}x^{-k}\mu) \in L(\z_{p^n}) [[t]]\otimes D_{\dR}(V).\]
\end{theo} 
\subsection{Remerciements}

Ce travail est ma th\`ese de doctorat, sous la direction de Pierre Colmez.  Il est \'evident au lecteur combien cet artcile doit \`a Colmez comme cet article vit sur son travail manifique $\cite{PC1}$. Je voudrais exprimer toute ma gratitude \`a lui. Je remercie vivement Denis Benois pour ses remarques utiles,  Ga\"etan Chenevier et Gabriel Dospinescu pour des discussions et ses encourages.  

%% file: systeme.tex
\section{Syst\`eme d'Euler de Kato}\label{section0}
\subsection{S\'eries d'Eisenstein-Kronecker  et la distribution
$z_{\mathbf{Eis}}$}\label{section1}
\subsubsection{Formes modulaires}\label{formes}

Soient $\cH=\{x+iy,y>0\}$ le demi-plan de Poincar\'e, $A\subset \R$
un sous anneau. On note
$\GL_{2}(A)_{+}=\{\gamma\in\GL_2(A)|\det\gamma>0\}$, et on d\'efinit
une action \`a droite de $\GL_2(\R)_{+}$ sur l'ensemble des
fonctions de classe $\cC^{\infty}$ de $\cH$ dans $\C$ par la formule:
\begin{equation}\label{slash}
(f_{|_k}\gamma)(\tau)=\frac{(\det\gamma)^{k-1}}{(c\tau+d)^k}f(\frac{a\tau+b}{c\tau+d}) \text{ si }\gamma=\bigl(\begin{smallmatrix}a & b\\ c & d\end{smallmatrix}\bigr).
\end{equation}

Soit $f$ une fonction de classe $\cC^{\infty}$ de $\cH$ dans $\C$ fix\'ee par un sous groupe $\Gamma$ de $\SL_2(\Z)$ d'indice fini. Alors cette fonction $f$ est une fonction
p\'eriodique de p\'eriode $N$ pour un certain entier $N\geq 1$. Le
$q$-d\'eveloppement de $f$ s'\'ecrit sous la forme ci-dessous:
\[f(\tau)=\sum_{n\in \Z[\frac{1}{N}]}a_{n}e^{2i\pi n }=\sum_{n\in \Z[\frac{1}{N}]}a_{n}q^{n}, \text{ o\`u } q=e^{2i\pi\tau }.\]
\begin{defn}Soit $\Gamma$ un sous-groupe de $\SL_2(\Z)$ d'indice fini. Une forme modulaire de poids $k$ pour $\Gamma$ est une fonction holomorphe $f:\cH\ra\C$ verifiant les propri\'et\'es suivantes:
\begin{itemize}
\item[(1)]$f_{|_k}\gamma=f$ pour $\gamma\in\Gamma$;
\item[(2)]$f$ est holomorphe en $i\infty$ (i.e. quel que soit $\gamma\in\Gamma\setminus\SL_2(\Z)$, le $q$-d\'eveloppement de $f_{|_k}\gamma(\tau)$ est de la forme $\sum\limits_{n\in\Q_{+}}a_nq^n$, o\`u $q=e^{2i\pi \tau}$: il n'y a pas de termes n\'egatifs).
    \end{itemize}
\end{defn}

Si $K$ est un corps alg\'ebriquement clos et si $\Gamma$ est un sous-groupe distingu\'e d'indice fini de $\SL_2(\Z)$, alors le groupe des automorphismes de $\cM(\Gamma,K)$ sur $\cM(\SL_2(\Z),K)$ est $\SL_2(\Z)/\Gamma$. Ceci implique $\Pi_K=\widehat{\SL_2(\Z)}$, o\`u $\widehat{\SL_2(\Z)}$ est le compl\'et\'e profini de $\SL_2(\Z)$. Dans le cas g\'en\'eral, on dispose d'une suite exacte:
\[1\ra\Pi_{\ol{K}}\ra\Pi_K\ra\cG_K\ra1,\]
qui admet une section $\cG_K\ra\Pi_K$ naturelle,  en faisant $\cG_K$ agir sur les
coefficients du $q$-d\'eveloppement des formes modulaires.

Soient $f\in\cM_k(\Gamma,\ol{K})$ et $\alpha\in\GL_2(\Q)_{+}$.
On peut verifier que $(f*\alpha)_{|_k}(\alpha^{-1}\gamma\alpha)=f*\alpha$ pour tout $\gamma\in\Gamma$. Donc
$f*\alpha$ est invariante pour le groupe
$\alpha^{-1}\Gamma\alpha\bigcap\SL_2(\Z)$. Comme $\alpha$ peut
s'\'ecrire sous la forme $\gamma\bigl(\begin{smallmatrix}a & b\\ 0 &
d\end{smallmatrix}\bigr)$ avec  $\gamma\in\SL_2(\Z)$, on d\'eduit
\[(f*\alpha)(\tau)=d^{-k}(f*\gamma)(\frac{a\tau+b}{d})=d^{-k}(f_{|_k}\gamma)(\frac{a\tau+b}{d}),\]
ce qui montre que $f*\alpha$ est holomorphe en $i\infty$. Donc
$f*\alpha$ est une forme modulaire et $\cM(\bar{K})$ est stable sous
l'action de $\GL_2(\Q)_{+}$. On d\'efinit $\Pi_{K}^{'}$ comme le
groupe des automorphismes de $\cM(\ol{K})$ engendr\'e par $\Pi_K$ et $\GL_2(\Q)_{+}$. Plus g\'en\'eralement, si $S\subset\mathscr{P}$ est fini, on note $\Pi_K^{(S)}$ le sous-groupe de $\Pi_K^{'}$ engendr\'e par $\Pi_K$ et $\GL_2(\Z^{(S)})$, o\`u $\Z^{(S)}$ est le sous-anneau de $\Q$ obtenu en inversant tous les nombres premiers qui n'appartiennent pas \`a $S$.

Le groupe des automorphimes de $\cM^{\con}(\Q^{\cycl})$ sur $\cM(\SL_2(\Z),\Q^{\cycl})$ est le groupe profini $\SL_2(\hat{\Z})$, le compl\'et\'e profini de $\SL_2(\Z)$ par rapport aux sous-groupes de congruence. D'autre part, quel que soit $f\in\cM^{\con}(\Q^{\cycl})$, le groupe $\cG_{\Q}$ agit sur les coefficients du $q$-d\'eveloppement de $f$ \`a travers son quotient $\Gal(\Q^{\cycl}/\Q)$ qui est isomorphe \`a $\hat{\Z}^{*}$ par le caract\`ere cyclotomique. On note $H$ le groupe des automorphismes de $\cM^{\con}(\Q^{\cycl})$ sur $\cM(\SL_2(\Z),\Q)$.  La sous-alg\`ebre $\cM^{\con}(\Q^{\cycl})$ est stable par $\Pi_{\Q}$ qui agit \`a travers $H$.

\begin{theo}\label{diagram} On a un diagramme commutatif de groupes:
\[\xymatrix{
1\ar[r]&\Pi_{\bar{\Q}}\ar[r]\ar[d]&\Pi_{\Q}\ar[r]\ar[d]&\cG_{\Q}\ar[r]\ar[d]\ar@{.>}@/^/[l]&1\\
1\ar[r]&\SL_{2}(\hat{\Z})\ar[r]\ar[d]&H\ar[r]\ar@{.>}[d]^{\alpha}&\Gal(\Q^{\cycl}/\Q)\ar[r]\ar[d]^{\chi_{\cycl}}\ar@{.>}@/^/[l]^{\iota_0}&1                     \\
1\ar[r]&\SL_{2}(\hat{\Z})\ar[r]&\GL_2(\hat{\Z})\ar[r]^{\det}&\hat{\Z}^{*}\ar[r]\ar@{.>}@/^/[l]^{\iota}&1                     },\]
o\`u  $\alpha: H\ra \GL_2(\hat{\Z})$ est un isomorphisme, ce qui permet d'identifier $H$ et $\GL_2(\hat{\Z})$. Dans cette identification la section de $\cG_{\Q}$ dans $\Pi_{\Q}$ d\'ecrite plus haut envoie $u\in \hat{\Z}^{*}$ sur la matrice $(\begin{smallmatrix}1&0\\0&u\end{smallmatrix})\in \GL_2(\hat{\Z})$.
\end{theo}
La d\'emonstration repose sur l'\'etude des s\'eries d'Eisenstein qui se trouve plus loin, et donc on donne l'id\'ee ici.

$\bullet$Construire une bijection $\alpha:H\ra \GL_2(\hat{\Z})$:\\
Soit $u\in\hat{\Z}^{*}$ et soit $\sigma_u\in\Gal(\Q^{\cycl}/\Q)$ l'image inverse de $u$ dans $\Gal(\Q^{\cycl}/\Q)$ via l'isomorphisme $\chi_{\cycl}:\Gal(\Q^{\cycl}/\Q)\cong\hat{\Z}^{*}$. On peut d\'ecomposer $H$ (resp. $\GL_2(\hat{\Z})$) en l'ensemble $\SL_2(\hat{\Z})\sigma_u$ (resp. $\SL_2(\hat{\Z})(\begin{smallmatrix}1&0\\0&u\end{smallmatrix})$) par la classe \`a droite suivant $\SL_2(\hat{\Z})$. Alors, on d\'efinit une bijection de $H$ sur $\GL_2(\hat{\Z})$ en envoyant la classe \`a droite $\gamma\sigma_u$ sur $\gamma(\begin{smallmatrix}1&0\\0&u\end{smallmatrix})$, si $\gamma\in\SL_2(\hat{\Z})$.
 
$\bullet$L'\'etude des s\'eries d'Eisenstein  (c.f. Prop. $\ref{eis}$) montre que la bijection $\alpha:H\ra \GL_2(\hat{\Z})$ construite dans l'\'etape pr\'ec\'edente est un morphisme de groupes:\\
En utilisant la bijection $\alpha$, on d\'efinit une action de $\GL_2(\hat{\Z})$ sur $\cM^{\con}(\Q^{\cycl})$ induit par celle de $H$:
\begin{defn} Si $\gamma=\gamma_0(\begin{smallmatrix}1& 0\\0&u \end{smallmatrix})\in \GL_2(\hat{\Z})$ avec $\gamma_0\in\SL_2(\hat{\Z})$ et $u\in\hat{\Z}^*$, et si $f\in \cM^{\con}(\Q^{\cycl})$, on d\'efinit
\[f*\gamma:=f*(\alpha^{-1}(\gamma))=f*(\gamma_0\sigma_u). \]
\end{defn}
En utilisant le fait que l'alg\`ebre $\cM(\SL_2(\Z),\Q)$ est engendr\'ee par les s\'eries d'Eisenstein, on conclut que le corps des fractions de $\cM^{\con}(\Q^{\cycl})$ est engendr\'e par les s\'eries d'Eisenstein par la th\'eorie de Galois pour l'extension $\Frac(\cM^{\con}(\Q^{\cycl}))$ sur $\Frac(\cM(\SL_2(\Z),\Q) )$.
On d\'efinit une autre action du groupe $\GL_2(\hat{\Z})$ sur les s\'eries d'Eisenstein ( c.f. d\'ef. \ref{act} ), dont on peut le prolonger en $\cM^{\con}(\Q^{\cycl})$,
 et on montre dans la proposition \ref{eis} que ces deux actions co\"incident.  Donc $\alpha$ est un morphisme de groupes et cela nous permet d'identifier le groupe $H$ au groupe $\GL_2(\hat{\Z})$ naturellement.

\subsubsection{S\'eries d'Eisenstein-Kronecker }\label{EK}
Les r\'esultats dans ce paragraphe peuvent se trouver dans le livre de Weil $\cite{Weil}$.

\begin{defn}Si $(\tau,z)\in\cH\times\C$, on pose $q=e^{2i\pi\tau}$ et
$q_z=e^{2i\pi z}$. On introduit l'op\'erateur
$\partial_{z}:=\frac{1}{2i\pi}\frac{\partial}{\partial
z}=q_z\frac{\partial}{\partial q_z}$. On pose aussi  $e(a)=e^{2i\pi
a}$. Si $k\in\N$, $\tau\in\cH$, et $z,u\in\C$, la s\'erie d'Eisenstein-Kronecker est
\[\rH_k(s,\tau,z,u)=\frac{\Gamma(s)}{(-2i\pi)^k}(\frac{\tau-\bar{\tau}}{2i\pi})^{s-k}\sideset{}{'}\sum_{\omega\in\Z+\Z\tau}\frac{\overline{\omega+z}^k}{|\omega+z|^{2s}}e(\frac{\omega\bar{u}-u\bar{\omega}}{\tau-\bar{\tau}}),\]
qui converge pour $\Re(s)>1+\frac{k}{2}$, et poss\`ede un
prolongement
 m\'eromorphe \`a tout le plan complexe avec des p\^oles simples
en $s=1$ (si $k=0$ et $u\in\Z+\Z\tau$) et $s=0$ (si $k=0$
et $z\in\Z+\Z\tau$). Dans la formule ci-dessus $\sideset{}{'}\sum$
signifie (si $z\in\Z+\Z\tau$) que l'on supprime le terme
correspondant \`a $\omega=-z$. De plus, elle v\'erifie l'\'equation fonctionnelle:
\[\rH_k(s,\tau,z,u)=e(\frac{z\bar{u}-u\bar{z}}{\tau-\bar{\tau}})\rH_k(k+1-s,\tau,u,z).\]
\end{defn}
Si $k\geq 1$, on d\'efinit les fonctions suivantes:
\begin{align*}
E_k(\tau,z)&=\rH_k(k,\tau,z,0)(=\frac{\Gamma(k)}{(-2i\pi)^k}\sideset{}{'}\sum_{\omega\in\Z+\Z\tau}\frac{1}{(\omega+z)^k} \text{ si } k\geq 3) ,\\
F_k(\tau,z)&=\rH_k(k,\tau,0,z)(=\frac{\Gamma(k)}{(-2i\pi)^k}\sideset{}{'}\sum_{\omega\in\Z+\Z\tau}\frac{1}{(\omega)^k}e(\frac{\omega\bar{z}-z\bar{\omega}}{\tau-\bar{\tau}})\text{ si } k\geq 3 ).
\end{align*}
Les fonctions $E_k(\tau,z)$ et $F_k(\tau,z)$ sont p\'eriodiques
en $z$ de p\'eriode $\Z\tau+\Z$. De plus on a:
\[E_{k+1}(\tau,z)=\partial_z E_k(\tau,z), \text{ si } k\in\N \text{ et } E_0(\tau,z)=\log|\theta(\tau,z)| \text{ si }z\notin \Z+\Z\tau,\]
o\`u $\theta(\tau,z)$ est donn\'ee par le produit infini:
\[\theta(\tau,z)=q^{1/12}(q_z^{1/2}-q_z^{-1/2})\sideset{}{}\prod_{n\geq 1}((1-q^n q_z)(1-q^nq_z^{-1})).\]
On note $\Delta=(\partial_z
\theta(\tau,z)|_{z=0})^{12}=q\sideset{}{}\prod_{n\geq
1}(1-q^n)^{24}$ la forme modulaire de poids $12$.

Soient $(\alpha,\beta)\in (\Q/\Z)^2$ et $(a,b)\in\Q^2$ qui a pour image
$(\alpha,\beta)$ dans $(\Q/\Z)^2$. Si $k=2$ et $(\alpha,\beta)\neq(0,0)$, ou si $k\geq 1$ et $k\neq 2$,  on d\'efinit:
\[ E_{\alpha,\beta}^{(k)}=E_k(\tau,a\tau+b) \text{ et } F_{\alpha,\beta}^{(k)}=F_k(\tau,a\tau+b).\]
Si $k=2$ et $(\alpha,\beta)=(0,0)$, on d\'efinit\footnote{La s\'erie $H_2(s,\tau,0,0)$ converge pour $\Re(s)>2$, mais pas pour $s=2$.} $E_{0,0}^{(2)}=F_{0,0}^{(2)}:=\lim_{s\ra 2}H_2(s,\tau,0,0).$
\begin{lemma}\label{lemma15}Les functions $E_{\alpha,\beta}^{(k)}, F_{\alpha,\beta}^{(k)} $ satisfont les relations de distribution suivantes, quel que soit l'entier $f\geq 1$:
\begin{align}\label{rel}
\sideset{}{}\sum_{f\alpha'=\alpha,f\beta'=\beta}E_{\alpha',\beta'}^{(k)}=f^{k}E_{\alpha,\beta}^{(k)}&\text{\
\ et }
&\sideset{}{}\sum_{f\alpha'=\alpha,f\beta'=\beta}F_{\alpha',\beta'}^{(k)}=f^{2-k}F_{\alpha,\beta}^{(k)}\\
\label{rel1}
\sideset{}{}\sum_{f\beta'=\beta}E_{\alpha,\beta'}^{(k)}(\frac{\tau}{f})=f^{k}E_{\alpha,\beta}^{(k)}&\text{\
\  et }
&\sideset{}{}\sum_{f\beta'=\beta}F_{\alpha,\beta'}^{(k)}(\frac{\tau}{f})=fF_{\alpha,\beta}^{(k)}.
\end{align}
\end{lemma}
\begin{proof}[D\'emonstration]
Soit $(a,b)\in \Q^2$ un repr\'esentant de $(\alpha,\beta)\in (\Q/\Z)^2$. Les relations pour $E_{\alpha,\beta}^{(k)}$ se d\'eduisent du calcul suivant pour $k\geq 3$ (pour $k=1,2$, il faut utiliser un prolongement analytique  ):
\begin{equation*}
\begin{split}
\sum\limits_{\substack{f\alpha'=\alpha\\
f\beta'=\beta}}E_{\alpha',\beta'}^{(k)}
&=\sum\limits_{\substack{0 \leq i\leq f-1\\ 0\leq j\leq
f-1}}\frac{\Gamma(k)}{(-2i\pi)^k}\sideset{}{'}\sum_{w\in\Z+\Z\tau}\frac{1}{(w+\frac{a+i}{f}\tau+\frac{b+j}{f})^k}
=f^k E_{\alpha,\beta}^{(k)};\\
\sum\limits_{f\beta'=\beta}E^{(k)}_{\alpha,\beta'}(\frac{\tau}{f})
&=\sum_{j=0}^{f-1}\frac{\Gamma(k)}{(-2i\pi)^k}\sideset{}{'}\sum_{w\in\Z+\Z\frac{\tau}{f}}\frac{1}{(w+\frac{a\tau}{f}+\frac{b+j}{f})^k}=f^kE_{\alpha,\beta}^{(k)}.
\end{split}
\end{equation*}
La relation $\sum\limits_{
f\beta'=\beta}F^{(k)}_{\alpha,\beta'}(\frac{\tau}{f})=fF^{(k)}_{\alpha,\beta}$ se d\'eduit du m\^eme genre de calculs que les relations $\sum\limits_{\substack{f\alpha'=\alpha\\
f\beta'=\beta}}F^{(k)}_{\alpha',\beta'}=f^{2-k}F^{(k)}_{\alpha,\beta}$ et $\sum\limits_{
f\beta'=\beta}E^{(k)}_{\alpha,\beta'}(\frac{\tau}{f})=f^kE^{(k)}_{\alpha,\beta}$; et
la relation $\sum\limits_{\substack{f\alpha'=\alpha\\
f\beta'=\beta}}F^{(k)}_{\alpha',\beta'}=f^{2-k}F_{\alpha,\beta}^{(k)}$
se d\'eduit facilement en utilisant 
les deux \'egalit\'es suivantes:
\begin{itemize}
 \item[(1)]Quels que soient $w\in\Z+\Z\tau$ et $0 \leq i,j \leq f-1$ , on a
\[e(\frac{i(w\bar{\tau}-\tau\bar{w})+j(w-\bar{w})}{\tau-\bar{\tau}})=1.\]
\item[(2)] Quel que soient $w=fw'+m+n\tau$ avec $w'\in \Z+\Z\tau$ et $ m$ ou $n\neq 0$, on a
    \[
    \sum\limits_{\substack{0 \leq i\leq f-1\\ 0\leq j\leq f-1}}\frac{1}{w^k}e(\frac{w(\frac{a+i}{f}\bar{\tau}+\frac{b+j}{f})-(\frac{a+i}{f}\tau+\frac{b+j}{f})\bar{w}}{\tau-\bar{\tau}})=0.
    \]
   \end{itemize}
 \end{proof}

On dispose d'une action du groupe $\GL_2(\hat{\Z})$ sur les s\'eries d'Eisenstein.
\begin{defn}\label{act}Si $\gamma=(\begin{smallmatrix}a& b\\c&d \end{smallmatrix})\in \GL_2(\hat{\Z})$, $k\geq 1$ et $(\alpha,\beta)\in (\Q/\Z)^2$, on d\'efinit
\begin{equation*}
E^{(k)}_{\alpha,\beta}\circ\gamma=E_{a\alpha+c\beta,b\alpha+d\beta}^{(k)} \text{ et } F^{(k)}_{\alpha,\beta}\circ\gamma=F_{a\alpha+b\beta,c\alpha+d\beta}^{(k)}.
\end{equation*}
\end{defn}
Nous allons v\'erifier que ces s\'eries d'Eisenstein appartiennent \`a $\cM^{\con}(\Q^{\cycl})$ (c.f. prop. \ref{EF}) et  que l'action de $\GL_2(\hat{\Z})$ sur $\cM^{\con}(\Q^{\cycl})$ via la bijection $\alpha:H\ra \GL_2(\hat{\Z})$, d\'efinie plus haut, induit l'action pr\'ec\'edente sur les s\'eries d'Eisenstein.
\begin{prop}\label{EF}
\begin{itemize}
\item[(1)]$E_{0,0}^{(2)}=F_{0,0}^{(2)}=\frac{-1}{24}E_2^{*},$
o\`u $E_2^{*}=\frac{6}{i\pi(\tau-\bar{\tau})}+1-24\sideset{}{}\sum_{n=1}^{+\infty}\sigma_1(n)q^n$
est la s\'erie d'Eisenstein non holomorphe de poids $2$ habituelle.

\item[(2)] Si $N\alpha=N\beta=0$, alors
\begin{itemize}
\item[(a)] $\tilde{E}_{\alpha,\beta}^{(2)}=E_{\alpha,\beta}^{(2)}-E_{0,0}^{(2)}\in\cM_2(\Gamma_N,\Q(\zeta_N))$
et $E_{\alpha,\beta}^{(k)}\in\cM_k(\Gamma_N,\Q(\z_N))$ si $k\geq 1
$ et $k\neq 2$.
\item[(b)] $F_{\alpha,\beta}^{(k)}\in\cM_k(\Gamma_N,\Q(\z_N))$ si $k\geq
1, k\neq 2$ ou si $k=2,(\alpha,\beta)\neq(0,0)$.
\end{itemize}
\end{itemize}
\end{prop}
Les r\'esultats au-dessus sont bien connus. Pour faciliter la lecture, on donne l'id\'ee; les d\'etails se trouvent plus loin.
\begin{proof}[D\'emonstration]
\begin{itemize}
\item[(1)]Par d\'efinition, on a
\[E_{0,0}^{(2)}=F_{0,0}^{(2)}=\lim_{s\ra2}\frac{\Gamma(s)}{(-2i\pi)^k}(\frac{\tau-\bar{\tau}}{2i\pi})^{s-k}\sideset{}{'}\sum_{\omega\in\Z+\Z\tau}\frac{\bar{\omega}^k}{|\omega|^{2s}}. \]
On applique la formule de Poisson pour la somme (voir la d\'emonstration de la proposition (\ref{q-deve})) et prend la limite.
\item[(2)]
    On consid\`ere le $q$-d\'eveloppement des s\'eries d'Eisenstein et on va montrer dans la proposition \ref{q-deve} que les coefficients sont dans l'extension cyclotomique. Il ne reste qu'\`a v\'erifier que les s\'eries sont fix\'ees par le sous-groupe de congruence $\Gamma_N$. Mais ce fait est v\'erifi\'e par des formules plus g\'en\'erales dans la proposition \ref{eis}.
\end{itemize}
\end{proof}

\subsubsection{Les $q$-d\'eveloppements de s\'eries d'Eisenstein}\label{section4}
Soit $A$ un sous-anneau de $\C$. On note $\Dir(\C)$ le $\C$-espace vectoriel des s\'eries de Dirichlet formelles \`a coefficients dans $\C$, ainsi que $\Dir(A)$ le sous $A$-module de $\Dir(\C)$ des s\'eries de Dirichlet formelles dont les coefficients sont dans $A$. On d\'efinit une action du groupe de Galois $\cG_{\Q}=\Gal(\ol{\Q}/\Q)$ sur $\Dir(\ol{\Q})$ en agissant sur les coefficients.

Soit $\alpha\in\Q/\Z$. On d\'efinit les s\'eries de Dirichlet formelles $\z(\alpha,s)$ et
$\z^{*}(\alpha,s)$, appartenant \`a $\Dir(\Q^{\cycl})$, par les formules:
\[\z(\alpha,s)=\sum\limits_{\substack{n\in\Q_{+}^{*}\\ n=\alpha \mod\Z}}n^{-s} \text{ et } \z^{*}(\alpha,s)=\sum_{n=1}^{\infty}e^{2i\pi\alpha n}n^{-s}.\]
La fonction $\z(\alpha,s)$ admet un prolongement m\'eromorphe \`a tout le plan complexe, holomorphe en dehors de p\^oles simples en $s=1$ de r\'esidu $1$. 

Consid\'erons l'application surjective $\chi_{\cycl}:\cG_{\Q}\ra\hat{\Z}^{*}$. Soit $d\in\hat{\Z}^{*}$ et soit $\sigma_d$ un rel\`evement de $d$ dans $\cG_{\Q}$. Alors on d\'efinit l'action de $d\in\hat{\Z}^{*}$ sur les s\'eries de Dirichlet formelles  $\z(\alpha,s)$ et $\z^{*}(\alpha,s)$ via $\sigma_d$ agissant sur les coefficients \footnote{ L'action de $\sigma_d$ sur $e^{2i\pi\alpha}$ est donn\'ee par $\sigma_d*e^{2i\pi\alpha}=e^{2i\pi d\alpha}$, o\`u $d\alpha$ est bien d\'efini car on a l'isomorphisme $(\Q\otimes\hat{\Z})/\hat{\Z}\cong \Q/\Z$. }.

\begin{lemma}\label{Diri} On a
$\z(\alpha,s)*d=\z(\alpha,s)$ et $\z^{*}(\alpha,s)*d=\z^{*}(d\alpha,s)$.
\end{lemma}
La proposition suivante d\'ecrit les $q$-d\'eveloppements de s\'eries d'Eisenstein et elle montre que les coefficients du $q$-d\'eveloppement des s\'eries d'Eisenstein sont dans $\Q^{\cycl}$. En particulier, celle-ci nous permet de conclure le resultat de la proposition \ref{EF}.
\begin{prop}\label{q-deve}
\begin{itemize}
\item[(i)]Si $k\geq 1, k\neq 2$ et $\alpha, \beta\in\Q/\Z$, alors le $q$-d\'eveloppement $\sum\limits_{n\in\Q_{+}}a_nq^n$ de $E_{\alpha,\beta}^{(k)}$ est donn\'e par
\begin{equation*}
\sum_{n\in\Q_{+}^{*}}\frac{a_n}{n^s}=\z(\alpha,s)\z^{*}(\beta,s-k+1)+(-1)^{k}\z(-\alpha,s)\z^{*}(-\beta,s-k+1).
\end{equation*}
De plus, on a: soit $k\neq 1 $. $a_0=0 ($resp. $a_0=\z^{*}(\beta,1-k))$ si $\alpha\neq 0$ $($resp. $\alpha=0)$; soit $k=1$. On a $a_0=\z(\alpha,0)$ $($resp.
$a_0=\frac{1}{2}(\z^{*}(\beta,0)-\z^{*}(-\beta,0)) )$ si $\alpha\neq
0$ $($resp. $\alpha=0)$.
\item[(ii)]Si $k\geq 1$ et $\alpha, \beta\in\Q/\Z ($si $k=2$, $(\alpha,\beta)\neq (0,0))$, alors le $q$-d\'eveloppement
$\sum\limits_{n\in\Q^{+}}a_nq^n$ de $F_{\alpha,\beta}^{(k)}$ est donn\'e par
\begin{equation*}
\sum_{n\in\Q_{+}^{*}}\frac{a_n}{n^s}=\z(\alpha,s-k+1)\z^{*}(\beta,s)+(-1)^{k}\z(-\alpha,s-k+1)\z^{*}(-\beta,s).
\end{equation*}
De plus, on a: soit $k\neq 1 $, $a_0=\z(\alpha,1-k)$;\\
soit $k=1$, $a_0=\z(\alpha,0) ($resp.
$a_0=\frac{1}{2}(\z^{*}(\beta,0)-\z^{*}(-\beta,0)))$ si $\alpha\neq
0 ($resp. $\alpha=0)$.
\end{itemize}
\end{prop}

\begin{proof}[D\'emonstration]Choisisons une pr\'esentation $(\tilde{\alpha},\tilde{\beta})\in\Q^2$ de $(\alpha,\beta)\in(\Q/\Z)^2$, alors la fonction $F_{\alpha,\beta}^{(k)}$ (resp. $E_{\alpha,\beta}^{(k)}$) est d\'efinie par \'evaluer la fonction $H_k(s,\tau,0,\tilde{\alpha}\tau+\tilde{\beta})$ (resp. $H_k(s,\tau,\tilde{\alpha}\tau+\tilde{\beta},0)$) en $k=s$. La fonction $H_k(s,\tau,z,u)$  converge si $\Re(s)>1+\frac{k}{2}$, poss\`ede un prolongement m\'eromorphe \`a tout le plan complexe, holomorphe en dehors de p\^ole simple en $s=1$ (si $k=0$ et $u\in\Z+\Z\tau$) et $s=0$ (si $k=0$ et $z\in\Z+\Z\tau$).

Pour obtenir les formules dans la proposition, on applique la formule de Poisson (pour $k=1,2$, il faut utiliser un prolongement analytique.)
Montrons $(ii)$ en utilisant la formule de Poisson.

Pour simplifier la formule, on \'ecrit $z=\tilde{\alpha}\tau+\tilde{\beta}$ et on pose $\tau=x+iy\in\cH$ avec $x,y\in \R$.

   De la  d\'efinition, on a:
 \begin{equation*}
    \begin{split}
     H_k(s,\tau,0,z)&=\frac{\Gamma(s)}{(-2\pi i)^k}(\frac{-2iy}{2\pi i})^{s-k}\sum_{m,n\in\Z}\frac{1}{(m\tau+n)^k|m\tau+n|^{2(s-k)}}e(\frac{w\bar{z}-z\bar{w}}{\tau-\bar{\tau}})
    \\
    &=\frac{\Gamma(s)}{(-2\pi i)^s}\sum_{m,n\in\Z}\frac{(2iy)^{s-k}}{(m\tau+n)^s(m\bar{\tau}+n)^{s-k}}e(m\tilde{\beta}-n\tilde{\alpha})
  \end{split}
 \end{equation*}
 \begin{equation*}
 \begin{split}
 =&\frac{\Gamma(s)}{(-2\pi i)^s}(\sum_{m\geq 0}\sum_{n\in\Z}\frac{(2iy)^{s-k}}{(m\tau+n)^s(m\bar{\tau}+n)^{s-k}}e(m\tilde{\beta}-n\tilde{\alpha})+
\\  &+(-1)^{2s-k}\sum_{m> 0}\sum_{n\in\Z}\frac{(2iy)^{s-k}}{(m\tau-n)^s(m\bar{\tau}-n)^{s-k}}e(-m\tilde{\beta}-n\tilde{\alpha}))
   \end{split}
   \end{equation*}

    Si $\Re(s)>1+\frac{k}{2}$, la fonction $f(t)=\frac{e(m\tilde{\beta}-t\tilde{\alpha})}{(m\tau+t)^s(m\bar{\tau}+t)^{s-k}}$ v\'erifie la condition de la formule de Poisson. Alors,
 \begin{equation*}
    \begin{split}
    \sum_{n\in\Z}f(n)
    &=\sum_{n\in\Z}\cF(f)(n)=\sum_{n\in\Z}\int_{-\infty}^{+\infty}\frac{e^{-2i\pi nt}e(m\tilde{\beta}-t\tilde{\alpha})}{(m\tau+t)^s(m\bar{\tau}+t)^{s-k}}dt
    \\
    &=\sum\limits_{n\in\Z }e(m((n+\tilde{\alpha})x+\tilde{\beta}))\int_{-\infty}^{+\infty}\frac{e^{-2i\pi(n+\tilde{\alpha}) t}}{(imy+t)^s(-imy+t)^{s-k}}dt
    \end{split}
 \end{equation*}

Alors, si $s=k$, on a $k\geq 3$ et on a
$\int_{-\infty}^{+\infty}\frac{e^{-2i\pi(n+\tilde{\alpha}) t}}{(imy+t)^k}dt
=(-2i\pi)^k\Gamma(k)^{-1}(n+\tilde{\alpha})^{k-1}e^{-2\pi ym(n+\tilde{\alpha})};$
gr\^ace \`a la m\'ethode des r\'esidus. Ceci nous donne
\[\sum_{n\in\Z}f(n)=\sum_{n\in\Q_+, n\equiv \alpha[\Z]}(-2i\pi)^k\Gamma(k)^{-1} n^{k-1}q^{mn}e(m\tilde{\beta}).\]

On peut aussi appliquer la formule de Poisson \`a la fonction $g(t)=\frac{e(-m\tilde{\beta}-t\tilde{\alpha})}{(m\tau-t)^s(m\bar{\tau}-t)^{s-k}}$ pour $s=k\geq 3$ et obtient
\[\sum_{n\in\Z}g(n)=\sum\limits_{n\in\Q_{+},n\equiv -\alpha [\Z]}
(-2i\pi)^k\Gamma(k)^{-1}n^{k-1}q^{mn}e(-m\tilde{\beta}).\]
Donc  on a
\begin{equation*}
\begin{split}
F_{\alpha,\beta}^{(k)}&=\sum_{m\geq 0}\sum_{n\in\Q_{+},n\equiv \alpha [\Z]}n^{k-1}q^{mn}e(m\tilde{\beta})+(-1)^k\sum_{m>0}\sum_{n\in\Q_{+},n\equiv -\alpha [\Z]}n^{k-1}q^{mn}e(-m\tilde{\beta})\\
&=\left\{
 \begin{aligned}&\z(\alpha, 1-k)+\sum_{m> 0}\sum_{n\in\Q_{+},n\equiv \alpha [\Z]}n^{k-1}q^{mn}e(m\tilde{\beta})+\\ &+(-1)^k\sum_{m>0}\sum_{n\in\Q_{+},n\equiv -\alpha [\Z]}n^{k-1}q^{mn}e(-m\tilde{\beta}); \text{ si } \alpha\neq 0,
 \\
 &\frac{1}{2}(\z^{*}(\beta,0)-\z^{*}(-\beta,0))+\sum_{m> 0}\sum_{n\in\Q^{*}_{+},n\equiv \alpha [\Z]}n^{k-1}q^{mn}e(m\tilde{\beta})+\\ &+(-1)^k\sum_{m>0}\sum_{n\in\Q^{*}_{+},n\equiv -\alpha [\Z]}n^{k-1}q^{mn}e(-m\tilde{\beta}); \text{ si }\alpha=0 \text{ et } k=1;
         \end{aligned} \right.
\end{split}
\end{equation*}

Donc, \begin{equation*}
 \begin{split} \sum\limits_{n\in\Q_{+}^{*}}\frac{a_n}{n^s}&=\sum\limits_{n\in\Q_{+}^{*},n\equiv\alpha[\Z]}\sum_{m=1}^{+\infty}e^{2i\pi m\tilde{\beta}}n^{k-1}\frac{1}{(mn)^s}+(-1)^k\sum\limits_{n\in\Q_{+}^{*},n\equiv-\alpha[\Z]}\sum_{m=1}^{+\infty}e^{-2i\pi m\tilde{\beta}} n^{k-1}\frac{1}{(mn)^s}\\
 &=\sum\limits_{\substack{n\in\Q_{+}^{*},\\ n\equiv\alpha[\Z]}}\frac{1}{m^s}\sum_{m=1}^{+\infty}e^{2i\pi m\tilde{\beta}}\frac{1}{n^{s-k+1}}+(-1)^k\sum\limits_{\substack{n\in\Q_{+}^{*},\\ n\equiv-\alpha[\Z]}}\frac{1}{m^s}\sum_{n=1}^{+\infty}e^{-2i\pi m\tilde{\beta}}\frac{1}{n^{s-k+1}}\\
 &=\z(\alpha,s-k+1)\z^{*}(\beta,s)+(-1)^k\z(-\alpha,s-k+1)\z^{*}(-\beta,s).
 \end{split}
 \end{equation*}

\end{proof}
\begin{remark}
\begin{itemize}
\item[(i)]D'apr\`es cette proposition, on voit que $E_{\alpha,\beta}^{(1)}$ et
$F_{\alpha,\beta}^{(1)}$ ont le m\^eme $q$-d\'eveloppement pour tous
les $\alpha, \beta\in\Q/\Z$ (On peut aussi d\'eduire ce r\'esultat
par l'\'equation fonctionnelle de $H_k(s,\tau,z,u)$).
\item[(ii)]
Soit $d\in\hat{\Z}^{*}$ et soit $\sigma_d$ un rel\`evement de $d$ dans $\cG_{\Q}$. L'action de $\hat{\Z}^{*}$ induit par l'action de $H$ sur les s\'eries d'Eisenstein est donn\'ee par $\sigma_d$ agissant sur les coefficients du $q$-d\'eveloppement de formes modulaires. Il \'equivaut \`a trouver l'action de $\hat{\Z}^{*}$ sur les s\'eries de Dirichlet formelles associ\'ees.   Soit $\sum_{n\in\Q_{+}}a_nq^n$ le $q$-d\'eveloppement de $E_{\alpha,\beta}^{(k)}$ (resp.  $F_{\alpha,\beta}^{(k)}$).  Alors le $q$-d\'eveloppement $\sum_{n\in\Q_{+}}b_nq^n$ de $E_{\alpha,\beta}^{(k)}*d$  est donn\'e par
\begin{equation*}
\begin{split} &\sum_{n\in\Q_{+}^{*}}\frac{b_n}{n^s}=\sum_{n\in\Q_{+}^{*}}\frac{a_n}{n^s}*d=\z(\alpha,s)\z^{*}(d\beta,s-k+1)+(-1)^k\z(-\alpha,s)\z^{*}(-d\beta,s-k+1),\\
&b_0=\left\{
 \begin{aligned} & 0 ( \text{resp. } \z^*(d\beta,1-k) ); \text{ si } k\neq 1, \alpha\neq 0 ( \text{resp. si } \alpha=0 ) ,
 \\
 &\z(\alpha,0) (\text{resp.} \frac{1}{2}(\z^{*}(d\beta,0)-\z^{*}(-d\beta,0))); \text{ si }k=1, \alpha\neq 0 (\text{resp.  } \alpha=0 ); \end{aligned} \right.
\end{split}
\end{equation*}
et celui de $F_{\alpha,\beta}^{(k)}*d$ est donn\'e par
\begin{equation*}
\begin{split}
&\sum_{n\in\Q_{+}^{*}}\frac{b_n}{n^s}=\z(\alpha,s-k+1)\z^{*}(d\beta,s)+(-1)^k\z(-\alpha,s-k+1)\z^{*}(-d\beta,s);
\\
&b_0=\left\{
 \begin{aligned} &  \z(\alpha,1-k) ); \text{ si } k\neq 1,
 \\
 &\z(\alpha,0) \left( \text{ resp.} \frac{1}{2}(\z^{*}(d\beta,0)-\z^{*}(-d\beta,0)) \right); \text{ si }k=1, \alpha\neq 0 (\text{ resp.  } \alpha=0 ); \end{aligned} \right..
\end{split}
\end{equation*}

Donc on a $E_{\alpha,\beta}^{(k)}*d=E_{\alpha,d\beta}^{(k)}$ et $F_{\alpha,\beta}^{(k)}*d=F_{\alpha,d\beta}^{(k)}$.
\end{itemize}
\end{remark}

\begin{prop}\label{eis} Si $\gamma=\bigl(\begin{smallmatrix}a & b\\ c &
d\end{smallmatrix}\bigr)\in\GL_2(\hat{\Z}), k\geq 1$ et
$(\alpha,\beta)\in(\Q/\Z)^2$, on a:
\[E_{\alpha,\beta}^{(k)}*\gamma=E_{a\alpha+c\beta,b\alpha+d\beta}^{(k)}\text{ et } F_{\alpha,\beta}^{(k)}*\gamma=F_{a\alpha+c\beta,b\alpha+d\beta}^{(k)}.\]
\end{prop}
\begin{proof}[D\'emonstration]

Comme $\GL_2(\hat{\Z})=\cup_{d\in\hat{\Z}^*}\SL_2(\hat{\Z})(\begin{smallmatrix}1&0\\0&d\end{smallmatrix})$, il suffit de v\'erifier pour  $\gamma\in\SL_2(\hat{\Z})$ et $(\begin{smallmatrix}1&0\\ 0&d\end{smallmatrix})$ avec $d\in\hat{\Z}^{*}$. Le cas de $(\begin{smallmatrix}1&0\\ 0&d\end{smallmatrix})$ suit du lemme \ref{Diri} et du $q$-d\'eveloppement de la proposition $\ref{q-deve}$. Le cas de $\SL_2(\hat{\Z})$ suit du calcul pour $\SL_2(\Z)$ et se d\'eduit par continuit\'e.

Le calcul pour $\SL_2(\Z)$ est facile et on seulement donne le calcul  pour la relation $E_{\alpha,\beta}^{(k)}*\gamma=E_{a\alpha+c\beta,b\alpha+d\beta}^{(k)}$. 
Choisissons une pr\'esentation $(\tilde{\alpha},\tilde{\beta})\in\Q^2$ de $(\alpha,\beta)\in(\Q/\Z)^2$, et si $\gamma=\bigl(\begin{smallmatrix}a & b\\ c &
d\end{smallmatrix}\bigr)\in\SL_2(\Z)$, alors on a:

\begin{equation*}
\begin{split}
E_{\alpha,\beta}^{(k)}*\gamma&=E_{\alpha,\beta}^{(k)}|_{k}\gamma(\tau)=\frac{1}{(c\tau+d)^{k}}E_k(\gamma\tau,\tilde{\alpha}\gamma\tau+\tilde{\beta})=\frac{1}{(c\tau+d)^k}\frac{\Gamma(k)}{(-2i\pi)^k}\sum\limits_{\omega\in\Z+\Z\gamma\tau}\frac{1}{(\omega+\tilde{\alpha}\gamma\tau+\tilde{\beta})^{k}}\\
&=\frac{\Gamma(k)}{(-2i\pi)^k}\sideset{}{}\sum_{\omega\in\Z+\Z\tau}\frac{1}{(\omega+\tilde{\alpha}(a\tau+b)+\tilde\beta(c\tau+d))^k}=E_{a\alpha+c\beta,b\alpha+d\beta}^{(k)}(\tau)
\end{split}
\end{equation*}

\end{proof}

\subsubsection{Les distributions $z_{\mathbf{Eis}}(k),z'_{\mathbf{Eis}}(k),  \text{ et }
z_{\mathbf{Eis}}(k,j)$}\label{distributionEis}
 Soient $X=(\Q\otimes\hat{\Z})^2, G=\GL_2(\Q\otimes\hat{\Z})$ et $V=\cM_k^{\con}(\Q^{\cycl})$.
Alors $X$ est un espace topologique localement profini et $G$ est un groupe localement profini, agissant contin\^ument \`a droite sur $X$ par la multiplication de matrices.

L'action de $G$ sur $V$ \`a droite, not\'e par $*$,  se d\'eduit de l'action de
$\Pi_{\Q}^{'}$ sur $V$ et l'action de $\Pi_{\Q}$
se factorise \`a travers $\GL_2(\hat{\Z})$. Comme tout
$\gamma\in\GL_2(\Q\otimes\hat{\Z})$ peut s'\'ecrire sous la
forme $\gamma=g_1\bigl(\begin{smallmatrix}r_0 & 0\\ 0 &
r_0\end{smallmatrix}\bigr)\bigl(\begin{smallmatrix}1 & 0\\ 0 &
e\end{smallmatrix}\bigr)g_2$ avec
$g_1,g_2\in\GL_2(\hat{\Z}),r_0\in\Q^{*}_{+}$ et $e$ un entier $\geq 1$, il suffit de donner les formules pour $\gamma\in
\GL_2(\hat{\Z}),\gamma=\bigl(\begin{smallmatrix}r_0&0\\0&r_0\end{smallmatrix}\bigr)$
et $\gamma=\bigl(\begin{smallmatrix}1&0\\0&e\end{smallmatrix}\bigr)$
respectivement.
Comme
$\bigl(\begin{smallmatrix}r_0 & 0\\ 0 & r_0\end{smallmatrix}\bigr)$
et $\bigl(\begin{smallmatrix}1 & 0\\ 0 & e\end{smallmatrix}\bigr)$
apparaissent dans $\GL_2(\Q)_{+}$, on prend ses actions par la formule (\ref{etoi})  dans ``notation''. Si $\gamma\in\GL_2(\hat{\Z})$, en utilisant la d\'ecomposition $\GL_2(\hat{\Z})=\cup_{d\in\hat{\Z}^*}\SL_2(\hat{\Z})(\begin{smallmatrix}1&0\\0&d\end{smallmatrix})$, on d\'ecompose l'action de $\gamma$ en deux parties.
Comme on est en poids $k$, l'action de $\SL_2(\hat{\Z})$ est l'action $|_k$. L'action de $(\begin{smallmatrix}1&0\\0&d\end{smallmatrix})$ est via un rel\`evement $\sigma_d$ dans $\cG_{\Q}$ agissant sur les coefficients du $q$-d\'eveloppement. En particulier, au cas des s\'eries d'Eisenstein, la proposition \ref{eis} nous donne les formules explicites. Les formules (\ref{actiondis}) dans "notation" s'applique au cas au-dessus.

\begin{theo}\label{eiskj}Si $k\geq 1$, il existe une distribution alg\'ebrique $z_{\mathbf{Eis}}(k)$ $($resp. $z'_{\mathbf{Eis}}(k)$$)$ $\in
\fD_{\alg}((\Q\otimes\hat{\Z})^2,\cM_k^{\con}(\Q^{\cycl}))$ v\'erifiant:
quels que soient $r\in\Q^{*}$ et $(a,b)\in\Q^2$, on a
\begin{align*}
\int_{(a+r\hat{\Z})\times(b+r\hat{\Z})}z_{\mathbf{Eis}}(k)&=r^{-k}E^{(k)}_{r^{-1}a,r^{-1}b} \bigl(resp. \int_{(a+r\hat{\Z})\times(b+r\hat{\Z})}z_{\mathbf{Eis}}(k)=r^{-k}\tilde{E}^{(k)}_{r^{-1}a,r^{-1}b} \text{ si } k=2 \bigr)\\
\int_{(a+r\hat{\Z})\times(b+r\hat{\Z})}z'_{\mathbf{Eis}}(k)&=r^{k-2}F^{(k)}_{r^{-1}a,r^{-1}b}.
\end{align*}
De plus, si $\gamma\in\GL_2(\Q\otimes\hat{\Z})$, alors
$z_{\mathbf{Eis}}(k)*\gamma=z_{\mathbf{Eis}}(k) \text{ et } z'_{\mathbf{Eis}}(k)*\gamma=|\det\gamma|^{1-k}z'_{\mathbf{Eis}}(k).$
\end{theo}
\begin{proof}[D\'emonstration]L'existence de la distribution r\'esulte des relations de distribution de $E^{(k)}_{\alpha,\beta}$ et $F_{\alpha,\beta}^{(k)}$ dans le lemme (\ref{lemma15}).
Comme tout
$\gamma\in\GL_2(\Q\otimes\hat{\Z})$ peut s'\'ecrire sous la
forme $\gamma=g_1\bigl(\begin{smallmatrix}r_0 & 0\\ 0 &
r_0\end{smallmatrix}\bigr)\bigl(\begin{smallmatrix}1 & 0\\ 0 &
e\end{smallmatrix}\bigr)g_2$ avec
$g_1,g_2\in\GL_2(\hat{\Z}),r_0\in\Q^{*}_{+}$ et $e$ un entier $\geq 1$.
Alors,  il ne reste qu'\`a calculer les actions de $\GL_2(\hat{\Z})$, $\bigl(\begin{smallmatrix}r_0 & 0\\ 0 &
r_0\end{smallmatrix}\bigr) $ et $\bigl(\begin{smallmatrix}1 & 0\\ 0 &
e\end{smallmatrix}\bigr)$ respectivement.  
On donnera le calcul pour la distribution alg\'ebrique $z_{\mathbf{Eis}}(k)$ et la relation pour $z_{\mathbf{Eis}}^{'}(k)$ est du calcul de la m\^eme mani\`ere. 
\begin{itemize}
\item
Si $\gamma\in\GL_2(\hat{\Z})$, on a $|\det\gamma|=1$.

\begin{equation*}
\begin{split}
\int_{(a+r\hat{\Z})\times(b+r\hat{\Z})}z_{Eis}(k)*\gamma=\left(\int_{((a+r\hat{\Z})\times(b+r\hat{\Z}))*\gamma^{-1}}z_{Eis}(k)\right)*\gamma=r^{-k}E^{(k)}_{r^{-1}a,r^{-1}b}.
\end{split}
\end{equation*}

La derni\`ere \'equation dans la formule au-dessus se d\'eduit de la proposition \ref{eis}.

\item Si $\gamma=\bigl(\begin{smallmatrix}r_0&0\\0&r_0\end{smallmatrix}\bigr)$, on a:
\begin{equation*}
\begin{split}
&\int_{(a+r\hat{\Z})\times(b+r\hat{\Z})}z_{Eis}(k)*\gamma =\left(\int_{(\frac{a}{r_0}+\frac{r}{r_0}\hat{\Z})\times(\frac{b}{r_0}+\frac{r}{r_0}\hat{\Z})}z_{Eis}(k)\right)*\gamma=(\frac{r}{r_0})^{-k}(E^{(k)}_{r^{-1}a,r^{-1}b})*\gamma\\
&=(\frac{r}{r_0})^{-k}\frac{1}{(0\tau+r_0)^k}E_{r^{-1}a,r^{-1}b}^{(k)}(\gamma\tau)=r^{-k}E^{(k)}_{r^{-1}a,r^{-1}b};
\end{split}
\end{equation*}

\item Si $\gamma=\bigl(\begin{smallmatrix}1&0\\0&e\end{smallmatrix}\bigr)$, on a:
\begin{equation*}
\begin{split}
&\int_{(a+r\hat{\Z})\times(b+r\hat{\Z})}z_{Eis}(k)*\gamma=(\int_{((a+r\hat{\Z})\times(b+r\hat{\Z}))*\gamma^{-1}}z_{Eis}(k))*\gamma=(\int_{(a+r\hat{\Z})\times(\frac{b}{e}+\frac{r}{e}\hat{\Z})}z_{Eis}(k))*\gamma\\
&=\sum_{i=0}^{e-1}(\int_{(a+r\hat{\Z})\times(\frac{b}{e}+\frac{ir}{e}+r\hat{\Z})}z_{Eis}(k))*\gamma=\frac{1}{(0\tau+e)^k}r^{-k}\sum_{i=0}^{e-1}E^{(k)}_{r^{-1}a, \frac{r^{-1}b+i}{e}}(\frac{\tau}{e})=r^{-k}E^{(k)}_{r^{-1}a,r^{-1}b},
\end{split}
\end{equation*}
o\`u la derni\`ere \'equation se d\'eduit de la relation (\ref{rel1}) dans le lemme
\ref{lemma15}.

\end{itemize}

\end{proof}

On peut identifier
$(\Q\otimes\hat{\Z})^2\times(\Q\otimes\hat{\Z})^2$ avec
$\bM_2(\Q\otimes\hat{\Z})$ via le morphisme
$((a,b),(c,d))\mapsto\bigl(\begin{smallmatrix}a & b\\ c &
d\end{smallmatrix}\bigr)$. En utilisant le fait que le produit de
deux formes modulaires de poids $i$ et $j$ est une forme modulaire
de poids $i+j$, on obtient une application naturelle:
\[\fD_{\alg}((\Q\otimes\hat{\Z})^2,\cM_i(\overline{\Q}))\otimes\fD_{\alg}((\Q\otimes\hat{\Z})^2,\cM_j(\overline{\Q}))\ra\fD_{\alg}(\bM_2(\Q\otimes\hat{\Z}),\cM_{i+j}(\overline{\Q})).\]
Si $k\geq 2$ et $1\leq j\leq k-1$, on d\'efinit
\[z_{\mathbf{Eis}}(k,j)=\frac{(-1)^j}{(j-1)!}z'_{\mathbf{Eis}}(k-j)\otimes z_{\mathbf{Eis}}(j)\in\fD_{\alg}(\bM_2(\Q\otimes\hat{\Z}),\cM_k(\overline{\Q})).\]

\subsubsection{Une variante des s\'eries d'Eisenstein et la distribution $z_{\mathbf{Eis},c,d}(k,j)$}

Soit $\langle\cdot\rangle:\Z_p^{*}\ra\hat{\Z}^*$ l'inclusion naturelle  en envoyant $x$ sur $\langle x\rangle=(1,\cdots, x,1,\cdots)$, o\`u $x$ est \`a la place $p$. Consid\'erons l'inclusion de $\hat{\Z}^{*}$ dans $\GL_2(\hat{\Z})$ en envoyant $d$ sur $(\begin{smallmatrix}d&0\\ 0& d\end{smallmatrix})$. D'apr\`es la proposition \ref{eis}, cela d\'efinit une action de $d\in\hat{\Z}^{*}$ sur les s\'eries d'Eisenstein par les formules: si $k\geq 1$ et $(\alpha,\beta)\in (\Q/\Z)^2$, on a
\begin{equation}
d\cdot E_{\alpha,\beta}^{k}=E_{d\alpha, d\beta}^{(k)}=E_{\alpha,\beta}^{(k)}*(\begin{smallmatrix}d&0\\ 0& d\end{smallmatrix}) \text{ et } d\cdot F_{\alpha,\beta}^{(k)}=F_{d\alpha,d\beta}^{(k)}=F_{\alpha,\beta}^{(k)}*(\begin{smallmatrix}d&0\\ 0& d\end{smallmatrix}),
\end{equation}
o\`u l'action de $*$ est celle de $\GL_2(\hat{\Z})$ sur les s\'eries d'Eisenstein.

Consid\'erons l'injection de $\cM_k^{\con}(\ol{\Q})$ dans $\cM_k^{\con}(\ol{\Q}_p)$.
 On peut d\'efinir une variante des s\'eries d'Eisenstein \`a coefficients dans $\ol{\Q}_p$ ci-dessous:  si $c\in\Z_p^{*}$, on pose
 \begin{equation}\label{variant}
 \begin{split}
 E^{(k)}_{c,\alpha,\beta}&=\left\{
 \begin{aligned}&c^2E^{(k)}_{\alpha,\beta}-c^{k}E_{\langle c\rangle\alpha,\langle c\rangle\beta}^{(k)}; \text{ si } k\geq 1 \text{ et } k\neq 2,
 \\
 &c^2\tilde{E}_{\alpha,\beta}^{(2)}-c^2\tilde{E}_{\langle c\rangle\alpha,\langle c\rangle\beta}^{(2)}; \text{ si }k=2;
         \end{aligned} \right. \\
 F^{(k)}_{c,\alpha,\beta}&=c^2F_{\alpha,\beta}^{(k)}-c^{2-k}F_{\langle c\rangle\alpha,\langle c\rangle\beta}^{(k)} \text{ si } k\geq 1 \text{ et } k\neq 2, \text{ ou si } (\alpha,\beta)\neq (0,0)\text{ et } k=2.
 \end{split}
 \end{equation}
Elles sont des combinaisons lin\'eaires des s\'eries d'Eisenstein. On d\'efinit une d\'erivation $\partial_z$ sur $E^{(k)}_{c,\alpha,\beta}$ comme suivant: choisissons une suite $\{c_n: c_n\in\N\}_{n\in\N}$ telle que $\lim_{n\ra +\infty}c_n=c$ et posons 
\[\partial_zE^{(k)}_{c,\alpha,\beta}=\lim_{n\ra+\infty}\partial_z(c_n^2E_k(\tau,z)-c_n^{k}E_k(\tau,c_nz))|_{z=\alpha\tau+\beta}.\]
De la relation $E_{k+1}(\tau,z)=\partial_z E_k(\tau,z)$ si $k\in\N$, on d\'eduit que:
\begin{lemma}\label{partialle}
On a $\partial_zE^{(k)}_{c,\alpha,\beta}=E^{(k+1)}_{c,\alpha,\beta}$.
\end{lemma}

D'apr\`es le paragraphe $\S \ref{distributionEis}$.
la proposition suivante est imm\'ediate:
\begin{prop}
\begin{itemize}
\item[(1)]Soit $c\in\Z_p^{*}$. Si $k\geq 1$, il existe une distribution alg\'ebrique $z_{\mathbf{Eis}, c}(k)$ (resp. $z^{'}_{\mathbf{Eis},c}$) $\in \fD_{\alg}((\Q\otimes\hat{\Z})^2, \cM_k^{\con}(\Q_p^{\cycl}))$ v\'erifiant: quel que soient $r\in\Q^*$ et $(a,b)\in\Q^2$, on a
\begin{equation*}
\begin{split}
\int_{(a+r\hat{\Z})\times(b+r\hat{\Z})}z_{\mathbf{Eis},c}(k)=r^{-k}E^{(k)}_{c,r^{-1}a,r^{-1}b}
(\text{resp. } \int_{(a+r\hat{\Z})\times(b+r\hat{\Z})}z^{'}_{\mathbf{Eis},c}(k)=r^{k-2}F^{(k)}_{c,r^{-1}a,r^{-1}b}. )
\end{split}
\end{equation*}
De plus, si $\gamma\in\GL_2(\Q\otimes \hat{\Z})$, alors
on a \[z_{\mathbf{Eis},c}(k)*\gamma=z_{\mathbf{Eis},c}(k) \text{ et } z_{\mathbf{Eis},c}^{'}*\gamma=|\det\gamma|^{1-k}z_{\mathbf{Eis},c}^{'}(k).\]
\item[(2)]Soient $c,d\in\Z_p^{*}$. Si $k\geq 2$ et $1\leq j\leq k-1$, la distribution
    \[z_{\mathbf{Eis},c,d}(k,j)=\frac{(-1)^j}{(j-1)!}z'_{\mathbf{Eis},c}(k-j)\otimes z_{\mathbf{Eis},d}(j)\] appartient \`a $\fD_{\alg}(\bM_2(\Q\otimes\hat{\Z}),\cM_k(\Q^{\cycl}_p)).$ De plus, elle v\'erifie les propri\'et\'es suivantes:\\
$\bullet$Si $M,N$ sont deux entiers $\geq 1$, on pose
$O_{M,N}=\{\bigl(\begin{smallmatrix}a_0 & b_0\\ c_0 &
d_0\end{smallmatrix}\bigr), a_0-1,b_0\in M\hat{\Z}, c_0,d_0-1\in N\hat{\Z}\}$
et  $\phi_{M,N}$ la fonction caract\'eristique de $O_{M,N}$. Alors on
a:
\[\int\phi_{M,N}z_{\mathbf{Eis},c,d}(k,j)=\frac{(-1)^j}{(j-1)!}M^{k-j-2}N^{-j}F_{c,\frac{1}{M},0}^{(k-j)}E^{(j)}_{d,0,\frac{1}{N}}.  \]
$\bullet$Si $\gamma\in\GL_2(\Q\otimes\hat{\Z})$, $z_{\mathbf{Eis},c,d}(k,j)_{|_k}\gamma=|\det\gamma|^{j-1}z_{\mathbf{Eis},c,d}(k,j).$
\end{itemize}
\end{prop}

\subsection{Unit\'es de Siegel et distribution  $z_{\mathbf{Siegel}}$}\label{section2}
\begin{defn}Soit $\Gamma$ un sous-groupe de $\SL_2(\Z)$ d'indice fini. 
\begin{itemize}
\item[(1)]Une fonction modulaire de poids $0$ pour $\Gamma$ est une fonction holomorphe $f(\tau)$ sur $\cH$ et m\'eromorphe sur $(\cH\cup\bP^1(\Q))/\Gamma $, telle que, $f(\frac{a\tau+b}{c\tau+d})=f(z)$  pour  $\gamma\in\Gamma$;
\item[(2)]
Une unit\'e modulaire pour $\Gamma$ est une unit\'e de l'anneau des fonctions modulaires de poids $0$ pour $\Gamma$ .
\end{itemize}
\end{defn}

Soit $K$ un sous corps de $\C$. On note $\cU(\Gamma, K)$ le groupe des unit\'es modulaires pour $\Gamma$ dont le $q$-d\'eveloppement est \`a coefficients dans $K$.
 On note
$\cU(K)$ la r\'eunion des $\cU(\Gamma,K)$, o\`u $\Gamma$ d\'ecrit tous les
sous-groupes d'indice fini de $\SL_2(\Z)$.

La fonction th\^eta $\theta(\tau,z)$ est d\'efinie par le produit infini
\[\theta(\tau,z)=q^{1/12}(q_z^{1/2}-q_z^{-1/2})\sideset{}{}\prod_{n\geq 1}((1-q^n q_z)(1-q^nq_z^{-1})).\]
Elle n'est pas p\'eriodique en $z$ de p\'eriode $\Z+\Z\tau$. Rappelons ci-apr\`es les propri\'et\'es fondamentales de la fonction
$\theta$ ( c.f. \cite{SL} chapitre $19$):
\begin{itemize}
\item $\theta$ est homog\`ene de degr\'e $0$;
\item Soit $\gamma=\bigl(\begin{smallmatrix}a_0&b_0\\c_0&d_0\end{smallmatrix}\bigr)\in\SL_2(\Z)$.
Alors $\theta(\tau,z)$ satisfait une \'equation fonctionnelle:
\[\theta(\gamma(\begin{smallmatrix}\tau\\ 1\end{smallmatrix}),z)=\z_{12,\gamma}\theta(\tau,z)e(\frac{\pi ic_0z^2}{(c_0\tau+d_0)})\]  o\`u  $\z_{12,\gamma}$ est une racine d'unit\'e d'ordre  $12$  qui d\'epend de $\gamma$ et $\gamma(\begin{smallmatrix}\tau\\ 1\end{smallmatrix})$ est le produit de matrices usuel (i.e. $\gamma(\begin{smallmatrix}\tau \\ 1 \end{smallmatrix})=(\begin{smallmatrix}a_0\tau+b_0 \\ c_0\tau+d_0\end{smallmatrix})$).
\end{itemize}

 Notons $\Lambda_\tau$ le r\'eseau $\Z+\Z\tau$.  Pour un
entier $c>2, (c,6)=1$, on d\'efinit une autre fonction
$g_{0,c}(z)=\theta(\tau,z)^{c^2}\theta(\tau,cz)^{-1}$ \`a partir de la fonction th\^eta.
Soit $a$ un entier $\geq 1$; on note $\cO(\C/\Lambda_\tau,a)$ l'anneau des fonctions holomorphes sur $\C/\Lambda_\tau-(a^{-1}\Lambda_\tau)/\Lambda_\tau$ et on note $\cO(\C/\Lambda_\tau,a)^*$ le groupe des unit\'es de $\cO(\C/\Lambda_\tau,a)$.
Soit $a,b$ deux entiers tels que $(a,b)=1$; on d\'efinit une application de norme du groupe $\cO(\C/\Lambda_\tau,ab)^{*}$ dans le groupe $\cO(\C/\Lambda_\tau,b)^{*}$ comme suit: soit $f(z)\in\cO(\C/\Lambda_\tau,ab)^{*}$, on a
\[N_a(f(z)):=\prod\limits_{k=0}^{a-1}\prod\limits_{j=0}^{a-1}f(\frac{z}{a}+\frac{k}{a}+\frac{j\tau}{a}).\]
\begin{lemma}
\begin{itemize}
\item[(1)]La fonction $g_{0,c}(z)$, qui est une fonction elliptique sur  $\C/\Lambda_\tau$,  est une unit\'e de l'anneau $\cO(\C/\Lambda_\tau,c)$ des fonctions holomorphes sur $\C/\Lambda_\tau-(c^{-1}\Lambda_\tau)/\Lambda_\tau$.
\item[(2)]Quel que soit $a$ un entier $\geq 1$ tel que $(a,c)=1$, on a $N_a(g_{0,c})=g_{0,c}$.
\end{itemize}
\end{lemma}
\begin{proof}[D\'emonstration]
\begin{itemize}
\item[(1)]La premi\`ere assertion de $(1)$ du lemme suit d'un calcul direct.

Observons que le diviseur
associ\'e \`a $g_{0,c}$ sur $\C/\Lambda_\tau$ est
$c^2(\bar{0})-c^{-1}\Lambda_\tau/\Lambda_\tau$ ( cela suit de l'expression explicite de
$g_{0,c}$.)
Il s'ensuit que $g_{0,c}$ est une unit\'e de l'anneau des fonctions holomorphes sur $\C/\Lambda_\tau-(c^{-1}\Z+c^{-1}\Z\tau)/\Lambda_\tau$.

\item[(2)]De la d\'efinition de l'application $N_a$, on a $N_a(N_b(g_{0,c}))=N_b(N_a(g_{0,c}))$ pour $a,b\in\N $ tels que $(ab,c)=1$. Comme $N_a(g_{0,c})$ et $g_{0,c}(z)$ ont le m\^eme diviseur, il existe une constante $u_a\in\C^{*}$ telle que $N_a(g_{0,c})=u_ag_{0,c}$.  Par la relation $N_aN_b=N_bN_a$, on a $u_b^{a^2-1}=u_a^{b^2-1}$. Si on pose $g=u_2^{-3}u_3g_{0,c}$, alors on a
    \[N_a(g)=u_2^{-3a^2}u_3^{a^2}u_ag_{0,c}=u_2^{-3(a^2-1)}u_3^{a^2-1}u_ag=u_a^{-3(2^2-1)}u_a^{3^2-1}u_ag=g.\]
    L'assertion se d\'eduit des relations $N_2(g_{0,c})=g_{0,c}$ et $N_3(g_{0,c})=g_{0,c}$. En effet, quel que soit $a\geq 1$ tel que $(a,c)=1$, de la formule $\prod_{i=0}^{a-1}(1-X\z_a^i)=1-X^a$, on a
   \begin{equation*}
    N_a(1-q^nq_z^{\pm1})=\prod_{j=0}^{a-1}(1-q^{an\pm j}q_z^{\pm1});
    N_a(1-q^nq_z^{\pm c})=\prod_{j=0}^{a-1}(1-q^{an\pm cj}q_{cz}^{\pm1}).
    \end{equation*}

    Alors,   on a la relation
    \[N_a(\prod_{n\geq 0}(1-q^nq_z)\prod_{n\geq 1}(1-q^nq_z^{-1}))=\prod_{n\geq 0}(1-q^nq_z)\prod_{n\geq 1}(1-q^nq_z^{-1});\]
cela simplifie les calculs et on v\'erifie facilement pour $a=2,3$, $ N_a(g_{0,c})=g_{0,c}.$
\end{itemize}
\end{proof}

Soit $d$ un autre entier v\'erifiant $(d,6)=1$. On d\'efinit $c^{*}g_{0,d}(z)=g_{0,d}(cz)$.
\begin{lemma}Soient $c,d$ deux entiers tels que $(cd,6)=1$. Alors on a
 \begin{equation}\label{siegel}
(g_{0,d})^{c^2}(c^{*}g_{0,d})^{-1}=(g_{0,c})^{d^2}(d^{*}g_{0,c})^{-1}.
\end{equation}
\end{lemma}
\begin{proof}[D\'emonstration]Consid\'erons les fonctions m\'eromorphes
$(g_{0,d})^{c^2}(c^{*}g_{0,d})^{-1}$ et
$(g_{0,c})^{d^2}(d^{*}g_{0,c})^{-1}$, elles ont le m\^eme diviseur
sur $\C/\Lambda_\tau:
c^2d^2(\bar{0})-c^2(d^{-1}\Lambda_\tau/\Lambda_\tau)-d^2(c^{-1}\Lambda_\tau/\Lambda_\tau)+((cd)^{-1}\Lambda_\tau/\Lambda_\tau)$.
Donc
$\frac{(g_{0,d})^{c^2}(c^{*}g_{0,d})^{-1}}{(g_{0,c})^{d^2}(d^{*}g_{0,c})^{-1}}$
est une fonction holomorphe sur $\C/\Lambda_\tau$; en particulier, elle est une
fonction constante non nulle $u$. Comme $g_{0,c}$ est stable sous
l'application de norme $N_a$ pour $(a,c)=1$, il en r\'esulte que
pour $a=2$ (resp. $3$), $u^4=u$ ( resp. $u^9=u$). Donc $u=1$.
\end{proof}

Soit $(\alpha,\beta)\in(\Q/\Z)^2$; on d\'efinit une action de $\SL_2(\Z)$ \`a droite par le produit de matrices usuel:
\[\text{si } \gamma=(\begin{smallmatrix}a_0& b_0\\ c_0& d_0\end{smallmatrix}), (\alpha,\beta)*\gamma=(a_0\alpha+c_0\beta, b_0\alpha+d_0\beta) .\]
Pour $(\alpha,\beta)\in(\Q/\Z)^2$,
choisissons un rel\`evement $(a,b)$ de $(\alpha,\beta)$ dans
$\Q^2$, et posons
\[g_{c,\alpha,\beta}(\tau)=g_{0,c}(a\tau+b).\]

\begin{prop}\label{siegle}Soient $\alpha,\beta\in\frac{1}{N}\Z/\Z$.
\begin{itemize}
\item[(i)] Si $(c,6)=1$ et $(c\alpha,c\beta)\neq(0,0)$, alors $g_{c,\alpha,\beta}$ est une unit\'e modulaire dans
$\cU(\Gamma_N,\Q(\z_N))$.
\item[(ii)] Si $(c,6)=1$, l'\'el\'ement $g_{\alpha,\beta}=g_{c,\alpha,\beta}^{1/(c^2-1)}$
de $\Q\otimes\cU(\overline{\Q})$ ne d\'epend pas du choix de $c\equiv1\mod N$. De plus, on a
$g_{c,\alpha,\beta}=g_{\alpha,\beta}^{c^2}g_{c\alpha,c\beta}^{-1}.$
\end{itemize}
\end{prop}

\begin{proof}[D\'emonstration] Soit $(a,b)$ un rel\`evement de $(\alpha,\beta)$ dans $\Q^2$. \\
$(i)$ Comme $(c\alpha,c\beta)\neq(0,0)$, on obtient $a\tau+b\notin c^{-1}\Z+c^{-1}\Z\tau$. Comme $g_{0,c}$ est une unit\'e de l'anneau des fonctions holomorphes sur $\C/\Lambda_\tau-(c^{-1}\Lambda_\tau)/\Lambda_\tau$,
la fonction $g_{c,\alpha,\beta}(\tau)$ n'a ni z\'eros, ni
p\^oles sur $\cH$. 

Consid\'erons le $q$-d\'eveloppement de la fonction $g_{c,\alpha,\beta}$, on trouve que
les coefficients sont dans $\Q(\z_N)$. Donc il suffit de
v\'erifier que c'est une fonction modulaire pour le sous groupe de
congruence $\Gamma_{N}$.

Consid\'erons la fonction $\theta(\tau, a\tau+b)$ sous l'action de
$\gamma\in\SL_2(\Z)$. L'action de
$\gamma$ sur $(\begin{smallmatrix}\tau\\1\end{smallmatrix})$ nous donne
 une base nouvelle du r\'eseau $(\tau,1)$, et  l'action de $\gamma$ sur $a\tau+b$ est donn\'ee par transformation
de M\"obius (i.e. le point $a\tau+b$ est envoy\'e en $a\gamma\tau+b$).

Comme $(c,6)=1$, on a $12|(c^2-1)$. Si $\gamma\in\Gamma_N$, on a $(\alpha,\beta)*\gamma=(\alpha,\beta)$. Avec les propri\'et\'es fondamentales de la fonction $\theta$ ci-dessus, on a\footnote{ le deuxi\`eme \'egalit\'e est du fait que $\theta$ est homog\`ene de degr\'e $0$ .} alors 
\begin{equation*}
\begin{split}
g_{c,\alpha,\beta}(\gamma\tau)&=\frac{\theta(\gamma\tau, a\gamma\tau+b)^{c^2}}{\theta(\gamma\tau,c(a\gamma\tau+b))}=\frac{\theta(\gamma(\begin{smallmatrix}\tau\\1\end{smallmatrix}),(aa_0+bc_0)\tau+b_0a+d_0b)^{c^2}}{\theta(\gamma(\begin{smallmatrix}\tau\\1\end{smallmatrix}),c((aa_0+bc_0)\tau+ab_0+bd_0))}\\
&=\frac{\z_{12,\gamma}^{c^2}\theta(\tau,(aa_0+bc_0)\tau+b_0a+d_0b)^{c^2}e(\frac{c^2\pi ic_0((aa_0+bc_0)\tau+b_0a+d_0b)^2}{(c_0\tau+d_0)})}{\z_{12,\gamma}\theta(\tau,c((aa_0+bc_0)\tau+b_0a+d_0b))e(\frac{\pi ic_0(c((aa_0+bc_0)\tau+b_0a+d_0b))^2}{(c_0\tau+d_0)})}\\
&=\z_{12,\gamma}^{c^2-1}\frac{\theta(\tau,(aa_0+bc_0)\tau+b_0a+d_0b)^{c^2}}{\theta(\tau,c((aa_0+bc_0)\tau+b_0a+d_0b))}=g_{c,(\alpha,\beta)*\gamma}(\tau)=g_{c,(\alpha,\beta)}(\tau).
\end{split}
\end{equation*}
$(ii)$ Si on \'evalue $c^{*}g_{0,d}(z)$ en $a\tau+b$, on obtient
$c^{*}g_{0,d}(a\tau+b)=g_{0,d}(ca\tau+cb)$. Soient $c\equiv d\equiv 1 [N]$. On a $c^{*}g_{0,d}(a\tau+b)=g_{d,\alpha,\beta}$ et $d^{*}g_{0,c}(a\tau+b)=g_{c,\alpha,\beta}$. On en d\'eduit que
$(g_{d,\alpha,\beta})^{c^2}(g_{d,\alpha,\beta})^{-1}=(g_{c,\alpha,\beta})^{d^2}(g_{c,\alpha,\beta})^{-1}$.
Autrement dit, $g_{\alpha,\beta}=g_{c,\alpha,\beta}^{1/(c^2-1)}$ ne
d\'epend pas du choix de $c\equiv 1 \mod N$. Soient $(c,6N)=1$ et $d\equiv 1\mod N$. En \'evaluant en $a\tau+b$,  la r\'elation (\ref{siegel}) se traduit en
$(g_{d,\alpha,\beta})^{c^2}g_{d,c\alpha,c\beta}^{-1}=g_{c,\alpha,\beta}^{d^2}g_{c,\alpha,\beta}^{-1}$, et donc
$g_{c,\alpha,\beta}=(g_{\alpha,\beta})^{c^2}g_{c\alpha,c\beta}^{-1}$.

\end{proof}
\begin{remark} Il y a une preuve geom\'etrique de cette proposition, en utilisant l'espace de module des courbes elliptiques dans \cite{KK}.
\end{remark}

 L'action de $\GL_2(\hat{\Z})$ sur $\cM^{\con}(\Q^{\cycl})$ induit celle de $\GL_2(\hat{\Z})$ sur $\Q\otimes \cU^{\con}(\Q^{\cycl})$.
\begin{lemma}\label{ac}Soit $(\alpha,\beta)\in (\Q/\Z)^2$ et soit $\gamma\in \GL_2(\hat{\Z})$. Alors on a $g_{\alpha,\beta}*\gamma=g_{(\alpha,\beta)*\gamma}.$
\end{lemma}
\begin{proof}[D\'emonstration]
Soient $\alpha,\beta\in \frac{1}{N}\Z/\Z$. Si $\gamma\in\SL_2(\hat{\Z})$, on peut choisir un entier $c$ tel que $(c,6)=1$, et $(c\alpha,c\beta)\neq (0,0)$. On a d\'eja calcul\'e l'action de $\gamma\in\SL_2(\Z)$ sur $g_{c,\alpha,\beta}$ dans la proposition ci-dessus: $g_{c,\alpha,\beta}*\gamma=g_{c,(\alpha,\beta)*\gamma}.$
Celle-ci induit la formule de l'action de $\gamma\in\SL_2(\Z)$ sur $g_{\alpha,\beta}$: $g_{\alpha,\beta}*\gamma=g_{(\alpha,\beta)*\gamma}$.
Elle se prolonge par continuit\'e en une action de $\SL_2(\hat{\Z})$ avec la m\^eme formule. Si $\gamma=(\begin{smallmatrix}1&0\\0& d\end{smallmatrix})$ avec $d\in\hat{\Z}^{*}$, l'action de $\gamma$ sur $g_{c,\alpha,\beta}$ est donn\'ee par la formule:
$g_{c,\alpha,\beta}*\gamma=g_{c,\alpha,d\beta}=g_{c,(\alpha,\beta)*\gamma}$,
 cela se voit directement sur le $q$-d\'eveloppement de $g_{c,\alpha,\beta}$ (on rappelle que $(\begin{smallmatrix}1&0\\0&d\end{smallmatrix})$ agit par $\sigma_d$). Alors, l'action de $\gamma\in\GL_2(\hat{\Z})$ sur $g_{\alpha,\beta}$ est donn\'ee par la formule: $g_{\alpha,\beta}*\gamma=g_{(\alpha,\beta)*\gamma}$.

\end{proof}

\begin{theo}Il existe une distribution alg\'ebrique $z_{\mathbf{Siegel}}\in\fD_{\alg}((\Q\otimes\hat{\Z})^2-(0,0),\Q\otimes\cU(\Q^{\cycl}))$,
telle que, quels que soient $r\in\Q^{*}$ et $(a,b)\in\Q^2-(r\Z, r\Z)$, on
ait:
\[\int_{(a+r_1\hat{\Z})\times(b+r_2\hat{\Z})}z_{\mathbf{Siegel}}=g_{r^{-1}a,r^{-1}b}(\tau).\]
De plus, $z_{siegel}$ est invariante sous l'action de
$\Pi^{'}_{\Q}$.
\end{theo}

\begin{proof}[D\'emonstration] 

Soit $e$ un entier $\geq 1$ et soit $N$ le plus petit entier tel que  $ Nr,Na$  et $Nb$ sont des entiers. On peut trouver $c\in\Z$ tel que $(c,e)=1$ et $c
\equiv 1\mod N$. Alors on a 
$\prod_{l=0}^{e-1}\prod_{k=0}^{e-1}g_{\frac{kr+a}{er},\frac{b+lr}{er}}
=N_a(g_{c,r^{-1}a,r^{-1}b})^{1/(c^2-1)}=g_{c,r^{-1}a,r^{-1}b}^{1/(c^2-1)}$.

Si $r_1,r_2\in \Q^{*}$ deux nombres rationnelles distincts et si $(a_1,a_2)\in\Q^2-(r_1\Z, r_2\Z)$, alors il existe $r\in \Q^{*}$ tel que $k_1=\frac{r}{r_1}$ et 
$k_2=\frac{r}{r_2}$ appartient \`a $\Z$. Pour $i=1,2$,  $a_i+r_i\hat{\Z}$ est une r\'eunion disjointe des $ a_i+kr_i+r \hat{\Z}$ pour $ 0\leq k\leq k_i$. On se remarque que   
si $r_1,r_2\in \Q^{*}$ deux nombres rationnels et si $(a_1,a_2)\in\Q^2-(r\Z, r\Z)$, l'int\'egration de $z_{\mathbf{Siegel}}$ sur $(a_1+r_1\hat{\Z})\times(a_2+r_2\hat{\Z}) $ est 
d\'efinie par la formule:
\[\int_{(a_1+r_1\hat{\Z})\times(a_2+r_2\hat{\Z}) } z_{\mathbf{Siegel}}=\prod_{e_1=0}^{k_1} \prod_{e_2=0}^{k_2}g_{r^{-1}(a_1+e_1r_1), r^{-1}(a_2+e_2r_2)}.\]
Ceci ne d\'epend pas du choix de $r$ \`a gr\^ace au calcul au-dessus et d\'efinit une relation de distribution alg\'ebrique. 
Donc $z_{siegel}$ est une distribution alg\'ebrique.

Par d\'efinition de $z_{\mathbf{Siegel}}$ et $g_{\alpha,\beta}$, $z_{\mathbf{Siegel}}$ est une distribution alg\'ebrique \`a
valeurs dans $\Q\otimes
\cU(\Q^{\cycl})$. L'action de
$\Pi^{'}_{\Q}$ sur $\cM^{\con}(\Q^{\cycl})$ se factorise \`a travers
$\GL_2(\Q\otimes\hat{\Z})$. Donc en utilisant le m\^eme argument
que dans le th\'eor\`eme \ref{eiskj}, il suffit de v\'erifier
l'invariance pour $\gamma\in
\GL_2(\hat{\Z}),\gamma=\bigl(\begin{smallmatrix}r_0&0\\0&r_0\end{smallmatrix}\bigr)$
et $\gamma=\bigl(\begin{smallmatrix}1&0\\0&e\end{smallmatrix}\bigr)$
respectivement. L'action de $\GL_2(\Q)_{+}$  sur l'espace des formes modulaires $\cM(\ol{\Q})$ est
donn\'ee par la formule (\ref{etoi}) dans ``notations", c'est-\`a-dire, $(f*\gamma)(\tau)= f(\gamma\tau)$ car on est en poids $0$.
\begin{itemize}
\item Si $\gamma\in\GL_2(\hat{\Z})$, l'assertion est de la d\'efinition de $z_{\mathbf{Siegel}}$ et du lemme \ref{ac}. 
\item Si $\gamma=\bigl(\begin{smallmatrix}r_0&0\\0& r_0\end{smallmatrix}\bigr)$, on a:
\begin{equation*}
\begin{split}
\int_{(a+r\hat{\Z})\times(b+r\hat{\Z})}z_{\mathbf{Siegel}}*\gamma&=\bigl(\int_{(\frac{a}{r_0}+\frac{r}{r_0}\hat{\Z})\times(\frac{b}{r_0}+\frac{r}{r_0}\hat{\Z})}z_{\mathbf{Siegel}}\bigr)*\gamma=(g_{r^{-1}a,r^{-1}b})*\gamma\\
&=(\det\gamma)^{1-k}g_{r^{-1}a,r^{-1}b}|_{\gamma}=g_{r^{-1}a,r^{-1}b};
\end{split}
\end{equation*}
\item Si $\gamma=\bigl(\begin{smallmatrix}1&0\\0& e\end{smallmatrix}\bigr)$, on a:
\begin{equation*}
\begin{split}
\int_{(a+r\hat{\Z})\times(b+r\hat{\Z})}z_{\mathbf{Siegel}}*\gamma&=\bigl(\int_{(a+r\hat{\Z})\times(\frac{b}{e}+\frac{r}{e}\hat{\Z})}z_{\mathbf{Siegel}}\bigr)*\gamma=\prod\limits_{i=0}^{e-1}\bigl(\int_{(a+r\hat{\Z})\times(\frac{b}{e}+\frac{ir}{e}+r\hat{\Z})}z_{\mathbf{Siegel}}\bigr)*\gamma\\
&=\prod\limits_{i=0}^{e-1}g_{r^{-1}a,\frac{r^{-1}b+i}{e}}(\frac{\tau}{e})=\prod\limits_{i=0,j=0}^{e-1}g_{\frac{r^{-1}a+j}{e},\frac{r^{-1}b+i}{e}}(\tau)=g_{r^{-1}a,r^{-1}b}.
\end{split}
\end{equation*}
\end{itemize}
\end{proof}

\subsection{Th\'eorie de Kummer $p$-adique}\label{section3}

\subsubsection{Th\'eorie de Kummer $p$-adique}
Soit $G$ un groupe localement profini. Soit $X$ un espace topologique
localement profini muni d'une action continue de $G$ \`a droite. Soit $M$ un $G$-module topologique muni d'une action \`a droite de
$G$. On note $\rH^{i}(G,M)$ le $i$-i\`eme groupe de cohomologie
continue de $G$ \`a valeurs dans $M$. Si $X$ est de plus muni d'une $G$-action \`a gauche
( not\'ee $(g,x)=g\cdot x$ ) commutant \`a l'action \`a droite de
$G$, alors $\rH^{i}(G,\fD_{\alg}(X,M))$ est muni d'une structure de
$G$-module \`a gauche donn\'ee par la formule:
\[ \int_X g\cdot\mu=\int_{gX}\mu \text{ si } \mu\in\rH^{i}(G,\fD_{\alg}(X,M)).\]

Posons: $X_1=(\Q\otimes\hat{\Z})^2-(0,0),X_2=\{\bigl(\begin{smallmatrix}a & b\\ c &
d\end{smallmatrix}\bigr)\in\bM_2(\Q\otimes\hat{\Z})|(a,b),(c,d)\in X_1\},$ et $X_2^{(p)}:=\bM_2(\Q\otimes\hat{\Z})^{(p)}\subset X_2.$
Dans la section
pr\'ec\'edente, on a obtenu une distribution alg\'ebrique
\[z_{\mathbf{Siegel}}\in\fD_{\alg}\bigl(X_1,\Q\otimes\cU(\Q^{\cycl})\bigr),\]
qui est invariante sous l'action de $\Pi_{\Q}^{'}$.

Notons $Z^{0}=\{(x_n)_{n\in\N}|
x_n\in\cU(\overline{\Q}),(x_{n+1})^p=x_n\}$
et $Z=Z^0\otimes\Q$. Alors $Z$ est muni d'une action de
$\Pi_{\Q}^{'}$, agissant sur chaque composante de $Z$.

On d\'efinit une projection  $\theta$ de $Z^0$ sur $\cU(\ol{\Q})$ en envoyant $(x_n)_{n\in\N}$ \`a $x_0$.
\begin{lemma}La projection $\theta:
Z^0\ra \cU(\ol{\Q})$ est surjective, dont le noyau est
\[\ker(\theta)=\{(1,\eps_p,\eps_{p^{n}},...\eps_{p^{n}},...)\}\cong\Z_p(1).\]
\end{lemma}
\begin{proof}[D\'emonstration]Soit $\Gamma$ un sous-groupe de $\SL_2(\Z)$ d'indice fini et soit $x\in\cU(\ol{\Q})$ une unit\'e modulaire pour $\Gamma$. Alors $x^{\frac{1}{p}}$ est encore une fonction holomorphes sur $\cH$ et m\'eromorphe sur $\cH\cup\bP^1(\Q)$.

Soit $\gamma\in \Gamma$, alors $x^{\frac{1}{p}}*\gamma=\z_{p,\gamma}x^{\frac{1}{p}}$, o\`u $\z_{p,\gamma}$ est une racine d'unit\'e d'ordre $p$ qui d\'epend de $\gamma$; ce qui nous fournit un caract\'ere $\chi$ de $\Gamma$ sur $\mu_p$ le groupe des racine d'unit\'e d'ordre $p$. Par cons\'equent, le noyau du carat\`ere $\chi$ est un sous-groupe de $\Gamma$ d'indice fini, qui fixe $x^{\frac{1}{p}}$. Ceci permet de conclure la surjectivit\'e.

Le reste est imm\'ediat.
\end{proof}

Comme $\Q$ est plat sur $\Z$, on obtient une suite exacte de
$\Pi_{\Q}^{'}$-modules:
\[0\ra\Q_p(1)\ra Z\ra\Q\otimes\cU(\overline{\Q})\ra 0.\]
Cela nous fournit une suite exacte de $\Pi_{\Q}^{'}$-modules :
\[0\ra\fD_{\alg}(X_1,\Q_p(1))\ra \fD_{\alg}(X_1,Z)\ra\fD_{\alg}(X_1,\Q\otimes\cU(\overline{\Q}))\ra 0.\]
En prenant la cohomologie continue de $\Pi_{\Q}^{'}$, on obtient une suite exacte longue:
\begin{equation*}
\begin{split}
0&\ra\rH^{0}(\Pi_{\Q}^{'},\fD_{\alg}(X_1,\Q_p(1)))\ra
\rH^{0}(\Pi_{\Q}^{'},\fD_{\alg}(X_1,Z))\\
&\ra
\rH^{0}(\Pi_{\Q}^{'},\fD_{\alg}(X_1,\Q\otimes\cU(\overline{\Q})))
\xrightarrow{\delta}\rH^{1}(\Pi_{\Q}^{'},\fD_{\alg}(X_1,\Q_p(1)))\\
&\ra\rH^{1}(\Pi_{\Q}^{'},\fD_{\alg}(X_1,Z))\ra\cdots
\end{split}
\end{equation*}
On appelle ``application de Kummer" le morphisme $\delta$ .
Notons
\[z_{\mathbf{Siegel}}^{(p)}\in\rH^{1}(\Pi_{\Q}^{'},\fD_{\alg}(X_1,\Q_p(1))),\]
 l'image de $z_{\mathbf{Siegel}}$ par l'application de Kummer.
\begin{lemma}\label{1co}Il existe une distribution alg\'ebrique
$\mu\in\fD_{\alg}(X_1,Z)$ telle que $z_{\mathbf{Siegel}}^{(p)}$ est l'image du $1$-cocycle
 $\sigma\ra\mu*\sigma-\mu$ dans $\rH^{1}(\Pi_{\Q}^{'},\fD_{\alg}(X_1,\Q_p(1)))$ .
\end{lemma}
\begin{proof}[D\'emonstration]Soit $\{\phi_i\}_{i\in I}$ une base de $\LC_c(X_1,\Z)$ sur $\Z$. On peut fabriquer une distribution alg\'ebrique $\mu\in \fD_{\alg}(X_1,Z)$ en prenant pour $\int_{X_1}\phi_i\mu$ n'importe quel rel\`evement dans $Z$ de $\int_{X_1}\phi_iz_{\mathbf{Siegel}}$ et alors $z_{\mathbf{Siegel}}^{(p)}$ est l'image du $1$-cocycle
 $\sigma\ra\mu*\sigma-\mu$ dans $\rH^{1}(\Pi_{\Q}^{'},\fD_{\alg}(X_1,\Q_p(1))).$
\end{proof}

En utilisant l'application de cup-produit, on obtient un \'el\'ement 
\[z_{\mathbf{Siegel}}^{(p)}\otimes z_{\mathbf{Siegel}}^{(p)}\in \rH^{2}(\Pi_{\Q}^{'},\fD_{\alg}(X_2,\Q_p(2))).\]
On d\'efinit $z_{kato}$ comme l'image de $z_{\mathbf{Siegel}}^{(p)}\otimes
z_{\mathbf{Siegel}}^{(p)}$ sous l'application de restriction:
\[\rH^{2}(\Pi_{\Q}^{'},\fD_{\alg}(X_2,\Q_p(2)))\ra\rH^{2}(\Pi_{\Q}^{(p)},\fD_{\alg}(X_2^{(p)},\Q_p(2))).\]

\subsubsection{Passer \`a la  mesure}\label{gcd}

Soit $\langle\cdot\rangle: \Z_p^{*}\ra \hat{\Z}^{*}$ l'inclusion naturelle en envoyant $x$ sur $\langle x\rangle=(1,\cdots,x,1,\cdots)$, o\`u $x$ est \`a la place $p$. Consid\'erons l'inclusion de $\hat{\Z}^{*}$ dans $\GL_2(\hat{\Z})$ en envoyant $d$ sur $(\begin{smallmatrix}d&0\\0&d\end{smallmatrix})$. D'apr\`es le lemme $\ref{ac}$, cela d\'efinit une action de $d\in\hat{\Z}^{*}$ sur l'unit\'e de Siegel $g_{\alpha,\beta}\in\Q\otimes\cU(\overline{\Q})$ par la formule:
 \[d\cdot g_{\alpha,\beta}=g_{d\alpha,d\beta}=g_{\alpha,\beta}*(\begin{smallmatrix}d&0\\0&d\end{smallmatrix}),\]
o\`u l'action $*$ est celle de $\GL_2(\hat{\Z})$ sur $g_{\alpha,\beta}$.

Rappelons que l'on a le diagramme commutatif suivant:
\[
\xymatrix{ 0\ar[r] &\Z_p(1) \ar[r]\ar[d]&Z^{0}
\ar[r]\ar[d]&\cU(\overline{\Q})
\ar[r]\ar[d] &0\\
0\ar[r] &\Q_p(1) \ar[r] &Z
\ar[r]&\Q\otimes\cU(\overline{\Q}) \ar[r]
&0},\]
o\`u $Z^{0}$ et $\cU(\ol{\Q})$ sont des $\Z$-modules. En tensorant par $\Z_p$, on  obtient un diagramme commutatif suivant de $\Z_p$-modules:
\[
\xymatrix{ 0\ar[r] &\Z_p(1) \ar[r]\ar[d]&\Z_p\otimes Z^{0}
\ar[r]\ar[d]&\Z_p\otimes \cU(\overline{\Q})
\ar[r]\ar[d] &0\\
0\ar[r] &\Q_p(1) \ar[r] &\Q_p\otimes Z
\ar[r]&\Q_p\otimes\cU(\overline{\Q}) \ar[r]
&0}.\]
De plus, la th\'eorie de Kummer $p$-adique ci-dessus s'applique aussi \`a la suite exacte
\[0\ra \Q_p(1) \ra \Q_p\otimes Z
\ra \Q_p\otimes\cU(\overline{\Q}) \ra 0. \]
Soit $u\in\Z_p^*$, on d\'efinit un op\'erateur $r_u$ sur $g_{\alpha,\beta}$ par la formule $r_ug_{\alpha,\beta}=(u^2-\langle u\rangle)g_{\alpha,\beta}$, o\`u l'action de $u^2$ est la multiplication par $u^2$ et l'action de $\langle u\rangle$ est d\'efinie ci-dessus.
\begin{lemma}\label{tor}Soit $u\in\Z_p^{*}$. L'\'el\'ement $r_u( g_{\alpha,\beta})$ appartient \`a $\Z_p\otimes \cU(\Q^{\cycl})$.
\end{lemma}
\begin{proof}[D\'emonstration]Rappelons que on a $N\alpha=N\beta=0$. On choisit un entier $d$ tel que $(d,6N)=1$; alors on a  $g_{\alpha,\beta}=\frac{1}{d^2-1}\otimes g_{d,\alpha,\beta}$. Soit $u_0$ un entier congru \`a $\langle u\rangle$ modulo $pN$. Alors $\langle u\rangle g_{\alpha,\beta}=g_{u_0\alpha,u_0\beta}$.

On a la relation:
$
r_u(g_{\alpha,\beta})
=\frac{u^2-u_0^2}{d^2-1}g_{d,\alpha,\beta}+(u_0^2-\langle u\rangle)g_{\alpha,\beta}=\frac{u^2-u_0^2}{d^2-1}g_{d,\alpha,\beta}+g_{u_0,\alpha,\beta}.$
Alors, $\frac{u^2-u_0^2}{d^2-1}g_{d,\alpha,\beta}$ appartient \`a $\Z_p\otimes \cU(\Q^{\cycl})$ si la valuation $p$-adique de $u-u_0$ assez grande.
Par cons\'equent, $r_ug_{\alpha,\beta}$ appartient \`a $\Z_p\otimes \cU(\Q^{\cycl})$.
\end{proof}

D'apr\`es le lemme ci-dessus, si $u\in \Z_p^*$,  alors on pose $g_{u,\alpha,\beta}=r_u(g_{\alpha,\beta})$; c'est un \'el\'ement de $\Z_p\otimes \cU(\Q^{\cycl})$.
Ceci induit un op\'erateur $r_u$ sur $z_{\mathbf{Siegel}}\in \fD_{\alg}(X_1,\Q\otimes\cU(\ol{\Q}))$ par la formule:
\[\int_{(a+r\hat{\Z})\times (b+r\hat{\Z})}r_uz_{\mathbf{Siegel}}=g_{u,\frac{a}{r},\frac{b}{r}}.\]
\begin{lemma} Si $u\in\Z_p^{*}$, alors
$r_u(z_{\mathbf{Siegel}})$ est une distribution sur $X_1$ \`a valeurs dans $\Z_p\otimes\cU(\Q^{\cycl})$, qui est invariante sous l'action de $\Pi_{\Q}^{'}$.
\end{lemma}
\begin{proof}[D\'emonstration]

Par d\'efintion, quel que soient $r\in\Q^*$ et $(a,b)\in\Q^2-(r\Z,r\Z)$, on a  \[\int_{(a+r\hat{\Z})\times (b+r\hat{\Z})}r_uz_{\mathbf{Siegel}}=r_u(g_{\frac{a}{r},\frac{b}{r}}).\]

De plus, $r_u(z_{\mathbf{Siegel}})$ est une distribution alg\'ebrique sur $X_1$ car $z_{\mathbf{Siegel}}$ l'est.
D'apr\`es le lemme $\ref{tor}$, on a que $r_u(z_{\mathbf{Siegel}})$ est  \`a valeurs dans $\Z_p\otimes\cU(\Q^{\cycl})$. 

Comme l'action de $\Pi_{\Q}^{'}$ sur $\Z_p\otimes\cU(\Q^{\cycl})$ se factorise \`a travers $\GL_2(\Q\otimes\hat{\Z})$, il suffit de v\'erifier l'invariance pour $\gamma\in\GL_2(\Q\otimes\hat{\Z})$. 
Ceci se d\'eduit de la commutativit\'e de l'action de $r_u$ et celle de $\Pi_\Q^{'}$ sur  $z_{\mathbf{Siegel}}$.
\end{proof}

 Par la th\'eorie de Kummer, $z_{\mathbf{Siegel}}^{(p)}$ est l'image du $1$-cocycle $\sigma\mapsto \mu*\sigma-\mu$, o\`u $\mu\in\fD_{\alg}(X_1,\Q_p\otimes Z)$ un rel\`evement de $z_{\mathbf{Siegel}}$ dans $\fD_{\alg}(X_1,\Q_p\otimes Z)$. Si $u\in\Z_p$, alors on d\'efinit $r_u(z_{\mathbf{Siegel}}^{(p)})$ l'image du $1$-cocycle $\sigma\mapsto \mu'*\sigma-\mu'$, o\`u $\mu'\in\fD_{\alg}(X_1,\Z_p\otimes Z)$ est un rel\`evement de $r_u(z_{\mathbf{Siegel}})$ dans $\fD_{\alg}(X_1,\Z_p\otimes Z)$. Donc $r_u(z_{\mathbf{Siegel}}^{(p)})$ est un \'el\'ement de $\rH^1(\Pi_{\Q}^{'},\fD_{\alg}(X_1^{(p)},\Z_p(1)))$.

Soient $c,d\in\Z_p^{*}$, on d\'efinit un op\'erateur $r_{c,d}$ sur $z_{\mathbf{Siegel}}^{(p)}\otimes z_{\mathbf{Siegel}}^{(p)}$ par la formule \[r_{c,d}(z_{\mathbf{Siegel}}^{(p)}\otimes z_{\mathbf{Siegel}}^{(p)})=r_c(z_{\mathbf{Siegel}}^{(p)})\otimes r_d(z_{\mathbf{Siegel}}^{(p)}).\] Donc,
pour tous $c,d\in \Z_p^{*}$, l'\'el\'ement $r_{c,d}(z_{\mathbf{Siegel}}^{(p)}\otimes z_{\mathbf{Siegel}}^{(p)})$ appartient \`a  $\rH^{2}(\Pi_{\Q}^{'},\fD_{0}(X_2,\Z_p(2)))$.
Ceci permet de d\'efinir
$z_{\mathbf{Kato},c,d}:=r_{c,d}z_{\mathbf{Kato}}$ comme l'image de $r_{c,d}(z_{\mathbf{Siegel}}^{(p)}\otimes z_{\mathbf{Siegel}}^{(p)})$ sous l'application de restriction: \[\rH^{2}(\Pi_{\Q}^{'},\fD_{\alg}(X_2^{(p)},\Z_p(2)))\ra\rH^{2}(\Pi_{\Q}^{(p)},\fD_{\alg}(X_2^{(p)},\Z_p(2))).\]
Par ailleurs, tout \'el\'ement de $\fD_{\alg}(X_2,\Z_p(2))$ s'\'etend par continuit\'e en une mesure (i.e. forme lin\'eaire continue sur les fonctions continues) sur $X_2$ \`a valeurs dans $\Z_p(2)$. Donc $z_{\mathbf{Kato},c,d}$ peut \^etre vu comme un \'el\'ement \`a valeurs dans l'espace  $\fD_0(X_2,\Z_p(2))$.

\subsubsection{Torsion \`a la Soul\'e}\label{soule}
On note $t=(\z_{p^n})_{n\in\N}$, le g\'en\'erateur canonique\footnote{ D'habitude, il n'y pas de g\'en\'erateur canonique de $\Z_p(1)$. Par contre, dans notre cas, on a fix\'e un plongement de $\ol{\Q}$ dans $\C$ et $\ol{\Q}_p$ respectivement, et on pose $\z_{p^n}=e^{\frac{2i\pi}{p^n}}$.} de
$\Z_p(1)$ et l'action de $\gamma\in\GL_2(\Z_p)$ sur $\Z_p(1)$ est
par multiplication par $\det\gamma$. On note $V_p=\Q_p e_1\oplus\Q_p
e_2$ la repr\'esentation de dimension $2$ de $\GL_2(\Z_p)$ donn\'ee
par les formules suivantes: si
$\gamma=\bigl(\begin{smallmatrix}a&b\\c&d\end{smallmatrix}\bigr)\in\GL_2(\Z_p)$,
$e_1*\gamma=ae_1+be_2$ et $e_2*\gamma=ce_1+de_2$. Si $k\geq 2$ et
$j\in\Z$, on note $V_{k,j}=\Sym^{k-2}V_p\otimes\Q_p(2-j)$.

Rappelons $X_2^{(p)}:=\bM_2(\Q\otimes\hat{\Z})^{(p)}=\GL_2(\Z_p)\times \bM_2(\Q\otimes\hat{\Z}^{]p[})$. Soit $x\in X_2^{(p)}$; on note $x_p=(\begin{smallmatrix}a_p&b_p\\c_p&d_p\end{smallmatrix})\in\GL_2(\Z_p)$ la composante de $x$ en $p$, qui est un \'el\'ement dans $\GL_2(\Z_p)$.
On consid\`ere la multiplication d'une mesure $\mu\in\fD_{0}(X_2^{(p)},\Z_p(2))$ par la fonction \[x\mapsto (e_1^{k-2}t^{-j})*x_p=(a_pe_1+b_pe_2)^{k-2}((\det x_p)t)^{-j},\] qui est donn\'ee par l'action de $\GL_2(\Z_p)$ sur $V_{k,j}$ et qui est continue
sur $X_2^{(p)}$. Ceci nous donne une mesure $(e_1^{k-2}t^{-j})*x_p\otimes \mu $ sur $X_2^{(p)}$ \`a valeurs dans $V_{k,j}$.

\begin{lemma}\label{torsion}La multiplication d'une mesure $\mu\in\fD_{0}(X_2^{(p)},\Z_p(2))$ par la fonction $x\mapsto (e_1^{k-2}t^{-j})*x_p$ induit un morphisme de $\Z_p[[\Pi_{\Q}^{(p)}]]$-modules de
$\fD_{0}(X_2^{(p)},\Z_p(2))$ dans $\fD_{0}(X_2^{(p)},V_{k,j})$.
\end{lemma}
\begin{proof}[D\'emonstration]
On applique la formule (\ref{actiondis})
\[(\phi*g)(x)=\phi(x*g^{-1}) \text{ et } \int_{X}\phi(\mu*g)=(\int_{X}(\phi*g^{-1})\mu)*g. \]
dans la ``notation'' au cas $X=X_2^{(p)}$, $G=\Pi_{\Q}^{(p)}$ et $V=\Z_p(2)$ ou $V=V_{k,j}$.
Le groupe $G=\Pi_{\Q}^{(p)}$ agit contin\^ument sur $X$ \`a travers $\GL_2(\Q\otimes\hat{\Z})^{(p)}=\GL_2(\Z_p)\times \GL_2(\Q\otimes\hat{\Z}^{]p[})$ par la multiplication de matrices usuelle \`a droite.

Soient $\phi\in \LC_{c}(X_2^{(p)},\Z_p)$,
$\tau\in\Pi_{\Q}^{(p)}$ et $\mu\in\fD_{0}(X_2^{(p)},\Z_p(2))$.

Si on consid\`ere $\mu*\tau\in\fD_{0}(X_2^{(p)},\Z_p(2))$, alors $\tau$ agit sur $e_{1}^{k-2}t^{-j}*x_p$
comme l'action sur une fonction et la formule (\ref{actiondis}) se traduit par  $\int\phi(x)(\mu*\tau)=\int\phi(x\tau)\chi_{\cycl}^2(\tau)\mu$, o\`u $x\tau$ est donn\'e par l'action de $\Pi_{\Q}^{(p)}$ sur $X_2^{(p)}$. Alors, on a la formule:
\[\int_{X_2^{(p)}}\phi(x)\bigl((e_1^{k-2}t^{-j}*x_p)\otimes(\mu*\tau)\bigr)
=\int_{X_2^{(p)}}\chi_{\cycl}^2(\tau)\phi(x\tau)\bigl((e_1^{k-2}t^{-j}*(x\tau)_p)\otimes\mu\bigr).\]

Si on consid\`ere $\left((e_{1}^{k-2}t^{-j}*x_p)\otimes\mu\right)*\tau\in\fD_{0}(X_2^{(p)},V_{k,j})$, alors l'action de $\tau$ sur $e_{1}^{k-2}t^{-j}*x_p$ est donn\'ee par l'action de $\tau$ sur l'espace $V_{k,j}$ et la formule (\ref{actiondis}) se traduit par
\[\int_{X_2^{(p)}}\phi(x)\left((e_1^{k-2}t^{-j}*x_p)\otimes\mu\right)*\tau=\int_{X_2^{(p)}}\chi_{\cycl}^2(\tau)\phi(x\tau)\bigl((e_1^{k-2}t^{-j}*(x\tau)_p)\otimes\mu\bigr).\]
La comparaison des deux formules permet de conclure.
\end{proof}
D'apr\`es le lemme $\ref{torsion}$, la multiplication par $e_1^{k-2}t^{-j}*x_p$ induit un morphisme naturel:
\[\rH^{2}(\Pi_{\Q}^{(p)},\fD_0(X_2^{(p)},\Z_p(2)))\ra\rH^{2}(\Pi_{\Q}^{(p)},\fD_0(X_2^{(p)},V_{k,j})).\]

Donc on peut d\'efinir, pour $j\in\Z$, 
\[z_{\mathbf{Kato},c,d}(k,j)=((e_1^{k-2}t^{-j})*x_p)\otimes z_{\mathbf{Kato},c,d}\in\rH^{2}(\Pi_{\Q}^{(p)},\fD_0(X_2^{(p)},V_{k,j})),\]
o\`u $\Pi_{\Q}^{(p)}$ agit sur $V_{k,j}$ \`a travers son quotient
$\GL_2(\Z_p)$. 

%% file: anneaux.tex
\section{Les anneaux de Fontaine}
\subsection{Le corps $\fK$ et les formes modulaires}
\subsubsection{Le corps $\fK$ }\label{corps}
Soit $\fK^+=\Q_p\{\frac{q}{p}\}$ l'alg\`ebre des fonctions analytiques sur la boule $v_p(q)\geq 1$ \`a coefficients dans $\Q_p$;  c'est un anneau principal complet pour la valuation $v_{p,\fK}$ d\'efinie par la formule:
\[ v_{p,\fK}(f)=\inf_{n\in\N}v_p(a_n), \text{ si } f=\sum_{n\in\N}a_nq^n/p^n\in \fK^+ .\]
Cette valuation est aussi la valuation spectrale: $v_{p,\fK}(f)=\inf_{v_p(q)\geq 1} v_p(f(q))$.
La restriction de la valuation $v_{p,\fK}$ \`a $\Q_p$ co\"incide avec la valuation $p$-adique normalis\'ee $v_p$ sur $\Q_p$. Dans la suite, on notera $v_p$ au lieu de $v_{p,\fK}$. L'anneau $\fK^+$ est un anneau de Dedekind, et donc chaque id\'eal premier de  $\fK^+$  d\'efinit une valuation sur  $\fK^+$. En particulier, on a la valuation normalis\'ee $v_q$ (i.e. $v_q(q)=1$) correspondant \`a l'id\'eal premier $(q)$ de $\fK^+$. 

On note $\fK$ le compl\'et\'e du corps des fractions de l'anneau $\fK^+$ pour la valuation $v_p$. Fixons une cl\^oture alg\'ebrique $\overline{\fK}$ de $\fK$. Comme $\fK$ est un corps complet pour la valuation $v_p$, on peut prolonger  $v_p$ sur $\fK$ \`a $\overline{\fK}$ de mani\`ere unique par la formule: \[v_p(x)=\frac{1}{[\fK[x]:\fK]}v_p(N_{\fK[x]/\fK}(x)), \text{ si } x\in \ol{\fK}.\]
On note le groupe de Galois de $\overline{\fK}$ sur $\fK$ par $\cG_{\fK}$.
\begin{remark}Il existe une mani\`ere de prolonger la valuation spectrale en une valuation spectrale sur $\ol{\fK}$: si $x\in\ol{\fK}$, on note $P(X)=X^n+a_1X^{n-1}+\cdots+a_n\in\fK[X]$ le polyn\^ome caract\'eristique de $y\mapsto xy$, $\forall y\in\fK[x]$. On d\'efinit la valuation spectrale $v_{\spe}$ sur $\fK[x]$ par la formule: $v_{\spe}(x)=\min_{1\leq i\leq n}\frac{v_p(a_i)}{i}.$
Elle co\"incide avec la valuation $v_p$ sur $\ol{\fK}$.
\end{remark}
Soit $M\geq 1$ un entier. On note $q_M$ ( resp. $\z_M$  ) une racine
$M$-i\`eme $q^{1/M}$ ( resp. $\exp(\frac{2i\pi}{M})$ ) de $q$ ( resp.
$1$ ). On note $F_M=\Q_p[\z_M]$. Soit $\fK_M=\fK[q_{M},\z_M]$ ; c'est une extension galoisienne de $\fK$.
Soit $\fF_M=\fK[\z_M]$ la sous-extension de $\fK_M$ sur $\fK$, qui est aussi une extension galoisienne de $\fK$; la cl\^oture int\'egrale $\fF_M^+$ de $\fK^+$ dans $\fF_M$ est $\fK^+[\z_M]$, qui est l'anneau des fonctions analytiques sur la boule $v_p(q)\geq 1$ \`a coefficients dans $F_M$. Alors, $\fK_M$ est une extension de Kummer de $\fF_M$ de groupe de Galois cyclique d'ordre $M$, dont un g\'en\'erateur $\sigma_M$ est d\'efini par son action sur $q_M$: $\sigma_Mq_M=\z_Mq_M.$

On note $\fK_{\infty}$ (resp. $\fF_{\infty}, F_\infty$) la r\'eunion des $\fK_M$ (resp. $\fF_M, F_M$), $M\geq 1$. On note $P_{\Q_p}$ (resp. $P_{\ol{\Q}_p}$) le groupe de Galois de $\ol{\Q}_p\fK_{\infty}$ sur $\fK$ (resp. $\ol{\Q}_p\fK$) . Le groupe $P_{\ol{\Q}_p}$ est un groupe profini qui est isomorphe au groupe $\hat{\Z}$. De plus, on a une suite exacte:
\[0\ra P_{\ol{\Q}_p}\ra P_{\Q_p}\ra \cG_{\Q_p}\ra 0.\]

Fixons $M$ un entier $\geq 1$. On note $\fK_{Mp^{\infty}}$ (resp. $\fF_{Mp^\infty}, F_{Mp^{\infty}}$) la r\'eunion des $\fK_{Mp^n}$ (resp. $\fF_{Mp^n}, F_{Mp^\infty}$), $n\geq 1$, ainsi que $P_{\fK_M}$ le groupe de Galois de $\fK_{Mp^{\infty}}$ sur $\fK_M$. On note $U_{\fK_M}$ le groupe de Galois de $\fK_{Mp^{\infty}}$ sur $\fF_{Mp^{\infty}}$, qui est isomorphe au groupe $\Z_p$, et on note $\Gamma_{\fK_M}$ le groupe de Galois de $\fF_{Mp^{\infty}}$ sur $\fK_M$, qui est isomorphe au groupe $\Gal(F_{Mp^{\infty}}/F_M)$. On a une suite exacte:
$0\ra U_{\fK_M}\ra P_{\fK_M}\ra \Gamma_{\fK_M}\ra 0.$

Soit $\ol{\fK}^{+}$ la cl\^oture int\'egrale  de $\fK^{+}$ dans $\ol{\fK}$.  La cl\^oture int\'egrale de $\Q_p$ dans $\ol{\fK}$ est une cl\^oture alg\'ebrique de $\Q_p$. Donc on a une inclusion $\ol{\Q}_p\subset \ol{\fK}^+$.
On note $\fK_M^+$ la cl\^oture int\'egrale de $\fK^+$ dans $\fK_M$, qui est aussi la cl\^oture int\'egrale de $\fF_M^+$ dans $\fK_M$.

\begin{lemma}\label{M}
\begin{itemize}
\item[(1)]Si $M\geq 1$ est un entier, on a $\fK_M^+=\fK^+[\z_M,q_M]$. En particulier, $\fK_M^+$ est l'anneau des fonctions analytiques sur la boule $v_p(q_M)\geq \frac{1}{M}$.
    \item[(2)]La valuation $v_p$ sur $\fK_M^+$ est donn\'ee par la formule:  \[v_p(x)=\inf_{v_p(q_M)\geq \frac{1}{M}}v_p(\sum_{n=0}^{+\infty}a_n(x)q_M^n).\]
    \end{itemize}
\end{lemma}
\begin{proof}[D\'emonstration]
\begin{itemize}
\item[(1)]Si $x\in \fK_M^+$, il s'\'ecrit uniquement sous la forme  $x=\sum_{i=0}^{M-1}b_iq_M^i$ avec $b_i\in \fF_M$.
Comme $xq_M^{M-i}\in\fK_M^+$, on a $\frac{1}{M}\tr_{\fK_M/\fF_M}(xq_M^{M-i})=b_iq\in\fF_M^+$ et donc $b_i\in q^{-1}\fF_M^+$. 
La valuation normalis\'ee  $v_q$ de $\fK^+$ s'\'etend de mani\`ere unique en une valuation $v_{q}$ sur $\fK_M$ car $(q)$ est totalement ramifi\'e.
On constate que $v_q(\fF^+_M[q^{-1}])= \Z$ et les $v_q(b_iq_M^i)$ sont distincts deux \`a deux. On en d\'eduit que $0\leq v_q(x)= \inf_{i}(v_q(b_i)+\frac{i}{M})$  et  $\inf_i v_q(b_i)\geq 0$. 
Ceci permet de conclure que  $b_i\in \fF_M^+$ pour tous $i$.

Si $a\in \fF_M^+$, alors $a$ peut s'\'ecrire uniquement sous la forme  $a=\sum_{j=0}^{+\infty}a_{j}q^j$, o\`u $a_{j}$ est une suite d'\'el\'ements de $F_M$ telle que $\lim_{j\ra +\infty}v_p(a_{j})+j=+\infty$. On conclut le $(1)$ en appliquant cette \'ecriture \`a $b_i$ pour $0\leq i\leq M-1$.
\item[(2)]D'apr\`es $(1)$, $\fK^+_M$ est l'anneau des fonctions analytiques sur la boule $v_p(q_M)\geq \frac{1}{M}$ \`a coefficients dans $F_M$ et donc il est muni d'une valuation spectrale $v$ donn\'ee par la formule: $v(x)=\inf_{v_p(q_M)\geq \frac{1}{M}}v_p(\sum_{n=0}^{+\infty}a_n(x)q_M^n)$.
L'anneau $\fK^+$ s'identifie \`a un sous anneux de $\fK_{M}^+$ par changement de variable:
$f(q)=\sum_{n=0}^{+\infty}a_nq^n=\sum_{n=0}^{+\infty}a_{Mn}q_M^{Mn}$ si $f(q)\in \fK^+$. Alors la restriction de la valuation spectrale $v$ sur $\fK_M^+$ \`a $\fK^+$ co\"incide avec la valuation $v_p$ sur $\fK^+$. On d\'eduit la formule pour la valuation dans le lemme puisqu'il existe une mani\`ere unique de prolonger la valuation $v_p$ sur $\fK$ \`a $\fK_M$.
\end{itemize}
\end{proof}

Posons $I_n=\{\frac{a}{p^{n-1}}|a\in\N, 0\leq a\leq (p-1)p^{n-1}-1\}$ et $J_n=\{\frac{b}{p^n}|b\in\N, 0\leq b\leq p^n-1\}$. On note $I$ $($resp. $J)$  la r\'eunion des $I_n$ $($resp. $J_n)$, $n\geq 1$. Le lemme suivant, qui d\'ecrit l'\'ecriture explicite d'un \'el\'ement de $\fK_{Mp^{\infty}}^+$, est une cons\'equence directe du lemme \ref{M}.

\begin{lemma}\label{ec}Soit $M\geq 1$ un entier tel que $v_p(M)\geq v_p(2p)$ $($resp. $v_p(M)<v_p(2p))$.\\
$(1)$Les $\{q_M^j\}_{j\in J}$ forment une base de $\fK^+_{Mp^{\infty}}$ sur $\fF^+_{Mp^{\infty}}$. De plus, tout $x\in \fK_{Mp^{\infty}}^+$ peut s'\'ecrire uniquement sous la forme $\sum_{j\in J}\sum_{k\in\N}a_{jk}(x)q_M^{j+k}$, pour une suite double $a_{jk}(x)\in F_{Mp^\infty}$, telle que,
    \begin{itemize}
    \item[(i)]quel que soit $j\in J$, la s\'erie $\sum_{k\in\N}a_{jk}q_M^k$ converge dans $\fF_{Mp^{\infty}}^+$;
        \item[(ii)]L'ensemble des $j\in J$ tels qu'il existe $k\in\N$ avec $a_{jk}\neq 0$ est fini .
            \end{itemize}
            $(2)$Les $\{\z_M^iq_M^j\}_{(i,j)\in J\times J}$ $($resp. $\{\z_M^iq_M^j\}_{(i,j)\in I\times J})$ forment une base de $\fK^+_{Mp^{\infty}}$ sur $\fK^+_M$. De plus, tout $x\in \fK_{Mp^{\infty}}^+$ peut s'\'ecrire uniquement sous la forme
                \[\sum_{j\in J}\sum_{i\in J}\sum_{k\in \N}a_{ijk}(x)\z_M^iq_M^{j+k} \ (\text{resp.}  \sum_{j\in J}\sum_{i\in I}\sum_{k\in \N}a_{ijk}\z_M^iq_M^{j+k}), \]
                 pour une suite triple $a_{ijk}(x)\in F_M$ telle que :
                \begin{itemize}
                \item[(i)]Quel que soit $(i,j)\in J\times J$ $($resp. $(i,j)\in I\times J)$, 
                la s\'erie $\sum_{k\in\N}a_{ijk}(x)q_M^{k}$ converge dans $\fK_M^+$ $($i.e. $\lim_{k\ra+\infty}v_p(a_{ijk}(x))+\frac{k}{M}=+\infty.)$
                    \item[(ii)] L'ensemble $\{(i,j)| \exists k\in \N, a_{ijk}(x)\neq 0\}$ est fini.
                        \end{itemize}
                        $(3)$La valuation $v_p$ sur $\fK_{Mp^{\infty}}^+$ est donn\'ee par la formule:
                            \[v_p(x)=\inf_{v_p(q)\geq 1}(v_p(\sum_{j\in J}\sum_{i\in J}\sum_{k\in \N}a_{ijk}(x)\z_M^iq_M^{j+k})) (\text{resp.}\inf_{v_p(q)\geq 1}(v_p(\sum_{j\in J}\sum_{i\in I}\sum_{k\in \N}a_{ijk}(x)\z_M^iq_M^{j+k}))).\]
                        
\end{lemma}

\subsubsection{Le th\'eor\`eme d'Ax-Sen-Tate}
Soit $L$ un anneau de caract\'eristique $0$, muni d'une valuation $v_p$ normalis\'ee par $v_p(p)=1$. On note $\C(L)$ le compl\'et\'e de $L$ pour la valuation $v_p$. Le reste de ce paragraphe est de montrer un analogue (le th\'eor\`eme \ref{Ax}) du th\'eor\`eme d'Ax-Sen-Tate et de donner une description de l'anneau $\C(\fK_{Mp^{\infty}}^+)$ (c.f. Corollaire \ref{C}).

Soit $H$ un sous-groupe ferm\'e de $\cG_{\fK_M}$; si $\alpha\in\ol{\fK}$, on d\'efinit le diam\`etre $\Delta_H(\alpha)$ par rapport \`a un sous-groupe ferm\'e $H$ de $\cG_{\fK_M}$ par $\Delta_{H}(\alpha)=\inf_{g\in H}(v_p(g\alpha-\alpha))$. Notons que $\alpha\in \ol{\fK}^H$ si et seulement si $\Delta_H(\alpha)=+\infty$.

\begin{lemma}Soit $P(X)\in \ol{\fK}[X]$ $($resp. $\in\ol{\fK}^+[X])$, unitaire de degr\'e $n$, dont toutes les racines v\'erifient $v_p(\alpha)\geq u$.
\begin{itemize}
\item[(1)]Si $n=p^kd$ avec $(p,d)=1$ et $d\geq 0$ et si $l=p^k$, alors le polyn\^ome $P^{(l)}$, d\'eriv\'ee $l$-i\`eme de $P$, a au moins une racine $\beta\in\ol{\fK}$ $($resp.$\ol{\fK}^+)$ v\'erifiant $v_p(\beta)\geq u$.
\item[(2)]Si $n=p^{k+1}$ et $l=p^k$, alors $P^{(l)}$ a au moins une racine $\beta\in\ol{\fK}$ $($resp. $\ol{\fK}^+)$ v\'erifiant $v_p(\beta)\geq u-\frac{1}{p^{k+1}-p^k}$.
\end{itemize}
\end{lemma}

\begin{proof}[D\'emonstration]
La d\'emonstration pour $P\in\ol{\fK}[X]$ est classique. Par exemple, elle se trouve dans les notes du cours\footnote{ http://www.math.jussieu.fr/~colmez/nombres-p-adiques.pdf } de Colmez.  Elle marche aussi pour $P\in\ol{\fK}^+[X]$.
\end{proof}

\begin{lemma}[Ax]\label{Ax1}Il existe une constante $C=\frac{p}{(p-1)^2}$ telle que
\begin{itemize}
\item[(1)] si $\alpha\in\ol{\fK}$, alors il existe $a\in\fK_{Mp^{\infty}}$ v\'erifiant $v_p(\alpha-a)\geq \Delta_{\cG_{\fK_{Mp^{\infty}}}}(\alpha)-C$;
    \item[(2)]si $\alpha\in\ol{\fK}^+$, alors il existe $a\in\fK_{Mp^{\infty}}^+$ v\'erifiant $v_p(\alpha-a)\geq \Delta_{\cG_{\fK_{Mp^{\infty}}}}(\alpha)-C$;
\end{itemize}
\end{lemma}
\begin{proof}[D\'emonstration]Le $(1)$ correspond au cas trait\'e par Ax; nous ne traiterons donc que le $(2)$. Soit $\alpha\in \ol{\fK}^+$ tel que $[\fK_{Mp^{\infty}}[\alpha]:\fK_{Mp^{\infty}}]=n$ et soit $l(n)$ est le plus grand entier $l$ tel que $p^{l}\leq n$. On montre par r\'ecurrence sur $n$ le r\'esultat suivant:
il existe $a\in \fK_{Mp^{\infty}}^+$ v\'erifiant $v_p(a-\alpha)\geq \Delta_{\cG_{\fK_{Mp^{\infty}}}}(\alpha)-\sum_{i=1}^{l(n)}\frac{1}{p^{i}-p^{i-1}}$;  ce qui permet de conclure le lemme car $\sum_{i=1}^{+\infty}\frac{1}{p^{i}-p^{i-1}}=\frac{1}{(p-1)^2}$.

On prend la constante $u=\Delta_{\cG_{\fK_{Mp^\infty}}}(\alpha)$ dans le lemme pr\'ec\'edent. Posons $P(X)=Q(X+\alpha)$, o\`u $Q$ est le polyn\^ome minimal de $\alpha$ sur $\fK_{Mp^{\infty}}^+$. Comme les racines de $P$ sont les $\sigma(\alpha)-\alpha$, pour $\sigma\in \cG_{\fK_{Mp^{\infty}}}$,  $P(X)$ v\'erifie la condition du lemme pr\'ec\'edent.\\
$(i)$Si $n$ n'est pas une puissance de $p$, il existe $l\in \N$ tel que le polyn\^ome $Q^{(l)}\in \fK_{Mp^{\infty}}^+[X]$ ait une racine $\beta\in \ol{\fK}^+$ v\'erifiant $v_p(\beta-\alpha)\geq u$. D'autre part, si $\sigma\in \cG_{\fK_{Mp^{\infty}}}$, alors on a
   \[v_p(\sigma(\beta)-\beta)\geq \min(v_p(\sigma(\beta-\alpha)), v_p(\sigma(\alpha)-\alpha), v_p(\alpha-\beta))\geq u. \]
  Par ailleurs, on a $[\fK_{Mp^{\infty}}[\beta]:\fK_{Mp^{\infty}}]=\deg Q^{(l)}<n$. Cela permet de conclure par l'hypoth\`ese de r\'ecurrence.\\
$(ii)$Si $n=p^{k+1}$, on peut trouver une racine $\beta$ de $Q^{(p^k)}(X)$ v\'erifiant les conditions
    \[\label{loc}v_p(\beta-\alpha)\geq u-\frac{1}{p^{k+1}-p^k}
   \text{ et } [\fK_{Mp^{\infty}}[\beta]:\fK_{Mp^{\infty}}]<n.\] 
  On tire l'existence de $a\in \fK_{Mp^{\infty}}^+$, de l'hypoth\`ese de r\'ecurrence,  v\'erifiant 
  \[v_p(\beta-a)\geq \Delta_{\cG_{\fK_{Mp\infty}}}(\beta)-\sum_{i=1}^{k}\frac{1}{p^{i}-p^{i-1}} .\] On a $\Delta_{\cG_{\fK_{Mp\infty}}}(\beta)\geq u$ de la m\^eme mani\`ere de $(i)$.  Ceci permet de conclure le lemme.
\end{proof}

\begin{theo}[Ax-Sen-Tate]\label{Ax}
$(1)$Le corps $\fK_{Mp^{\infty}}$ est dense dans $\C(\ol{\fK})^{\cG_{\fK_{Mp^{\infty}}}}$.\\
$(2)$L'anneau $\fK_{Mp^{\infty}}^+$ est dense dans $\C(\ol{\fK}^+)^{\cG_{\fK_{Mp^{\infty}}}}$.
\end{theo}
\begin{proof}[D\'emonstration]On montre le $(2)$ et le $(1)$ se d\'eduit de la m\^eme mani\`ere.
L'inclusion $\C(\fK_{Mp^{\infty}}^+)\subset \rH^0(\cG_{\fK_{Mp^{\infty}}},\C(\ol{\fK}^+))$ est imm\'ediate. Il suffit de montrer l'inverse.

Si $\alpha\in  \rH^0(\cG_{\fK_{Mp^{\infty}}},\C(\ol{\fK}^+))$, il existe une suite $\alpha_n\in\ol{\fK}^+$ v\'erifiante $v_p(\alpha-\alpha_n)\geq n$. On note $\Delta_n=\Delta_{\cG_{\fK_{Mp^{\infty}}}}(\alpha_n)$. Quel que soit $\sigma\in\cG_{\fK_{Mp^{\infty}}}$, on a $v_p(\sigma(\alpha_n)-\alpha_n)\geq \min\{ v_p(\sigma(\alpha_n-\alpha)), v_p(\alpha_n-\alpha)\}\geq n .$ Ceci implique que $\Delta_n\geq n$.
 Par ailleurs, quel que soit $n\geq 1$, par le lemme d'Ax \ref{Ax1}, il existe un \'el\'ement $a_n\in\fK_{Mp^{\infty}}^+$ tel que $v_p(\alpha_n-a_n)\geq \Delta_n-\frac{p}{(p-1)^2} \geq n-\frac{p}{(p-1)^2}$. Ceci implique que  $v_p(\alpha-a_n)\geq n-\frac{p}{(p-1)^2} $ et $\alpha$ est la limite de la suite $\{a_n\}_{n\in\N}$.
\end{proof}

\begin{coro}\label{C}Soit $M\geq 1$ un entier tel que $v_p(M)\geq v_p(2p)$ $($resp. $v_p(M)<v_p(2p))$.\\
$(1)$ Tout \'el\'ement $x$ de $\C(\ol{\fK}^+)^{\cG_{\fK_{Mp^\infty}}}$ s'\'ecrit uniquement sous la forme
\[\sum_{i\in J}\sum_{ j\in J}\sum_{ k\in \N}a_{ijk}(x)\z_M^iq_M^{j+k}\  (\text{resp.} \sum_{i\in I}\sum_{ j\in J}\sum_{k\in\N}a_{ijk}(x)\z^iq_M^{j+k}),\]
pour une suite triple $\{a_{ijk}(x)\}_{i\in J , j\in J,k\in\N} ($ resp. $\{a_{ijk}(x)\}_{i\in I , j\in J, k\in\N} )$ d'\'el\'ements de $F_M$  telle que, $\forall N\in\N$, $\{(i,j,k)\in J\times J\times \N : v_p(a_{ij}(x))+\frac{j+k}{M}\leq N \}$ $($resp. $\{(i,j,k)\in I\times J\times \N: v_p(a_{ij}(x))+\frac{j+k}{M}\leq N  \})$  est fini.
\\
$(2)$ La valuation $v_p$ sur $\C(\fK_{Mp^{\infty}}^+)$ est donn\'ee par la formule:
        \[v_p(x)=\inf_{v_p(q)\geq 1}v_p(\sum_{i\in J}\sum_{ j\in J}\sum_{ k\in \N}a_{ijk}(x)\z_M^iq_M^{j+k}) \ ( \text{resp.} \inf_{v_p(q)\geq 1}v_p(\sum_{i\in I}\sum_{ j\in J}\sum_{ k\in \N}a_{ijk}(x)\z_M^iq_M^{j+k})).\]
        
\end{coro}
\begin{proof}[D\'emonstration]On d\'eduit le corollaire du lemme $\ref{ec}$ et du th\'eor\`eme d'Ax-Sen-Tate $\ref{Ax}$.
\end{proof}

\subsubsection{Trace de Tate normalis\'ee}

Dans ce paragraphe, on construira une application $\fK_{M}^{+}$-lin\'eaire continue $\rT_M$,  appel\'ee la trace de Tate normalis\'ee, de $\fK_{Mp^{\infty}}^+$ dans  $\fK_{M}^+$ pour tous $M\in\N$, et d\'ecrira ses propri\'et\'es. L'ingr\'edient principal est la construction de la trace de Tate normalis\'ee pour l'alg\`ebre de Tate de type $\Z_p\{T,T^{-1}\}$,  celui qui est bien \'etudi\'ee par Andreatta et Brinon dans $\cite{AB}$.
 
On note $R=\Z_p\{\frac{q}{p},\frac{p}{q}\}$ l'alg\`ebre de Tate dans les variables $q/p, p/q$ \`a coefficients dans $\Z_p$. Si $m\in \N$, on note $R_m$  la cl\^oture int\'egrale de $R$ dans $ \Frac(R)[\z_{p^{m}}]$ et on note $ R_{\infty}=\cup_{n\in\N}R_n$.
Si $m,n\in\N$, on note $R_m^n $ la cl\^oture int\'egrale de $R$ dans $\Frac(R_m)[(\frac{q}{p})^{\frac{1}{p^n}}] $. On pose $R_{\infty}^n=\cup_{m\in\N}R_m^n$ et $R^{\infty}_{\infty}=\cup_{n\in\N}R_n^n$. Si $m\geq n$, on a $R_m^n=R_m\otimes_{R_n}R_n^n$ . Par ailleurs, on a un isomorphisme de groupes de Galois $\Gal(R_{\infty}^{\infty}[\frac{1}{p}]/R_{\infty}[\frac{1}{p}])\cong\Gal(\fK_{p^{\infty}}/\fF_{p^{\infty}})\cong \Z_p$. On note $u$ un des g\'en\'erateurs; on pose $u_m=u^{p^m}$ qui est un g\'en\'erateur du groupe de Galois $\Gal(R_\infty^{\infty}[\frac{1}{p}]/R_\infty^m[\frac{1}{p}])\cong U_{\fK_{p^m}}$.

Si $m$ est un entier $\geq 0$,
on d\'efinit une application $R^m_{\infty}[\frac{1}{p}]$-lin\'eaire $\Lambda_m: R_{\infty}^\infty[\frac{1}{p}]\ra R_{\infty}^m[\frac{1}{p}] $ par la formule: \[\Lambda_m(x)=\frac{1}{p^{n-m}}\tr_{R_n^n[\frac{1}{p}]/R_n^m[\frac{1}{p}]}(x), \text{ si } n\geq m \text{ et } x\in R_n^n[\frac{1}{p}] .\]

\begin{lemma}
\begin{itemize}
\item[(1)]Il existe une constante $C$ telle que pour tout $m\in\N$ et $x\in R_{m+1}^{m+1}[\frac{1}{p}]$, on a $v_p((u_m-1)x)\geq v_p(x)+\frac{1}{p-1}-Cp^{-m}$.
\item[(2)]$\Lambda_m$ est continue. Plus pr\'ecis\'ement, on a $v_p(\Lambda_m(x))\geq v_p(x)-\frac{C}{p^{m-1}}$.
\end{itemize}
\end{lemma}
\begin{proof}[D\'emonstration]Ce lemme est un cas particulier dans (\cite{AB}, lemme $3.7$ et $3.8$ ), et l'ingr\'edient principal est la contr\^olation de ramification dans ($\cite{A}$, corollaire $3.10$).\\
$(1)$ D'apr\`es ($\cite{A}$, corollaire $3.10$), il existe une constante $C$ telle que, pour tous $m\geq 1$, on a
\[ p^{C/p^m}R_{m+1}^{m+1}\subset \oplus_{i=0}^{p-1}R_{m+1}^m\cdot(\frac{q}{p})^{\frac{i}{p^{m+1}}}\subset R_{m+1}^{m+1}.\]
Si $x\in R_{m+1}^{m+1}[\frac{1}{p}]$, $x$ s'\'ecrit sous la forme $\sum\limits_{i=0}^{p-1}x_i(\frac{q}{p})^{\frac{i}{p^{m+1}}}$, o\`u $x_i\in R_{m+1}^{m}$ et $v(x_i)\geq v(x)-\frac{C}{p^m}$. On en d\'eduit que 
$v_p((u_m-1)x)=v_p(\sum\limits_{i=0}^{p-1}x_i(\z_p^i-1)(\frac{q}{p})^{\frac{i}{p^{m+1}}})\geq v_p(x)+\frac{1}{p-1}-\frac{C}{p^{m}}.$\\
$(2)$ Si $x\in R_{m+1}^{m+1}[\frac{1}{p}]$,  on a $p\Lambda_m (x)=(\sum\limits_{i=0}^{p-1}u_m^i)(x)=((1-u_m)^{p-1}+pP(u_m))(x)$, o\`u $P(X)\in \Z[X]$ et $P(1)=1$. 
On d\'eduit de $(1)$ que $v_p(\Lambda_m (x))\geq v_p(x)-(p-1)\frac{C}{p^m}$. En cons\'equence, si $x\in R_{n}^n[\frac{1}{p}]$, on a
    \[v_p(\Lambda_m(x))\geq v_p(\Lambda_{m+1}(x))-(p-1)\frac{C}{p^m}\geq\cdots\geq v_p(x)-(p-1)\frac{C}{p^m}(\sum_{i=0}^{n-1}\frac{1}{p^i})\geq v_p(x)-\frac{C}{p^{m-1}}.\]
\end{proof}

On note $\pi_n=\z_{p^{n+1}}-1$.  Comme $(\frac{q}{p})^{\frac{1}{p^{n}}}=a\frac{q_{p^n}}{(\pi_n)^{p-1}}$, o\`u $a$ est une unit\'e de $\cO_{F_{p^{n+1}}}$, on a une inclusion naturelle de $\fK_{p^\infty}^+$ dans $R_{\infty}^\infty[\frac{1}{p}]$. 
Si $n\geq m$, on d\'efinit une application $\fK^+_{p^m}F_{p^n}$-lin\'eaire $\rT_{p^m,p^n}: \fK^+_{p^{\infty}}\ra \fK^+_{p^m}F_{p^n}$ en composant la restriction de $\Lambda_m$ \`a $\fK_{p^\infty}^+$ et la trace de Tate normalis\'ee de $F_{\infty}$ dans $F_{p^{n}}$. Soit $M=M_0p^m$ avec $(p,M_0)=1$. L'application $\rT_{p^m,p^m}$ s'\'etend en une trace de Tate normalis\'ee $\rT_{M,p^m}: \fK^+_{p^{\infty}}F_M\ra \fK^+_{p^m}F_M$ par $F_M$-lin\'eairit\'e.

\begin{lemma}\label{Q}La  trace de Tate normalis\'ee $\rT_{M,p^m}$ est $\fK^+_{p^m}F_M$-lin\'eaire,  donn\'ee par la formule:
\begin{equation*}
\begin{split}
\rT_{M,p^m}:  \fK^+_{p^{\infty}}F_M&\longrightarrow \fK^+_{p^m}F_M \\
\z_{M_0p^n}^{a}q_{p^n}^b &\mapsto
\left\{\begin{aligned}\z_{M_0p^n}^{a}q_{p^n}^b& ;
\text{ si } p^{n-m}|a \text{ et
} p^{n-m}|b;\\
0& ;   \text{ sinon}.
\end{aligned}\right.
\end{split}
\end{equation*}
\end{lemma}
Soit $e_1, \cdots, e_d$ une base de $\fK_M^+$ sur $\fK_{p^m}^+F_M$.  Comme $[\fK_{Mp^{\infty}}^+: \fK^+_{p^{\infty}}F_M]=[ \fK^+_M:\fK_{p^m}^+F_M]$, les $e_1, \cdots, e_d$ forment aussi une base de $\fK_{Mp^{\infty}}^+$ sur $\fK^+_{p^{\infty}}F_M$. Si $x\in \fK_{Mp^\infty}^+ $, alors $x$ peut s'\'ecrire uniquement sous la forme 
$x=\sum_{i=1}^da_i(x) e_i $ avec $a_i(x)\in \fK^+_{p^{\infty}}F_M$. On pose une application $\fK_M^+$-lin\'eaire $\rT_M:\fK_{Mp^{\infty}}^+\ra \fK_M^+ $ par la formule $\rT_M(x)=\sum_{i=1}^d \rT_{M,p^m}(a_i(x))e_i$.

\begin{remark}\label{base}Si on fixe une base de $\fK_M^+$ sur $\fK_{p^m}^+F_M$, les coefficients $a_i(x)$ dans l'\'ecriture de $x\in \fK_{Mp^n}$ se calculent de la mani\`ere suivante:
Soit $e_1^*, \cdots e_d^*$  la base duale de $\fK_M$ sur  $\fK_{p^m}F_M$ par rapport \`a la trace $\frac{1}{d}\tr_{\fK_M/\fK_{p^m}F_M}$. On a donc 
$a_i(x)= \frac{1}{d}\tr_{\fK_{p^nM}/\fK_{p^{m+n}}F_M}(xe_i^*). $

\end{remark}

\begin{lemma}\label{Tate0}Si $M\in\N$, l'application $\rT_M$ est continue et elle s'\'etend par continuit\'e en une application $\fK_{M}^+$-lin\'eaire $\rT_M:\C(\fK_{Mp^\infty}^+)\ra\fK_{M}^+$ qui commute \`a l'action de $\cG_\fK$.
\end{lemma}
\begin{proof}[D\'emonstration] Par construction, il  suffit de d\'emontrer le lemme pour $M=p^m$.  D'apr\`es Tate \cite{Ta}, la trace de Tate normalis\'ee de $F_{\infty}$ dans $F_{p^{m}}$ est continue. Comme $\Lambda_m$ est continue, on en d\'eduit que 
 $\rT_{p^m}$ est continue. 
 
 Si $i\in I , j\in J$ et $\sigma\in \cG_{\fK}$, on a $(\z^{i}q^j)*\sigma=\z^{(\chi_{\cycl}(\sigma)-1)i}\z^{c_q(\sigma)j}\z^{i}q^{j}$, o\`u $c_q$ est le $1$-cocycle associ\'e \`a $q$ \`a valeurs dans $\Z_p(1)$ par la th\'eorie de Kummer.
La commutativit\'e de $\rT_{p^m}$ et $\cG_{\fK}$ vient de la formule de l'action de $\cG_{\fK}$ et de la formule de $\rT_{p^m}$ sur $\fK_{p^{\infty}}^+$(c.f. lemme \ref{Q}).
\end{proof}

\subsubsection{Lien avec l'alg\`ebre $\cM(\ol{\Q})$ des formes moduaires}
En associant son $q$-d\'eveloppement \`a une forme modulaire, cela permet de voir les formes modulaires comme des \'el\'ements de $\ol{\Q}_p\fK_{\infty}^+$.

\begin{lemma}Le groupe de Galois $P_{\Q_p}$ de $\bar{\Q}_p\fK_{\infty}$ sur $\fK$  pr\'eserve l'alg\`ebre des formes modulaires $\cM(\ol{\Q})$; c'est-\`a dire, $P_{\Q_p}$ est un sous-groupe de $\Pi_{\Q}$.
\end{lemma}
 \begin{proof}[D\'emonstration]  Consid\'erons la suite exacte: $0\ra P_{\ol{\Q}_p}\ra P_{\Q_p}\ra \cG_{\Q_p}\ra 0,$ o\`u $P_{\ol{\Q}_p}\cong(\begin{smallmatrix}1&\hat{\Z}\\ 0&1\end{smallmatrix})$ le groupe de Galois de $\bar{\Q}_p\fK_{\infty}$ sur $\fK\ol{\Q}_p$. 
 Comme $\cG_{\Q_p}$ preserve $\cM(\ol{\Q})$, $\cG_{\Q_p}$ est un sous-groupe de $\Pi_{\Q}$.  On a deux actions de $(\begin{smallmatrix}1&\hat{\Z}\\ 0&1\end{smallmatrix})$ sur $\cM(\ol{\Q})$:
 \begin{itemize}
 \item[(1)]l'action induite par celle de $P_{\ol{\Q}_p}$ sur $\ol{\Q}_p\fK_{\infty}^+$;
 \item[(2)]l'action modulaire \'etendant par continuit\'e celle de $(\begin{smallmatrix}1&\Z\\ 0&1\end{smallmatrix})$ sur $\cM(\ol{\Q})$.
 \end{itemize}
 On constate que que ces deux actions de $(\begin{smallmatrix}1&\hat{\Z}\\ 0&1\end{smallmatrix})$ sur $\cM(\ol{\Q})$ co\"incident en comparant les formules: on a $(\begin{smallmatrix}1&b\\0&1\end{smallmatrix})q_M=q_M\z_M^b$ dans les deux cas. Donc $P_{\ol{\Q}_p}$ preserve l'alg\`ebre $\cM(\ol{\Q})$, ce qui permet de d\'eduire le lemme.
 \end{proof}

En composant les morphismes $\cG_{\fK}\ra P_{\Q_p}$, l'inclusion $P_{\Q_p}\ra \Pi_{\Q}$ induit un morphisme $\cG_\fK\ra \Pi_{\Q}$. De plus, si $M\geq 1$ un entier, $\cG_{\fK_M}$ est un sous-groupe de $\cG_{\fK}$ et donc on a un morphisme $\cG_{\fK_M}\ra \Pi_{\Q}$.


\subsection{Application de la construction de Fontaine \`a l'anneau $\fK^+$}
\subsubsection{La construction de Fontaine}

Soit $L$ un anneau de caract\'eristique $0$, qui est muni d'une
valuation $v_p$ telle que
 $v_{p}(p)=1.$ On note $\cO_L=\{x\in L, v_p(x)\geq 0\}$ l'anneau
des entiers de $L$ pour la topologie $p$-adique. On note $\cO_{\C(L)}$ le compl\'et\'e  de $\cO_{L}$ pour la valuation $v_p$. On pose $\C(L)=\cO_{\C(L)}[\frac{1}{p}]$.

\begin{defn}Soit $A_n=\cO_L/p\cO_L$ pour tous $n$; alors $\{A_n\}$ muni des morphismes de transition $ A_n\mapsto
A_{n-1}$ d\'efinis par l'application de Frobenius absolu
$x_n\mapsto x_n^p$ forme un syst\`eme
projectif. Notons $\R(L)=\plim A_n=\{(x_n)_{n\in \N}|
x_n\in \cO_L/p\cO_L \text{ et } x_{n+1}^p=x_n, \text{ si }n\in \N \}$.
\end{defn}

Si $x=(x_n)_{n\in\N}\in\R(L)$, soit $\hat{x}_n$ un rel\`evement de
$x_n$ dans $\cO_{\C(L)}$. La suite $(\hat{x}_{n+k}^{p^k})$ converge quand
$k$ tend vers l'infini. On note $x^{(n)}$ sa limite, qui ne
d\'epend pas du choix des rel\`evements $\hat{x}_n$. On obtient ainsi
une bijection: $\R(L)\ra\{(x^{(n)})_{n\in\N}| x^{(n)}\in\cO_{\C(L)},
(x^{(n+1)})^p=x^{(n)}, \forall n\}$. Si $x=(x^{(i)}), y=y^{(i)}$ sont deux \'el\'ements de $\R(L)$, alors leur somme $x+y$ et leur produit $xy$ sont donn\'es par:
$ (x+y)^{(i)}=\lim_{j\ra\infty}(x^{(i+j)}+y^{(i+j)})^{p^j} \text{ et } (xy)^{(i)}=x^{(i)}y^{(i)}.$

L'anneau $\R(L)$ est un anneau parfait  de caract\'eristique $p$ (i.e. le morphisme $x\mapsto x^p$ est bijectif. ).
On note $\A_{\inf}(L)$ l'anneau des vecteurs de Witt \`a coefficients dans
$\R(L)$. Alors $\A_{\inf}(L)$ est un
anneau $p$-adique (i.e. un anneau s\'epar\'e et complet pour la
topologie $p$-adique ), d'anneau r\'esiduel parfait de
caract\'eristique $p$. Si $x\in \R(L)$, on note
$[x]=(x,0,0,...)\in \A_{\Inf}(L)$ son repr\'esentant de Teichm\"uller.
Alors tout \'el\'ement $a$ de $\A_{\inf}(L)$ peut s'\'ecrire de
mani\`ere unique sous la forme $\sum\limits_{k=0}^{\infty}p^k[x_k]$
avec une suite $(x_k)\in (\R(L))^{\N}$.

On d\'efinit un morphisme d'anneaux $\theta: \A_{\inf}(L)\ra\cO_{\C(L)}$
par la formule
$\sum_{k=0}^{+\infty}p^k[x_k]\mapsto
\sum_{k=0}^{+\infty}p^kx_k^{(0)}.$
 On note
$\B_{\inf}(L)=\A_{\inf}(L)[\frac{1}{p}]$,  et on \'etend $\theta$ en
un morphisme $\B_{\inf}(L)\ra \C(L).$
 On note $\B_m(L)=\B_{\inf}(L)/(\ker\theta)^m$. On fait de $\B_m(L)$ un anneau de Banach en prenant l'image de $\A_{\inf}(L)$ comme anneau d'entiers.

On d\'efinit $\B_{\dR}^{+}(L):=\plim \B_m(L)$ comme le compl\'et\'e
$\ker (\theta)$-adique de $\B_{\Inf}(L)$; on le munit de la topologie de la limite projective, ce qui en fait un anneau de Fr\'echet. Donc $\theta$ s'\'etend en
un morphisme continu d'anneaux topologiques
$\B_{\dR}^{+}(L)\ra \C(L).$

On peut munir $\B_{\dR}^{+}(L)$  d'une filtration de la fa\c{c}on
suivante: pour $i\in\N$, notons $\Fil^i \B_{\dR}^{+}(L)$
la $i$-i\`eme puissance de l'id\'eal $\ker\theta$ de
$\B_{\dR}^{+}(L)$ .

L'anneaux $\A_{\inf}(L)$ s'identifie canoniquement \`a un sous-anneau de $\B_{\dR}^+(L)$ et
si $k\in\N, m\in\Z$, on pose \[U_{m,k}=p^m\A_{\inf}(L)+(\ker \theta)^{k+1}\B_{\dR}^{+}(L),\]
alors les $U_{m,k}$ forment une base de voisinages de $0$ dans $\B_{\dR}^{+}(L)$.

\begin{example}
\begin{itemize}
\item[(1)]Si $L=\Q_p$, alors la construction est triviale (i.e. $\B_{\dR}^+(\Q_p)=\Q_p$.)
\item[(2)]Si $L=\ol{\Q}_p$, on note $\rC_p=\C(\ol{\Q}_p)$, $\tilde{\rE}^+=\R(\ol{\Q_p}), \tilde{\rA}^+=\A_{\inf}(\ol{\Q}_p)$ , $\tilde{\rB}^+=\B_{\inf}(\ol{\Q}_p)$ 
et $\rB_{\dR}^+=\B_{\dR}^+(\ol{\Q}_p)$. On note $\tilde{\z}$ (resp. $\tilde{\z}_M$ pour un entier
$M\geq 1$) le repr\'esentant de Teichm\"uler de
$\z^{(1)}=(1,\z_p,\cdots,\z_{p^n},\cdots)$ (resp.
$\z^{(M)}=(\z_M,\cdots,\z_{Mp^n},\cdots)$) dans
$\A_{\Inf}(\ol{\fK}^+)$. Si $M|N$, alors on a
$\tilde{\z}_N^{N/M}=\tilde{\z}_M$. Le noyau du morphisme $\theta: \tilde{\rA}^+\ra\cO_{\rC_p}$ est un id\'eal principal engendr\'e  par $\omega=\frac{\tilde{\z}-1}{\tilde{\z}_p-1} $.
On pose $t=\log\tilde{\z}=-\sideset{}{}\sum_{n=1}^{+\infty}\frac{(1-\tilde{\z})^n}{n}$;
c'est le $2\pi i$ $p$-adique de Fontaine, qui appartient \`a $\rB_{\dR}^{+}$ et aussi engendre le noyau du morphisme $\theta:\rB_{\dR}^+\ra \rC_p$.
\item[(3)]On va consid\'erer les cas $L=\fK^+_{Mp^{\infty}}$ pour $M\geq 1\in \N$ et $L=\ol{\fK}^+$.
 Le $2\pi i$ $p$-adique de Fontaine $t$ appartient \`a $\B_{\dR}^{+}(L)$.
D'apr\`es la construction de Fontaine, on a un morphisme surjective d'anneaux: $\theta:\A_{\inf}(L)\ra \cO_{\C(L)}$ avec le noyau engendr\'e  par $\omega=\frac{\tilde{\z}-1}{\tilde{\z}_p-1} $. De plus, $\theta$ s'\'etend en un morphisme continu d'anneaux $\theta:\B_{\dR}^+(L)\ra \C(L),$ dont le noyau est l'id\'eal principal engendr\'e par $t$ ou $\omega$ , sur lequel $\cG_{\fK}$ agit par multiplication par $\chi_{\cycl}$.
\end{itemize}
\end{example}

\subsubsection{Trace de Tate normalis\'ee pour $\B_{\dR}^+(\fK_{Mp^{\infty}}^+)$}
Pour simplifier la notation, on note $\A_{\inf}$ (resp. $\B_{\inf}$ et $\B_{\dR}^+$ ) l'anneau $\A_{\inf}(\ol{\fK}^+)$ (resp. $\B_{\inf}(\ol{\fK}^+)$ et $\B_{\dR}^+(\ol{\fK}^+)$ ).
Soit $\tilde{q}$ ( resp. $\tilde{q}_M$ si $M\geq 1$ est un entier ) le repr\'esentant de Teichm\"uller dans $\A_{\Inf}$ de $(q,
q_p,\cdots,q_{p^n},\cdots)$ ( resp. $(q_M,\cdots,q_{Mp^n},\cdots)$ )
. Si $M|N$, on a
$\tilde{q}_N^{N/M}=\tilde{q}_M$.

On d\'efinit une application $\iota_{\dR}:\fK^+\ra \B_{\dR}^+$ par $f(q)\mapsto f(\tilde{q}) $; 
ce qui permet d'identifier $\fK^+$ \`a un sous-anneau de $\B_{\dR}^+$. 
On note $\alpha=\frac{\tilde{q}}{p}-[(\frac{q}{p},(\frac{q}{p})^{\frac{1}{p}},\cdots,)]$ 
et on a $\alpha\in\A_{\Inf}\cap\ker \theta$. Si $f(q)\in\fK^+$ est de valuation $p$-adique $\geq 0$ 
(i.e. $f(q)=\sum_{n\in\N}a_n(\frac{q}{p})^n\in\fK^+$ avec $a_n\in\Z_p$), on a
$f(\tilde{q})=\sum\limits_{n=0}^{+\infty}a_n(\alpha+[(\frac{q}{p},(\frac{q}{p})^{\frac{1}{p}},\cdots,)])^n
=\sum\limits_{k=0}^{+\infty}\alpha^k\sum_{n\geq k}a_n\binom{n}{k}[(\frac{q}{p},(\frac{q}{p})^{\frac{1}{p}},\cdots,)]^k$ et la s\'eries $\sum_{n\geq k}a_n\binom{n}{k}[(\frac{q}{p},(\frac{q}{p})^{\frac{1}{p}},\cdots,)]^k$ 
converge dans $\A_{\Inf}$. Donc $\iota_{\dR}$ est continue.
 Mais il faut faire attention au fait que $\iota_{\dR}(\fK^+)$ n'est pas stable par $\cG_{\fK}$ car $\tilde{q}\sigma=\tilde{q}\tilde{\z}^{c_q(\sigma)}$ si
$\sigma\in\cG_{\fK}$, o\`u $c_{q}$ est le $1$-cocycle \`a valeur dans $\Z_p(1)$ associ\'e \`a $q$
par la th\'eorie de Kummer.

Posons
$\tilde{\fK}^+=\iota_{\dR}(\fK^+)[[t]]$.  Si $M\geq 1$ est un entier, on note $\tilde{\fK}^+_M$ l'anneau
$\tilde{\fK}^+[\tilde{q}_M,\tilde{\z}_M]$ et pose $\tilde{\fK}_{Mp^{\infty}}^+=\bigcup_{n}\tilde{\fK}^+_{Mp^n}$. L'application $\iota_{\dR}$ s'\'etend en un morphisme continu 
de $\fK^+$-modules $\iota_{\dR}:\fK^+_M\ra \B_{\dR}^{+}$ en envoyant $\z_M$ et $q_M$ sur $\tilde{\z}_M$ et $\tilde{q}_M$ respectivement. 
On a $\tilde{\fK}_M^+=\iota_{\dR}(\fK_M^+)[[t]]$.

\begin{remark}
Le module $\iota_{\dR}(\fK_M^+)$ n'est pas un anneau car $\tilde{\z}=\tilde{\z}_M^M\notin \iota_{\dR}(\fK_M^+)$. Donc le morphisme $\iota_{\dR}:\fK^+_M\ra \B_{\dR}^{+}(\ol{\fK}^+) $ n'est plus un morphisme d'anneaux.
\end{remark}

\begin{lemma}\label{K}\begin{itemize}
\item[(1)]$\tilde{\fK}^+$ est stable sous l'action de
$\cG_{\fK}$.
\item[(2)]Le $\tilde{\fK}^+$-module $\tilde{\fK}_M^+$ est \'egal au $\tilde{\fK}^+$-module $\tilde{\fK}^+[\z_{M}, \tilde{q}_M]$.
\end{itemize}
\end{lemma}
\begin{proof}[D\'emonstration]
Si $\sigma\in\cG_{\fK}$,
$\sigma\tilde{q}=\tilde{q}\tilde{\z}^{c_q(\sigma)}$ et
$\tilde{\z}\in\fK[[t]]$. Alors $\sigma\tilde{q}\in\fK[[t]]$. Par ailleurs, on a
$\sigma t=\chi_{cycl}(\sigma)t$. Cela permet de conclure le $(1)$.
Le $(2)$ se d\'eduit du fait:  $\z_{M}=\tilde{\z}_M\exp(- \frac{t}{M}).$
\end{proof}

\begin{theo}Soit $M\geq 1$. On a
\begin{itemize}
\item[$(i)$]
$\rH^0(\cG_{\fK_{Mp^{\infty}}},\B_{\dR}^{+})=\B_{\dR}^{+}(\fK_{Mp^{\infty}}^+)$;
\item[$(ii)$] $\tilde{\fK}_{Mp^{\infty}}^+$ est dense dans
$\B_{\dR}^{+}(\fK_{Mp^{\infty}}^+)$;
\end{itemize}
\end{theo} 
Le th\'eor\`eme se d\'eduit des lemmes $\ref{DR}$ et $\ref{DR1}$ suivants:

\begin{lemma} \label{DR}Tout \'el\'ement $x$ de $(\B_{\dR}^+)^{\cG_{\fK_{p^{\infty}}}}$ s'\'ecrit de mani\`ere unique sous la forme
\[ \sum_{k=0}^{+\infty}\omega^k(\sum_{i\in I,j\in J}a_{ijk}(x)\tilde{\z}^i\tilde{q}^j) , \] o\`u $a_{ijk}(x)$ est une suite triple d'\'el\'ements de $\iota_{\dR}(\fK^+)$ avec la propri\'et\'e suivante:
 pour $k$ fix\'e et $N\in \N$, l'ensemble $\{(i,j)\in I \times J:v_p(a_{ijk}(x))+j\leq N \}$  est fini. De plus, $\tilde{\fK}^+_{p^\infty}$ est dense dans $(\B_{\dR}^+)^{\cG_{\fK_{p^{\infty}}}}$.
\end{lemma}
\begin{proof}[D\'emonstration]En composant l'application $\theta$, l'application $\rT_1: \C(\ol{\fK}^+)^{\cG_{\fK_{p^\infty}}}\ra \fK^+$ et l'application $\iota_{\dR}$, on d\'efinit une suite d'applications pour $i\in I$ et $j\in J$:
\[\theta_{ij}:(\B_{\dR}^+)^{\cG_{\fK_{p^\infty}}}\ra \iota_{\dR}(\fK^+); x\mapsto \theta_{ij}(x)=\frac{1}{\tilde{q}}\bigl(\iota_{\dR}\circ\rT_1(\theta(x)\z^{-i}q^{1-j})\bigr). \]
Soit $\eta$ l'application de $(\B_{\dR}^+)^{\cG_{\fK_{p^{\infty}}}}$ dans $(\B_{\dR}^+)^{\cG_{\fK_{p^{\infty}}}}$ d\'efinie par
$ \eta(x)=\omega^{-1}(x-\sum_{i\in I,j\in J}\theta_{ij}(x)\tilde{\z}^i\tilde{q}^j) $.
  Si $x\in(\B_{\dR}^+)^{\cG_{\fK_{p^{\infty}}}}$ et $n\in \N$, on a
$x=\omega^{n+1}\eta^{n+1}(x)+\sum_{k=0}^{n}\omega^k(\sum_{i\in I,j\in J}\theta_{ij}(\eta^k(x))\tilde{\z}^i\tilde{q}^j), $
ce qui montre que l'on peut poser $a_{ijk}(x)=\theta_{ij}(\eta^k(x))$; d'o\`u l'existence d'une telle \'ecriture. D'autre part, l'unicit\'e se d\'eduit de la construction de $\theta_{ij}$ et de l'unicit\'e d'\'ecriture d'\'el\'ement dans $\C(\ol{\fK}^+)^{\cG_{\fK_{p^{\infty}}}}$. 
Enfin, la densit\'e est une cons\'equence directe d'une telle \'ecriture.
\end{proof}
\begin{remark}
Tout \'el\'ement $x$ de $(\B_{\dR}^+)^{\cG_{\fK_{p^{\infty}}}}$ s'\'ecrit aussi de mani\`ere unique sous la forme
\[ \sum_{k=0}^{+\infty}t^k(\sum_{i\in I,j\in J}a^{'}_{ijk}(x)\tilde{\z}^i\tilde{q}^j) ),\] o\`u la suite triple $a^{'}_{ijk}(x)$ d'\'el\'ements de $\iota_{\dR}(\fK^+)$ v\'erifie la m\^eme propri\'et\'e dans le lemme.
\end{remark}

Passons au cas g\'en\'eral. Soit $M=M_0p^m$ avec $(p,M_0)=1$. Soient $l,h\in \N$ et $h\geq l$ tels que $[\fK_{M}^+\fK_{p^l}^+: \fF_M^+\fK_{p^l}]=[\fK_{Mp^{\infty}}^+:\fF_M^+\fK_{p^{\infty}}^+]=c$ et $[F_MF_{p^h}: F_{p^h}]=[F_{Mp^{\infty}}: F_{p^{\infty}}]=d$.  Supposons  que $e_1, \cdots, e_c$ une base de $\fK_{M}^+\fK_{p^l}^+$ 
sur $\fF_M^+\fK_{p^l}$ et $f_1,\cdots, f_d $ une base de $F_MF_{p^h}$ sur $F_{p^h}$. On peut d\'emontrer le lemme suivant de la m\^eme mani\`ere du lemme $\ref{DR}$. 

\begin{lemma}\label{DR1} Sous la hypoth\`ese au-dessus. Tout \'el\'ement $x$ de $(\B_{\dR}^+)^{\cG_{\fK_{Mp^{\infty}}}}$ s'\'ecrit de mani\`ere unique sous la forme:
$ \sum_{k=0}^{+\infty}\omega^k(\sum_{i=0}^{c}\sum_{j=0}^db_{ijk}(x)\iota_{\dR}(e_i)f_j) , $
avec $b_{ijk}(x)\in (\B_{\dR}^+)^{\cG_{\fK_{p^{\infty}}}}$.
\end{lemma}

\begin{prop}\label{const} Si $X$ est un ouvert compact de $\Q_p$ v\'erifiant $X+p^{-1}\Z_p=X$, il existe une unique application $\iota_{\dR}(\fK^+)$-lin\'eaire $\Res_X$ continue de $(\B_{\dR}^+)^{\cG_{\fK_{p^{\infty}}}}$ dans $(\B_{\dR}^+)^{\cG_{\fK_{p^{\infty}}}}$ telle que l'on ait $\Res_X(\tilde{\z}^x\tilde{q}^{y})=\tilde{\z}^x\tilde{q}^y$ si $x\in X,y\in X\cap \Z[\frac{1}{p}]$ et $\Res_X(\tilde{\z}^x\tilde{q}^y)=0$ sinon.
\end{prop}
\begin{proof}[D\'emonstration]
Comme $\omega=\sum_{i=0}^{p-1}\tilde{\z}_p^{i}$ est un polyn\^ome en $\tilde{\z}_p$, on voit que s'il existe une telle application, $\Res_X$ doit \^etre donn\'ee par la formule
\begin{equation}\label{Res}
\Res_X(x)=\sum_{k=0}^{+\infty}\omega^k(\sum\limits_{(i,j)\in(I\cap X)\times (J\cap X)} a_{ijk}(x)\tilde{\z}^i\tilde{q}^j).
\end{equation}
Ceci implique que $\Res_X(\A_{\inf}^{\cG_{\fK_{p^\infty}}})\subset \A_{\inf}^{\cG_{\fK_{p^\infty}}} \text{ et } \Res_X( (\ker\theta)^{k+1})=\Res_X(\omega^{k+1}(\B_{\dR}^+)^{\cG_{\fK_{p^{\infty}}}})\subset (\ker\theta)^{k+1},$
et donc que $\Res_X(U_{m,k})\subset U_{m,k} $, ce qui permet de conclure que l'application $\Res_X$ d\'efinie par la formule $(\ref{Res})$ est continue.

Il ne reste donc plus qu'\`a donner la formule explicite pour $\Res_X(\tilde{\z}^x)$ si $x\in\Q_p$. On a le fait dans (\cite{PC3} prop.4.2): $\Res_X(\tilde{\z}^x)=\tilde{\z}^x$ si $x\in X$ et $\Res_X(\tilde{\z}^x)=0$ sinon. Ceci permet de conclure le lemme en utilisant le formule explicite pour $\Res_X$.
\end{proof}

\begin{theo}Si $M\geq 1$ est un entier tel que $m=v_p(M)\geq v_p(2p)$, il existe une unique application $\bR_M:\B_{\dR}^{+}(\fK_{Mp^{\infty}}^+)\ra\tilde{\fK}_M^+$ qui est $\tilde{\fK}^+_M$-lin\'eaire et continue et telle que la restriction de $\bR_M$ \`a $\tilde{\fK}_{Mp^\infty}^+$ est donn\'ee par la formule:
\begin{equation*}
\begin{split}
\bR_M: \tilde{\fK}^+_{Mp^{\infty}}&\longrightarrow\tilde{\fK}^+_M \\
\tilde{\z}_{Mp^n}^{a}\tilde{q}_{Mp^n}^b &\mapsto
\left\{\begin{aligned}\tilde{\z}_{Mp^n}^{a}\tilde{q}_{Mp^n}^b& ;
\text{ si } p^n|a \text{ et
} p^n|b;\\
0& ;   \text{ sinon}.
\end{aligned}\right.
\end{split}
\end{equation*}
De plus, $\bR_M$ commutes \`a l'action de $\cG_{\fK}$.
\end{theo}
\begin{proof}[D\'emonstration] S'il existe une telle application, elle est unique par continuit\'e de $\bR_M$. 

Passons \`a l'existence.  On d\'ecompose $M=M_0p^m$ avec $(M_0,p)=1$. 

Soit $M_0=1$. Soit $S_m$ le sous-$\iota_{\dR}(\fK^+)$-module de $\B_{\dR}^+$ engendr\'e par les $\tilde{\z}^x\tilde{q}^y$ pour $x\in p^{-m}\Z_p, y\in p^{-m}\Z$. L'adh\'erence de $S_m$ dans $\B_{\dR}^+$ est $\tilde{\fK}^+_{p^m}$.
En appliquant la proposition $\ref{const}$ \`a $X=p^{-m}\Z_p$, on obtient une application $\tilde{\fK}_{p^m}^+$-lin\'eaire continue $\Res_{p^{-m}\Z_p}$ de $(\B_{\dR}^+)^{\cG_{\fK_{p^{\infty}}}}$ dans $(\B_{\dR}^+)^{\cG_{\fK_{p^{\infty}}}}$ telle que la restriction de $\Res_{p^{-m}\Z_p}$ \`a $\tilde{\fK}_{p^{\infty}}^+$ v\'erifie la formule voulue.
On pose $\bR_1=\Res_{p^{-m}\Z_p}$.

Passons au cas g\'en\'eral. D'apr\`es le lemme $\ref{DR1}$, si $x=\sum_{k=0}^{+\infty}\omega^k(\sum_{i=0}^{c}\sum_{j=0}^db_{ijk}(x)\iota_{\dR}(e_i)f_j) \in (\B_{\dR}^+)^{\cG_{\fK_{Mp^\infty}}} $, on pose une application de $(\B_{\dR}^+)^{\cG_{\fK_{Mp^\infty}}}$ \`a valeurs dans $\tilde{\fK}_M^+$ par la formule:
\[ \bR_M(x)=\sum_{k=0}^{+\infty}\omega^k(\sum_{i=0}^{c}\sum_{j=0}^d \bR_1(b_{ijk}(x))\iota_{\dR}(e_i)f_j) .\]
C'est une application $\tilde{\fK}_M^+$-lin\'eaire. La continuit\'e de $\bR_M$ est de celle de $\bR_1$.

Il ne reste donc plus qu'\`a v\'erifier que la restriction de $\bR_M$ \`a $\tilde{\fK}_{Mp^{\infty}}$ satisfait la formule voulue. On constate que $l=m$ v\'erifie la condition du lemme $\ref{DR1}$ et si $n\geq m$, 
on a $ 
[\fK_M^+\fK_{p^n}^+:\fK_{p^n}^+F_M]=M_0$ 
avec une base $\cB=\{q_{M_0}^i\}_{0\leq i\leq M_0-1}$. D'apr\`es le remarque $\ref{base}$, la formule voulue est un calcul direct. Ceci permet de montrer l'existence. Comme les $\{\tilde{\z}_{M}^i\tilde{q}_{M}^j\}_{i,j\in J }$ forment une base de  $\tilde{\fK}_{Mp\infty}^+$ sur $\tilde{\fK}_M^+$, on n'a pas du
 choix pour la restriction de $\bR_M$ \`a $\tilde{\fK}_{Mp^{\infty}}^+$ et donc aussi \`a  $(\B_{\dR}^+)^{\cG_{\fK_{Mp^{\infty}}}}$ par continuit\'e de $\bR_M$ et densit\'e de $\tilde{\fK}_{Mp^{\infty}}^+$ dans $(\B_{\dR}^+)^{\cG_{\fK_{Mp^{\infty}}}}$.
La commutativit\'e de $\bR_M$ et $\cG_{\fK}$ vient de la formule de l'action de $\cG_{\fK}$ sur les $\tilde{\z}_{Mp^n}^{a}\tilde{q}_{Mp^n}^b$ et de celle de $\bR_M$ sur les $\tilde{\z}_{Mp^n}^{a}\tilde{q}_{Mp^n}^b$.
\end{proof}

\subsection{Une application logarithme $\log$}\label{anneaux3}

On note $U_0(\ol{\fK}^{+})$ l'ensemble des \'el\'ements $x$ de $\R(\ol{\fK}^+)$ tels que  $v_p(x^{(0)}-1)>0$.

\begin{lemma}\label{converge}Il existe une unique application logarithme $x\mapsto \log[x]$ de $U_0(\ol{\fK}^{+})$ dans $\B_{\dR}^+(\ol{\fK}^+)$ qui v\'erifie  $\log[xy]=\log[x]+\log[y]$, et $\log[x]=\sum_{n=1}^{+\infty}\frac{(-1)^{n-1}([x]-1)^n}{n}$ .
De plus, on a $\sigma(\log[x])=\log[\sigma x]$ si $\sigma\in \cG_{\Q_p}$.
\end{lemma}
\begin{proof}[D\'emonstration]Si $x\in U_0( \ol{\fK}^{+})$, il existe $k\in\N$ tel que $kv_p(x^{(0)}-1)\geq 1$.  On constate que $([x]-1)^k-p\alpha$, o\`u $\alpha=[\frac{(x^{(0)}-1)^k}{p}]$, appartient \`a $\ker\theta$ et donc on a $([x]-1)^k=p\alpha+\omega\beta$.

Si $n\in \N$,  on a $n=km+r$ avec $0\leq r\leq k-1$. Donc on peut \'ecrire
\begin{equation*}
\begin{split}\frac{([x]-1)^n}{n}=\frac{([x]-1)^r(p\alpha+\omega\beta)^m}{n}=([x]-1)^r\sum_{i=0}^{m}(\omega\beta)^i \frac{\binom{m}{i}(p\alpha)^{m-i}}{km+r}.
\end{split}
\end{equation*}
Donc on a $\sum_{n=1}^{+\infty}\frac{(-1)^{n-1}([x]-1)^n}{n}=\sum_{r=0}^{k-1}([x]-1)^r\sum_{i=0}^{+\infty}(\omega\beta)^i \sum_{m\geq i} \frac{\binom{m}{i}(p\alpha)^{m-i}}{km+r}$.

Si $k$ est fix\'e et $m$ tend vers $+\infty$, on a $v_p(\binom{m}{i})-v_p(km+r)+m-i\geq m-i-v_p(km+r)-v_p(i)$ tend vers $+\infty$;  ce qui montre que $\sum_{m\geq i} \frac{\binom{m}{i}(p\alpha)^{m-i}}{km+r}$ converge dans $\B_{\inf}$ et la s\'erie $\sum_{n=1}^{+\infty}\frac{(-1)^{n-1}([x]-1)^n}{n}$ est donc convergente pour la topologie faible dans $\B_{\dR}^+$. La relation $\log[xy]=\log[x]+\log[y]$ se d\'eduit par un argument de s\'eries formelles.

\end{proof}

On note $\bar{q}=(q,q^{\frac{1}{p}},\cdots,)\in\R(\ol{\fK}^+)$; son repr\'esentant de Teichm\"uler est $\tilde{q}$. Il est \'evident que $\bar{q}$ n'appartient pas \`a $\tilde{\rE}^+U_0(\ol{\fK}^+)$.
On aimerait bien que l'application logarithme s'\'etend \`a $\tilde{\rE}^+U_0(\ol{\fK}^+)\times \bar{q}^{\Q}$. On a $\sigma \tilde{q}=\tilde{q}\tilde{\z}^{c_q(\sigma)}$, o\`u $c_q$ est le $1$-cocycle \`a valeurs dans $\Z_p(1)$ associ\'e \`a $q$ par la th\'eorie de Kummer, et le minimum que l'on puisse demander \`a $u_q=\log\tilde{q}$ est de v\'erifier la formule $\sigma u_q
 =u_q+c_q(\sigma)t$, quel que soit $\sigma\in\cG_{\fK}$. Mais on a le r\'esultat suivant qui dit qu'il n'existe pas un \'el\'ement $u_q$ dans $\B_{\dR}^+$ et, de plus, un tel $u_q$ est transcendant sur $\B_{\dR}^+$.

\begin{theo}\label{logq}
 $u_q$ est transcendant sur $\B_{\dR}^+$.
\end{theo}
\begin{proof}[D\'emonstration]
Supposons que $u_q$ est alg\'ebrique sur $\B_{\dR}^+$. Soit $P(X)=X^n+a_1X^{n-1}+\cdots+a_0\in\B_{\dR}^+[X]$ le polyn\^ome minimal de $u_q$ sur $\B_{\dR}^+$. Si $\sigma\in\cG_{\fK}$, alors on a
\[0=\sigma(P(X))=(X+c_q(\sigma)t)^n+\sigma(a_1)(X+c_q(\sigma)t)^{n-1}+\cdots+\sigma(a_n).\]
Comme $P(X)$ est le polyn\^ome minimal de $u_q$, on a $\sigma(P(X))=P(X)$. Ceci permet de d\'eduire $\sigma(a_1)=a_1-nc_q(\sigma)t$ et donc $\sigma(\frac{a_1}{n})=\frac{a_1}{n}-c_q(\sigma)t$. Ceci n'est pas possible car
il n'existe pas d'\'el\'ement $x$ de $\B_{\dR}^+$ tel que, si $\sigma\in\cG_{\fK}$, $\sigma x=x+c_q(\sigma)t$.

En effet, s'il existe un tel \'el\'ement $x$ dans $\B_{\dR}^+$, alors il est stable sous l'action de $\cG_{\fK_{Mp^{\infty}}}$ pour tous $M\geq 1$ et donc appartient \`a $(\B_{\dR}^+)^{\cG_{\fK_{Mp^{\infty}}}}=\B_{\dR}^+(\fK_{Mp^{\infty}}^+)$.
  Appliquons la trace de Tate normalis\'ee $\bR_M:\B_{\dR}^+(\fK_{Mp^{\infty}}^+)\ra \tilde{\fK}_{M}^+$ \`a $x$, et donc on obtient que $\bR_M(x)$ appartient \`a $\tilde{\fK}_M^+$ tel que
     \begin{equation}\label{qq}\sigma \bR_M(x)= \bR_M(x)+c_q(
    \sigma)t, \text{ pour tous } \sigma\in P_{\fK},
    \end{equation}
     Par ailleurs, la filtration sur $\B_{\dR}^+$ est stable sous l'action de $\cG_{\fK}$, 
alors on peut supposer $\bR_M(x)=a_1+a_2(\tilde{q}_M)t \mod t^2$ avec $a_i\in \iota_{\dR}(\fK_M^+)$ pour $i=1,2$. 
On d\'eduit la relation suivante de la formule $(\ref{qq})$: $\sigma\theta(a_2\chi_{\cycl}(\sigma))=\theta(a_2)+c_q(\sigma)$ pour tout $\sigma\in P_{\fK}$, 
ce qui est impossible car $\theta(a_2)\in\fK_M^+$ n'a qu'un nombre fini de conjugu\'es et $c_q(\sigma)$ prend un nombre d'infini de valeurs pour $\sigma\in U_m\subset P_{\fK}$.

\end{proof}
\begin{remark}On peut montrer de la m\^eme mani\`ere que $\log t$ est transcendant sur $\B_{\dR}^{+}$.
\end{remark}
On pose $\B_{\log}^+=\B_{\dR}^+[u_q]$ avec $\sigma(u_q)=u_q+c_q(\sigma)t$. On a bien que $U_0(\ol{\fK}^+)\times\bar{q}^{\Q}$ est stable sous l'action de $\cG_{\fK}$ car $\sigma\bar{q}=\bar{q} (1,\z_p,\cdots)$ pour $\sigma\in\cG_{\fK}$. Alors l'application $\log\circ[\cdot]$ s'\'etend \`a $U_0(\ol{\fK}^+)\times\bar{q}^{\Q}$ \`a valeurs dans $\B_{\log}^+$ telle que $\log[\bar{q}^a]=au_q$ si $a\in\Q$, et $\sigma(\log x)=\log(\sigma x)$ si $\sigma\in\cG_{\fK}$.
 On pose $\A_{\inf}^{**}=\{x=\sum_{k=0}^{+\infty}p^k[x_k]\in\A_{\inf}(\ol{\fK}^+): x_0\in U_0(\ol{\fK}^+), x_k\in\R(\ol{\fK}^+) \text{ si } k\geq 1 \}$.

\begin{prop}
\begin{itemize}
\item[(1)]Si $1+x\in \A_{\inf}^{**}$, alors la s\'erie $\sum_{n=1}^{+\infty}\frac{(-1)^{n-1}x^n}{n}$ converge dans $\B_{\dR}^+$.
\item[(2)]L'application $\log\circ[\cdot]: U_0(\ol{\fK}^+)\times\bar{q}^{\Q}\ra\B_{\log}^+$ s'\'etend en une application $\log: \A_{\inf}^{**}\times\tilde{q}^{\Q}\ra \B_{\dR}^+[u_q]$ telle que
\begin{itemize}
\item[$\bullet$]$\log([x])=\log[x]$, $\log \tilde{q}^a=au_q$ si $a\in \Q$;
\item[$\bullet$]$\log(xy)=\log(x)+\log(y)$;
\item[$\bullet$]$\log (x)=\sum_{n=1}^{+\infty}\frac{(-1)^{n-1}(x-1)^n}{n}$, si la s\'erie converge.
\end{itemize}
\end{itemize}
De plus, on a
$\sigma(\log x)=\log (\sigma x)$ si $\sigma\in\cG_{\fK}$.
\end{prop}

\begin{proof}[D\'emonstration]
$(1)$ Si $1+x\in \A_{\inf}^{**}$, alors il existe $k\in \N$, tel que $x^k-p\alpha$, o\`u $\alpha=[\frac{(\bar{x}^{(0)})^k}{p}]$, appartient \`a $\ker\theta$. Donc on a $x^k=p\alpha+\omega\beta$ avec $\beta\in\A_{\inf}$. Donc l'argument dans $\ref{converge}$ s'adapte \`a montrer la convergence.\\
$(2)$
Il suffit de montrer que $\log:\A_{\inf}^{**}\ra \B_{\log}^+$ est bien d\'efinie. Si $x_0\in U_0(\ol{\fK}^+)$, alors $[x_0]$ est inversible dans $\A_{\inf}^{**}$. Donc pout tout $x=\sum_{k=0}^{+\infty}p^k[x_k]\in\A_{\inf}^{**}$, on a $x=[x_0](1+pa)$ avec $a\in \A_{\inf}(\ol{\fK}^+)$. On constate que $\log(1+ pa)=\sum_{n\geq 1}\frac{(-p a)^{n}}{n}$ converge dans $\A_{\inf}(\ol{\fK}^+)$. Alors $\log x$ est bien d\'efini par multiplicativit\'e.
\end{proof} 

%% file: cohomologie.tex
\section{Cohomologie des repr\'esentations du groupe $P_{\fK_M}$}

Soit $M$ un entier $\geq 1$ et soit $m=v_p(M)\geq v_p(2p)$.
Le corps $\fK_{Mp^\infty}$ est une extension galoisienne de $\fK_M$ dont le groupe de Galois
\[ P_{\fK_M}\cong \{(\begin{smallmatrix}a&b\\ c&d\end{smallmatrix})\in \GL_2(\Z_p): a=1, c=0, b\in p^m\Z_p, d\in 1+p^m\Z_p\}=:\rP_m\]
est un groupe analytique $p$-adique de rang $2$. Si $g$  est  l'\'el\'ement de $P_{\fK_M}$ correspondant \`a une matrice $(\begin{smallmatrix}1&b\\ 0&d\end{smallmatrix})$, l'action de $g$ sur $\z_M^iq_M^j$ avec $(i,j)\in J\times J$ est donn\'ee par la formule:
\[g(\z_M^iq_M^j)=\z_M^iq_M^j\z_{p^m}^{i(d-1)+jb}.\]
Si $u,v\in p^m\Z_p$, on pose
$(u,v)$ l'\' el\'ement
$\bigl(\begin{smallmatrix}1&u\\0&e^v\end{smallmatrix}\bigr)$ de
$\rP_m$. La loi de groupe s'\'ecrit alors sous la forme:
\[(u_1,v_1)(u_2,v_2)=(e^{v_2}u_1+u_2,v_1+v_2).\]
Soient $U_m$ et $\Gamma_m$ les sous-groupes de $\rP_m$ eng\'endr\'es par $u_m=(p^m,0)$ et $\gamma_m=(0,p^m)$ respectivement. Si $(u,0)\in U_m$ et $(0,v)\in\Gamma_m$, on a $(0,v)(u,0)(0,v)^{-1}=(e^{-v}u,0)\in U_m$. Donc $U_m$ est distingu\'e dans $\rP_m$, et on a $\Gamma_m\cong \rP_m/U_m$. Ces deux sous-groupes $U_m$ et $\Gamma_m$ sont isomorphes \`a $\Z_p$; ils sont donc de dimension cohomologique $1$. Cela implique $\rP_m$ est de dimension cohomologique $\leq 2$.

Si $G$ est un groupe topologique, si $V$ est une $\Q_p$-repr\'esentation de $G$, on note le groupe de cohomologie $\rH^i(G,V)$ la cohomologie continue de $G$ \`a valeurs dans la repr\'esentation $V$.
Si $G$ est procyclique, si $g_0$ est un g\'en\'erateur de $G$, on a :
\begin{equation}\label{iso}
\rH^1(G,V)\cong V/(g_0-1),
\end{equation}
o\`u un $1$-cocycle $(g\mapsto c_g )$ est envoy\'e sur l'image de $c_{g_0}$ dans $V/(g_0-1)$.

\begin{lemma}L'action $\circ$ de $\Gamma_m$ sur $V/((p^m,0)-1)$ induite par celle sur $\rH^1(U_m,V)$, est obtenue en tordant l'action originale $*$ sur $V$ par le caract\`ere $\gamma_m^a\mapsto e^{-ap^m}$, o\`u $\gamma_m=(0,p^m)$. Plus pr\'ecisement, si $x\in V/(u_m-1)$ avec $u_m=(p^m,0)$ et si  $(0,v)\in\Gamma_m$ ,
\[ x\circ (0,v)=e^{-v}x*(0,v).\]
\end{lemma}
\begin{proof}[D\'emonstration]On d\'emontre le lemme par le calcul suivant:
si $(0,v)\in\Gamma_m $,  on a
\begin{equation*}
\begin{split}
(c*(0,v))_{u_m}=c_{(0,v)u_m(0,v)^{-1}}*(0,v)=c_{(e^{-v}p^m,0)}*(0,v)=c_{u_m^{e^{-v}}}*(0,v);
\end{split}
\end{equation*}
de plus,  la formule de $1$-cocycle et  l'identit\'e $c_{(p^m,0)}*(p^m,0)=c_{(p^m,0)}$ dans $V/((p^m,0)-1)$, nous donnent $c_{u_m^a}=ac_{u_m}$ pour tout $a\in\Z_p$ et donc
\begin{equation*}
\begin{split}
c_{(p^m,0)^{e^{-v}}}*(0,v)-c_{(p^m,0)}=(e^{-v}c_{(p^m,0)})*(0,v)-c_{(p^m,0)}=c_{(p^m,0)}*(e^{-v}(0,v)-1),
\end{split}
\end{equation*}
o\`u $e^{-v}$ agit sur  $c_{(p^m,0)}$ par multiplication par $e^{-v}$.
\end{proof}

\begin{lemma}Soit $V$ une $\Q_p$-repr\'esentation de $\rP_m$; alors on a
\[\rH^2(\rP_m,V)\cong V/((p^m,0)-1, (0,p^m)-e^{p^m}).\]
\end{lemma}
\begin{proof}[D\'emonstration]

On a une suite exacte de groupes $1\ra U_m\ra \rP_m\ra \Gamma_m\ra 1$ et la dimension cohomologique de $\rP_m$ est $\leq 2$. Donc la suite spectrale de
Hochschild-Serre nous fournit un isomorphisme: $\rH^2(\rP_m, V)\cong\rH^1(\Gamma_m,\rH^1(U_m,V)),$
o\`u l'action de $\Gamma_m$ sur $\rH^1(U_m,V)$ est
donn\'ee par $(u\mapsto c_u)\mapsto (u\mapsto (c*\gamma)_u:=c_{\gamma
u\gamma^{-1}}*\gamma )$.

Les $U_m,\Gamma_m$ sont des groupes procycliques avec les g\'en\'erateurs $(p^m,0)$ et $(0,p^m)$ respectivement. Soit $(u\mapsto c_u)$ un $1$-cocycle repr\'esentant un \'el\'ement $c$ de $\rH^1(U_m,V)$. Alors, par l'isomorphisme (\ref{iso}), $c$ a pour l'image $c_{(p^m,0)}$ dans $V/((p^m,0)-1)$.
D'apr\`es le lemme ci-dessus,  l'action $\circ$ de $\Gamma_m$ sur $V/(u_m-1)$ induite par celle sur  $\rH^1(U_m, V)$ est donn\'ee par la formule:
$ x\circ (0,v)=e^{-v}x*(0,v).$ Donc on conclut par les isomorphismes suivants:
\[\rH^1(\Gamma_m,\rH^1(U_m,V))\cong (V/(u_m-1))/(\gamma_m-1)\cong V/(u_m-1,e^{-p^m}\gamma_m-1).\]
\end{proof}

\subsection{Cohomologie des repr\'esentations analytiques du groupe $\mathrm{P}_m$}\label{tech}
\begin{defn} Si $n\geq 1$, on note $\mathrm{u}=(u_1,\cdots, u_n)\in \C_p^n$.
On note $D(\mathrm{u}, m)$ la boule ferm\'ee
\[\{\mathrm{x}=(x_1,\cdots,x_n)\in\C_p^n, v_p(\mathrm{x}-\mathrm{u})=\inf_{1\leq i\leq n}v_p(x_i-u_i)\geq m  \}.\]
 Une fonction $f$ sur $D(\mathrm{u},m)$ \`a valeurs dans $\C_p$ est $\Q_p$-analytique s'il existe $\{a_{\mathrm{i}}(f,\mathrm{u})\in\Q_p\}_{(\mathrm{i}\in\N^n)}$ tels que $\lim\limits_{i_1+\cdots+i_n\ra+\infty}v_p(a_{\mathrm{i}}(f,\mathrm{u}))+(\sum_{j=1}^ni_j)m=+\infty$ et $f(\mathrm{x})=\sum_{\mathrm{i}\in \N^n}a_{\mathrm{i}}(f,\mathrm{u})(\mathrm{x}-\mathrm{u})^\mathrm{i}$ quel que soit $\mathrm{x}\in D(\mathrm{u},m)$.
\end{defn}

On note $\cC^{\an}(D(0,m),\Q_p)$ l'anneau des fonctions $\Q_p$-analytiques sur $D(\mathrm{u}=0,m)$. On d\'efinit une valuation $v_0$ sur $\cC^{\an}(D(0,m),\Q_p)$ \`a valeurs dans $\Z$:
\[v_0: \cC^{\an}(D(0,m),\Q_p)\ra \Z; v_0(\sum_{\mathrm{i}\in \N^n}a_{\mathrm{i}}(f)\mathrm{x}^\mathrm{i})=\inf\limits_{\mathrm{i}\in\N^n }\{\sum_{j=1}^{n}i_j|a_{\mathrm{i}}(f)\neq 0 \}.\]
Cette valuation induit une filtration sur $\cC^{\an}(D(0,m),\Q_p)$: pour tous $k\in\Z$,
\[\Fil^k\cC^{\an}(D(0,m),\Q_p)=\{f\in \cC^{\an}(D(0,m),\Q_p): v_0(f)\geq k \} .\]

Si $m\geq v_p(2p)$ est un entier, on pose $\mathrm{P}_m=\{\bigl(\begin{smallmatrix}1&b\\0&d
\end{smallmatrix}\bigr)|b\in p^m\Z_p, d\in1+p^m\Z_p\}$. C'est un groupe analytique $p$-adique compact. Si $u,v\in p^m\Z_p$, on note
$(u,v)$ l'\' el\'ement
$\bigl(\begin{smallmatrix}1&u\\0&e^v\end{smallmatrix}\bigr)$ de
$\rP_m$. La loi de groupe s'\'ecrit alors sous la forme:
$(u_1,v_1)(u_2,v_2)=(e^{v_2}u_1+u_2,v_1+v_2).$

\begin{defn}\label{defn1}
\begin{itemize}
\item[(1)]Une fonction $f$ sur $\rP_m$ \`a valeurs dans $\Q_p$ est $\Q_p$-analytique s'il existe une fonction $\Q_p$-analytique $\tilde{f}$ sur $D(0,m)$, telle que, $\forall u,v\in p^m\Z_p$, $f((\begin{smallmatrix}1&u\\0&e^v
\end{smallmatrix}))=\tilde{f}(u,v).$
\item[(2)]Une repr\'esentation analytique $V$ de $\rP_m$ est un $\Q_p$ espace vectoriel de dimension finie muni d'une action continue de $\rP_m$, et les coordonn\'ees de la matrice de $(u,v)\in \rP_m$ dans une base de $V$ soient $\Q_p$-analytiques.
\end{itemize}
\end{defn}

Soit $V$ une repr\'esentation analytique de $\rP_m$. Comme $V$ est une repr\'esentation analytique, on dispose des op\'erateurs $\partial_i:
V\ra V$, pour $i=1,2$,  d\'efinis par:
\begin{equation}\label{diff}
x*(u,v)=x+u\partial_1 x+v\partial_2 x+O((u,v)^2),
\end{equation}
o\`u $O((u,v)^2)$ est une fonction analytique sur $\rP_m$ de valuation $\geq 2$.
\begin{lemma}Ces op\'erateurs ont des propri\'et\'es de d\'erivations: si
$x_1\in V_1, x_2\in V_2$, o\`u $V_1,V_2$ sont des repr\'esentations
analytiques de $\rP_m$, et si $i=1,2$, on a
\begin{equation*}
\partial_i(x_1\otimes x_2)=(\partial_i
x_1)\otimes x_2+x_1\otimes\partial_i x_2.
\end{equation*}
\end{lemma}
\begin{proof}[D\'emonstration]
Soit $(u,v)\in\rP_m$ et soient $x_1\in V_1$ et $x_2\in V_2$. De la d\'efinition de $\partial_i$, on a
\[(x_1\otimes x_2)*(u,v)=x_1\otimes x_2+ u \partial_1(x_1\otimes x_2)+v\partial_2(x_1\otimes x_2)+O((u,v)^2).\]
D'autre part, on a
\begin{equation*}
\begin{split}
&(x_1\otimes x_2)*(u,v)\\
=&x_1\otimes x_2+u((\partial_1x_1)\otimes x_2+x_1\otimes(\partial_1x_2))+v((\partial_2x_1)\otimes x_2+x_1\otimes(\partial_2x_2))+O((u,v)^2).
\end{split}
\end{equation*}
On d\'eduit les propri\'et\'es de d\'erivations de $\partial_i$ pour $i=1,2$ en comparant les deux formules ci-dessus.
\end{proof}

Soit $\gamma\in\rP_m$; l'image de la fonction analytique $\alpha_{\gamma}:\Z_p\ra\rP_m, \alpha_{\gamma}(x)=\gamma^x$ est un sous-groupe \`a un param\`etre. Alors on peut d\'efinir une d\'erivation $\partial_{\gamma}:V\ra V$   par rapport \`a $\alpha_{\gamma}$ par la formule: $\partial_{\gamma}(x)=\lim_{n\ra\infty}\frac{x*\gamma^{p^n}-x}{p^{n}} .$

\begin{lemma}On a les relations suivantes: $\partial_1=p^{-m}\partial_{u_m}, \partial_2-1=p^{-m}\partial_{e^{-p^m}\gamma_m}.$
\end{lemma}
\begin{proof}[D\'emonstration]Soit $x\in V$.
Par d\'efinition, on a
\begin{equation*}
\begin{split}
\partial_{u_m}x=\lim_{n\ra\infty}\frac{x*(p^m,0)^{p^n}-x}{p^n}
=\lim_{n\ra\infty}\frac{x*(p^{m+n},0)-x}{p^n}=p^{m}\partial_1x.
\end{split}
\end{equation*}
Alors $\partial_1=p^{-m}\partial_{u_m}$. De la m\^eme mani\`ere, on a
$\partial_2=p^{-m}\partial_{\gamma_m}$ et $\partial_2-1=p^{-m}\partial_{e^{-p^m}\gamma_m}$ .

\end{proof}

\begin{remark}\label{coefficient}Si $V$ est une repr\'esentation analytique de $\rP_m$ munie d'un $\Z_p$-r\'eseau\footnote{ Il existe tel r\'eseau si $m$ assez grand. } $T$ tel que $T$ est stable sous l'action de $\rP_m$, et si $x\in T$, alors du lemme pr\'ec\'edent, on d\'eduit
\[x*(u,v)=\sum_{n=0}^{+\infty}u^iv^j\frac{\partial_1^i\partial_2^jx}{i!j!}\]
avec $\frac{\partial_1^i\partial_2^jx}{i!j!}\in p^{-s(i+j)}T\subset p^{-m(i+j)}T$ pour certain $s<m$.
\end{remark}

Si $x\in V$ est dans le noyau de l'op\'erateur $u_m-1$ ( resp. $e^{-p^m}\gamma_m-1$ ) sur $V$, alors $x$ est dans le noyau de l'op\'erateur $\partial_1=p^{-m}\lim\limits_{n\ra+\infty}\frac{u_m^{p^n}-1}{p^n}$ ( resp. $\partial_2-1$ ) sur $V$. Ceci nous fournit une application surjective naturelle :
\[\phi:V/(\partial_1,\partial_2-1)\ra V/((p^m,0)-1,(0,p^m)-e^{p^m}).\]
\begin{lemma} $\phi$ est un isomorphisme.
\end{lemma}
\begin{proof}[D\'emonstration]On est ramen\'e \`a montrer que $\ker \partial_1=\ker(u_m-1)$ ( resp. $\ker(\partial_2-1)=\ker(e^{-p^m}\gamma_m-1)$ ). On a montr\'e l'inclusion $\ker(u_m-1)\subset \ker\partial_1$ ( resp. $\ker(e^{-p^m}\gamma_m-1)\subset\ker(\partial_2-1)$ ). Alors il reste \`a montrer l'inclusion inverse.

Soit $x\in\ker \partial_1$.  Comme on a $\partial_1 u_m=u_m\partial_1$, l'espace $\ker \partial_1$ est stable sous l'action de $u_m$. De plus, il existe une base $S$ de $\ker \partial_1$ telle que la matrice de $u_m$ peut se mettre sous la forme de Jordan. Soit $\lambda$ une valeurs propre de $u_m|_{\ker\partial_1}$.
Comme $V$ est une repr\'esentation analytique de $\rP_m$,  $\lambda^x$ est une fonction analytique sur $\Z_p$. Comme la matrice de $\partial_1$ dans cette base est $0$, on a $\log \lambda=0$. Ceci implique que $\lambda$ est une racine de l'unit\'e $\z$. De plus, $\z^{x}$ est une fonction analytique en variable $x$ sur $\Z_p$  si et seulement si  $\z=1$. Donc la matrice de $u_m$ dans la base $S$ est unipotente et on peut supposer qu'elle est de la forme $I+A$ avec $A$ une matrice nilpotente.

Comme $\partial_1=0$ sur $\ker \partial_1$, on a $\log (I+A)=0$. D'autre part, on a $\log (I+A)=A (\sum_{n=0}^{+\infty}\frac{(-A)^n}{(n+1)!})$, o\`u $\sum_{n=0}^{+\infty}\frac{(-A)^n}{(n+1)!}$ est une somme finie. Comme $\sum_{n=0}^{+\infty}\frac{(-A)^n}{(n+1)!}$ est inversible, on obtient $A=0$. Ceci \'equivaut \`a la condition $\ker\partial_1\subset \ker (u_m-1)$.
De la m\^eme mani\`ere, on obtient $\ker(\partial_2-1)\subset\ker(e^{-p^m}\gamma_m-1).$
\end{proof}
Par cons\'equent, on obtient le corollaire suivant:
\begin{coro}\label{iso1}
 On a un isomorphisme naturel
\begin{equation}
\rH^2(\rP_m,V)\cong V/(\partial_1,\partial_2-1).
\end{equation}
\end{coro}

Rappelons que
le groupe de cohomologie $\rH^2(\rP_m,V)$ est la cohomologie continue d'un groupe $p$-adique  \`a valeurs dans la repr\'esentation analytique $V$. On pose $\cC^0(\rP_m,V)=V$ et note $\cC^n(\rP_m,V)$ le groupe des homomorphismes continus de $\rP_m^n$ \`a valeurs dans $V$. On peut le calculer par le complexe nonhomog\`ene des cochaines continues
\[
\xymatrix{
\cC^{\bullet}: 0\ar[r]& \cC^0(\rP_m,V)\ar[r]^-{d_0}& \cdots\ar[r]^-{d_{n-2}}&\cC^{n-1}(\rP_m,V)\ar[r]^-{d_{n-1}} & \cC^n(\rP_m,V)\ar[r] & \cdots,}
\]
avec les diff\'erentielles $d_n$ donn\'ees par les formules
\begin{equation*}
\begin{split}
(d_nf)(\gamma_1,\cdots,\gamma_{n+1})=&f(\gamma_2,\cdots,\gamma_n)*\gamma_1+\sum_{i=1}^{n}(-1)^if(\gamma_1,\cdots,\gamma_{i-1}, \gamma_i\gamma_{i+1},\cdots,\gamma_{n+1})\\
&+(-1)^{n+1}f(\gamma_1,\cdots,\gamma_n).
\end{split}
\end{equation*}
Comme $V$ est une repr\'esentation analytique, le complexe $\cC^{\bullet}$ des cochaines continues contient un sous-complexe des cochaines analytiques avec les m\^emes diff\'erentielles:
\[\xymatrix{\cC^{\an,\bullet}: 0\ar[r] & \cC^{\an,0}(\rP_m,V)\ar[r]^-{d_0}&\cdots\ar[r]^-{d_{n-2}}&\cC^{\an,n-1}(\rP_m,V)\ar[r]^-{d_{n-1}}& \cC^{\an,n}(\rP_m,V)\ar[r]&\cdots,} \]
o\`u $\cC^{\an,0}(\rP_m,V)=V$ et $\cC^{\an,n}(\rP_m,V)$ est le sous-module des fonctions analytiques sur $\rP_m^n$ \`a valeurs dans $V$ de $\cC^n(\rP_m,V)$. En particulier, tout \'el\'ement de $\cC^{\an,2}(\rP_m,V)$ peut s'\'ecrire sous la forme: $c_{(u,v),(x,y)}=\sum\limits_{i,j,k,l\geq 0}c_{i,j,k,l}u^iv^jx^ky^l$ avec $c_{i,j,l,k}\in V$.

La filtration sur $\cC^{\an,n}(D(0,m),\Q_p)$ d\'efinie par la valuation $v_0$ induit celle sur le complexe $\cC^{\an,\bullet}$. Quel que soit $k\geq 1$, on a $\Fil^0\cC^{\an,\bullet}/\Fil^k\cC^{\an,\bullet}\cong \cP^{\leq k,\bullet}$, o\`u $\cP^{\leq k,\bullet}$ est le complexe
\[\xymatrix{\cP^{\leq k,\bullet}: 0\ar[r]& \cP_{\leq k}^{0}(\rP_m,V)\ar[r]^-{d_0}& \cdots\ar[r]^-{d_{n-2}} & \cP_{\leq k}^{n-1}(\rP_m,V)\ar[r]^-{d_{n-1}} &\cP_{\leq k}^{n}(\rP_m,V)\ar[r]&\cdots,} \]
o\`u $\cP_{\leq k}^{i}$ l'ensemble des polyn\^omes de degr\'e $\leq k$ sur $\rP_m^i$ \`a coefficients dans $V$ et les differentielles sont ceux du complexe $\cC^{\an,\bullet}$ modulo les termes de valuation $k+1$.

Consid\'erons le compl\'et\'e du complexe $\cC^{\an,\bullet}$ pour sa filtration au-dessus, on obtient un complexe suivant:
\[\xymatrix{\cS^{\bullet}: 0\ar[r]& \cS^{0}(\rP_m,V)\ar[r]^-{d_0}&\cdots\ar[r]^-{d_{n-2}} &\cS^{n-1}(\rP_m,V)\ar[r]^-{d_{n-1}} &\cS^{n}(\rP_m,V)\ar[r]&\cdots,}\]
o\`u $\cS^i(\rP_m,V)$ l'ensemble des s\'eries formelles sur $\rP^i_m$ \`a coefficients dans $V$ et les differentielles sont ceux du complexe $\cC^{\an,\bullet}$. Alors $\cC^{\an,\bullet}$ est un sous-complexe du complexe $ \cS^{\bullet}$.

On note $\rH^{\an,i}(\rP_m,V)$ le groupe de cohomologie calcul\'e par le complexe $\cC^{\an,\bullet}$. La proposition suivante montre que le sous-complexe $\cC^{\an,\bullet}$ calcule le groupe de cohomologie $\rH^2(\rP_m,V)$:
\begin{prop}\label{analytic}Soit $V$ une repr\'esentation analytique de $\rP_m$. Alors
\begin{itemize}
\item[$(i)$] tout \'el\'ement de $ \rH^2(\rP_m,V) $ est repr\'esentable par un  $2$-cocycle analytique;
\item[$(ii)$] l'image d'un $2$-cocycle analytique,
\[((u,v),(x,y))\ra c_{(u,v),(x,y)}=\sum_{i+j+k+l\geq 2}c_{i,j,k,l}u^iv^jx^ky^l,\]
sous l'isomorphisme (\ref{iso1}) est aussi celle de
$\delta^{(2)}(c_{(u,v),(x,y)})=c_{1,0,0,1}-c_{0,1,1,0}$ dans
$V/(\partial_1,\partial_2-1)$.
\end{itemize}
\end{prop}
Par l'isomorphisme (\ref{iso1}), il suffit de montrer (ii) et montrer que l'application
\[\delta^{(2)}: \{ 2\text{-cocycle analytique}\}\ra V\] induit une surjection de
$\rH^{\an,2}(\rP_m,V)\ra V/(\partial_1,\partial_2-1).$

La d\'emonstration de la proposition ($\ref{analytic}$) se s\'epare en trois \'etapes:\\
$(1)$ On calcule la cohomologie  $\rH^2(\cP^{\leq k,\bullet})$. On est ramen\'e \`a \'etudier la cohomologie des complexes quotient $\Fil^i\cC^{\an,\bullet}/\Fil^{i+1}\cC^{\an,\bullet}$:
    \[\xymatrix{\cP^{i,\bullet}:0\ar[r] & \cP_{i}^{0}(\rP_m,V)\ar[r]^-{d_0}&\cdots\ar[r]^-{d_{n-2}} & \cP_{i}^{n-1}(\rP_m,V)\ar[r]^-{d_{n-1}} &\cP_{i}^{n}(\rP_m,V)\ar[r]&\cdots, }\] o\`u $\cP_i^{n}$ est l'ensemble des polyn\^omes homog\`enes de degr\'e $i$ sur $\rP_m^n$ \`a coefficients dans $V$. C'est le sujet du lemme \ref{coho}. \\
$(2)$ On montre que, soit $c_{(u,v),(x,y)}: ((u,v),(x,y))\ra c_{(u,v),(x,y)}=\sum_{i+j+k+l> 0}c_{i,j,k,l}u^iv^jx^ky^l$ un $2$-cocycle
du complexe $\cS^{\bullet}$, alors $c_{(u,v),(x,y)}-(c_{1,0,0,1}-c_{0,1,1,0})uy$ est un $2$-cobord du complexe $\cS^{\bullet}$ ( c.f. lemme \ref{coho1}). En effet, l'\'etape $(1)$ nous permet de montrer le r\'esultat \`a la main.\\
$(3)$ On montre que, quel que soit $c_{(u,v),(x,y)}$ un $2$-cocycle analytique dans $\cC^{\an,\bullet}$
\[((u,v),(x,y))\ra c_{(u,v),(x,y)}=\sum_{i+j+k+l> 0}c_{i,j,k,l}u^iv^jx^ky^l,\]
alors $c_{(u,v),(x,y)}-(c_{1,0,0,1}-c_{0,1,1,0})uy$ est un cobord analytique. En effet, on peut contr\^oler le $2$-cobord  dans l'\'etape $(2)$, ce qui permet de montrer qu'il est un $2$-cobord analytique.

Tout d'abord, on fait une remarque sur la relation de $2$-cocycle du complexe $\cP^{n,\bullet}$.
 Soit $c_{(u,v),(x,y)}$ un $2$-cocycle du complexe $\cP^{n,\bullet}$, par d\'efinition il
v\'erifie la relation:
\begin{equation*}
c_{(x,y),(\alpha,\beta)}*(u,v)-c_{(u,v)(x,y),(\alpha,\beta)}+c_{(u,v),(x,y)(\alpha,\beta)}-c_{(u,v),(x,y)}\equiv 0 \mod \Fil^{n+1}\cC^{\an,3}.
\end{equation*}
 La loi de groupe de $\rP_m$ devient simplement: $c_{(u,v)(x,y),(\alpha,\beta)}=c_{(u+v,v+y), (\alpha,\beta)}$.
De plus, soit $(u_i,v_i)\in \rP_m$ pour $i=1,2,3$; si $f((u_2,v_2), (u_3,v_3))$ est un polyn\^ome homog\`ene de degr\'e $n$ \`a valeurs dans $V$, alors il est de la forme $\sum\limits_{i+j+k+l=n}a_{i,j,k,l}u_2^iv_2^ju_3^kv_3^l$ avec $a_{i,j,k,l}\in V$ et $(u_1,v_1)$ agit sur $f$ \`a travers l'action de $\rP_m$ sur $V$:
\begin{equation}\label{agit}(\sum_{i+j+k+l=n}a_{i,j,k,l}u_2^iv_2^ju_3^kv_3^l)*(u_1,v_1)
=\sum_{i+j+k+l=n}a_{i,j,k,l}*(u_1,v_1)u_2^iv_2^ju_3^kv_3^l.
\end{equation}
 Comme l'action de $(u_1,v_1)$ sur les coefficients $a_{i,j,k,l}$ se factorise \`a travers les op\'erateurs
$\partial_1$ et $\partial_2$ dans la formule (\ref{diff}), l'action de $(u_1,v_1)$ est triviale dans  la formule (\ref{agit}).
Donc, la relation de $2$-cocycle ainsi que celle de $2$-cobord dans $\cP^{n,\bullet}$ se simplifient
\begin{equation}\label{2coc}
\begin{split}
&c_{(x,y),(\alpha,\beta)}-c_{(u+x,v+y),(\alpha,\beta)}
+c_{(u,v),(x+\alpha,y+\beta)}-c_{(u,v),(x,y)}=0;\\
&(dQ)_{(u,v),(x,y)}= Q_{(u,v)}-Q_{(u+x,v+y)}+Q_{(x,y)},
\text{ o\`u } Q\in\cP_n^{1}(\rP_m,V).
\end{split}
\end{equation}

\begin{defn}
On dira qu'un polyn\^ome $P((u,v),(x,y))=\sum\limits_{i+j+k+l=n}c_{i,j,k,l}u^iv^jx^ky^l\in \cP_n^2$ est de valuation faible $m$ si $m=\min\{i+j,k+l: c_{i,j,k,l}\neq 0\}$.
\end{defn}

Si $n\geq 2$, soit $c_{(u,v),(x,y)}$ un $2$-cocycle du complexe $\cP^{n,\bullet}$. On note la d\'erivation en la $i$-i\`eme variable par $\rD_i$ pour $1\leq i\leq 4$.

\begin{lemma}\label{coho}
\begin{itemize}
\item[(1)]Si $n=1$, alors tout $2$-cocycle de $\cP^{1,\bullet}$ est nul.
\item[(2)]Si $n=2$,  tout $2$-cocycle  du complexe $\cP^{2,\bullet}$ peut s'\'ecrire sous la forme suivante:
\[c_{(u,v),(x,y)}=\sum_{i+j+k+l=2}c_{i,j,k,l}u^iv^jx^ky^l,\]
o\`u $c_{ijkl}$ sont dans $V$.
Alors, son image dans le groupe de cohomologie $\rH^2(\cP^{2,\bullet})$ est aussi celle de $(c_{1,0,0,1}-c_{0,1,1,0})uy$.
\item[(3)]Si $n\geq 3$, le groupe de cohomologie $\rH^2(\cP^{n,\bullet})$ est nul. En particulier, si \[c_{(u,v),(x,y)}=\sum_{i+j+k+l=n}c_{i,j,k,l}u^iv^jx^ky^l\] est un $2$-cocycle du complexe $\cP^{n,\bullet}$, alors le polyn\^ome homog\`ene de degr\'e $n$
\[Q(X,Y)=\frac{1}{n}(c_{1,0,n-1,0}X^{n}+c_{0,n-1,0,1}Y^n)+\sum\limits_{k=0}^{n-2}\frac{1}{k+1}c_{n-k-1,k,0,1}X^{n-k-1}Y^{k+1}\]
v\'erifie la relation $c=-dQ$.
\end{itemize}
\end{lemma}
\begin{proof}[D\'emonstration]
$(1)$ Si $n=1$, un $2$-cocycle $c_{(u,v),(x,y)}$ de $\cP^{1,\bullet}$ s'\'ecrit sous la forme $a_1u+a_2v+a_3x+a_4y$  avec $a_i\in V$ pour $1\leq i\leq 4$.
 Alors la relation de $2$-cocycle (\ref{2coc}) se traduit en la relation:
 $a_1x+a_2y+a_3\alpha+a_4\beta=0,$
 ce qui implique $a_1=a_2=a_3=a_4=0$.\\
$(2)$
Si $n=2$, soit $c_{(u,v),(x,y)}$ un $2$-cocycle du complexe $\cP^{2,\bullet}$. Alors, $c_{(u,v),(x,y)}$ peut se ranger par la valuation faible sous la forme:
\[c_{(u,v),(x,y)}=P_{01}(u,v)+P_{02}(x,y)+P_1(u,v,x,y),\]
o\`u $P_{01}(u,v)=c_uu^2+c_vv^2+c_{uv}uv$, $P_{02}(x,y)=c_xx^2+c_yy^2+c_{xy}xy$ sont de valuation faible $0$, et $P_1(u,v,x,y)=\sum\limits_{\substack{i+j+k+l=2\\
(i,j)\neq (0,0), (k,l)\neq (0,0)}}c_{i,j,k,l}u^iv^jx^ky^l$ est de valuation faible $1$.

Pour les termes $c_{1,0,1,0}ux+c_{0,1,0,1}vy$ et $c_{1,0,0,1}uy+c_{0,1,1,0}vx$, on pose
$Q_{21}(x,y)=\frac{1}{2}(c_{1,0,1,0}x^2+c_{0,1,0,1}y^2)$ et $Q_{22}(x,y)=c_{0,1,1,0}xy$ dans $\cP_2^1$ respectivement. Donc on a
\begin{equation*}
\begin{split}
Q_{21}(u+x,v+y)-Q_{21}(u,v)-Q_{21}(x,y)&=c_{1,0,1,0}ux+c_{0,1,0,1}vy; \\
Q_{22}(u+x,v+y)-Q_{22}(x,y)-Q_{22}(u,v)&=c_{0,1,1,0}(uy+vx).
\end{split}
\end{equation*}
On pose $Q_2=Q_{21}+Q_{22}$ et alors
\[c_{(u,v),(x,y)}=P_{01}(u,v)+P_{02}(x,y)-dQ_2+(c_{1,0,0,1}-c_{0,1,1,0})uy.\]
Par ailleurs, $(c_{1,0,0,1}-c_{0,1,1,0})uy$ est un $2$-cocycle. On peut alors supposer que le $2$-cocycle $c_{(u,v),(x,y)}$ est de la forme
$P_{01}(u,v)+P_{02}(x,y).$
La relation de $2$-cocycle $(\ref{2coc})$ nous fournit la relation suivante:
\begin{equation*}
\begin{split}
P_{01}(x,y)-P_{01}(u+x,v+y)+P_{02}(x+\alpha,y+\beta)-P_{02}(x,y)=0.
\end{split}
\end{equation*}
Si on \'evalue en $(x,y)=(0,0)$, alors on a $-P_{01}(u,v)+P_{02}(\alpha,\beta)=0$,
cela implique que $P_{01}=P_{02}=0$.\\
$(3)$ Si $n\geq 3$,
    La relation de $2$-cocycle $(\ref{2coc})$ nous dit que la fonction analytique en six variable
   $g(u,v,x,y,\alpha,\beta)=c_{(x,y),(\alpha,\beta)}-c_{(u+x,v+y),(\alpha,\beta)}+c_{(u,v),(x+\alpha,y+\beta)}-c_{(u,v),(x,y)}$ est identiquement nulle. En d\'eveloppant cette fonction $g$ en variable $(u,v)$, on obtient:
    \begin{equation*}
    \begin{split}
    &-u(\rD_1c)_{(x,y),(\alpha,\beta)}-v(\rD_2c)_{(x,y),(\alpha,\beta)}+u(\rD_1c)_{(0,0),(x+\alpha,y+\beta)}
   +v(\rD_2c)_{(0,0),(x+\alpha,y+\beta)} \\
    &+c_{(0,0),(x+\alpha,y+\beta)}-c_{(0,0),(x,y)}-u(\rD_1c)_{(0,0),(x,y)}-v(\rD_2c)_{(0,0),(x,y)}+O((u,v)^2), \end{split}
    \end{equation*}
    o\`u $O((u,v)^2)$ note les termes de valuation faible $\geq 2$ de variable $(u,v)$.
    Alors, pour $i=1,2$, on a
    \begin{equation}\label{cocycle1}
    (\rD_ic)_{(x,y),(\alpha,\beta)}=(\rD_ic)_{(0,0),(x+\alpha,y+\beta)}-(\rD_ic)_{(0,0), (x,y)}.
    \end{equation}
    De la m\^eme mani\`ere, si on d\'eveloppe la fonction $g$ dans la variable $(\alpha,\beta)$, on obtient les relations suivantes pour $j=3,4$:
  $ (\rD_jc)_{(u,v),(x,y)}=(\rD_jc)_{(u+x,v+y),(0,0)}
    -(\rD_jc)_{(x,y),(0,0)}.$

    On pose deux polyn\^omes homog\`enes de degr\'e $n-1\geq 2$ de variable $(X,Y)$:
    \begin{equation}\label{cobord}
    \begin{split}Q_X(X,Y)&=(\rD_1c)_{(0,0),(X,Y)}=\sum\limits_{1+k+l=n}c_{1,0,k,l}X^kY^l;\\ Q_Y(X,Y)&=(\rD_4c)_{(X,Y),(0,0)}=\sum\limits_{1+i+j=n}c_{i,j,0,1}X^iY^j.
    \end{split}
    \end{equation}

    Ensuite, on d\'eveloppe la fonction $g$ dans la variable $(u,v,\alpha,\beta)$ et on obtient
    \begin{equation*}
    \begin{split}
    &-u\bigl((\rD_1c)_{(x,y),(0,0)}+\alpha(\rD_3\rD_1c)_{(x,y),(0,0)}+\beta(\rD_4\rD_1c)_{(x,y),(0,0)}\bigr)\\ &-v\bigl((\rD_2c)_{(x,y),(0,0)}+\alpha(\rD_3\rD_2c)_{(x,y),(0,0)}+\beta(\rD_4\rD_2c)_{(x,y),(0,0)}\bigr)+O((\alpha,\beta)^2)\\
    &+\alpha\bigl((\rD_3c)_{(0,0),(x,y)}+u(\rD_1\rD_3c)_{(0,0),(x,y)}+v(\rD_2\rD_3c)_{(0,0),(x,y)}\bigr)\\
    &+\beta\bigl((\rD_4c)_{(0,0),(x,y)}+u(\rD_2\rD_4c)_{(0,0),(x,y)}+v(\rD_2\rD_4c)_{(0,0),(x,y)}\bigr)+O((u,v)^2).
    \end{split}
    \end{equation*}
Les coefficients des termes $u,v, \alpha,\beta$, $u\alpha$, $u\beta$, $v\alpha$ et  $v\beta$ nous donnent les relations:
\begin{equation}\label{relation}
(\rD_1c)_{(x,y),(0,0)}=(\rD_2c)_{(x,y),(0,0)}=(\rD_3c)_{(0,0),(x,y)}=(\rD_4c)_{(0,0),(x,y)}=0;
\end{equation}
\begin{equation}
\begin{split}
(\rD_3\rD_1c)_{(x,y),(0,0)}&=(\rD_1\rD_3c)_{(0,0),(x,y)}; (\rD_4\rD_1c)_{(x,y),(0,0)}=(\rD_1\rD_4c)_{(0,0),(x,y)};\\
(\rD_3\rD_2c)_{(x,y),(0,0)}&=(\rD_2\rD_3c)_{(0,0),(x,y)};
(\rD_4\rD_2c)_{(x,y),(0,0)}=(\rD_2\rD_4c)_{(0,0),(x,y)}.
\end{split}
\end{equation}
Alors, on obtient \[\partial_YQ_X(X,Y)=(\rD_4\rD_1c)_{(0,0),(X,Y)}=(\rD_1\rD_4c)_{(X,Y),(0,0)}=\partial_XQ_Y(X,Y).\]
Par ailleurs, on a $Q_X(0,0)=Q_Y(0,0)=0$. On en d\'eduit qu'il existe un polyn\^ome $Q(X,Y)$ homog\`ene de degr\'e $n$ tel que $\partial_XQ=Q_X$ et $\partial_YQ=Q_Y$:
\[Q(X,Y)=\frac{1}{n}(c_{1,0,n-1,0}X^{n}+c_{0,n-1,0,1}Y^n)+\sum\limits_{k=0 }^{n-2}\frac{1}{k+1}c_{n-k-1,k,0,1}X^{n-k-1}Y^{k+1}.\]    Ceci nous donne un cobord
$dQ(x,y,\alpha,\beta)=Q(x,y)-Q(x+\alpha,y+\beta)+Q(\alpha,\beta)$ v\'erifiant
la relation $\rD_1c=-\rD_1 dQ$ et $\rD_4c=-\rD_4dQ$.

Le $2$-cocycle $c^{'}_{(x,y),(\alpha,\beta)}=(c+dQ)_{(x,y),(\alpha,\beta)}$ v\'erifie la relation $\rD_1c^{'}=\rD_4c^{'}=0$, ce qui montre que $c^{'}$ est de la forme $\sum_{j+k=n}c_{j,k}y^j\alpha^k$. Par ailleurs,  on a
\begin{equation*}
\begin{split}
(\rD_2c')_{(x,y,\alpha,\beta)}&=\sum_{j+k=n} jc_{j,k}y^{j-1}\alpha^k; \\ 
(\rD_2c')_{(0,0,x+\alpha,v+\beta)}&=c_{1,n-1}(x+\alpha)^{n-1};
(\rD_2c')_{(0,0,x,y)}=c_{1,n-1}x^{n-1}.
\end{split}
\end{equation*}

On d\'eduit que  $c_{j,k}=0$ si $j\neq 0$,  par la relation $(\ref{cocycle1})$.
 Alors on a $c'_{(x,y,\alpha,\beta)}=c_{0,n}\alpha^{n}.$
Enfin, la relation (\ref{relation}) nous permet de conclure que tout $2$-cocycle du complexe $\cP^{n,\bullet}$ pour $n\geq 3$ n'a pas les termes de valuation faible $0$. Alors on conclut que le $2$-cocycle $c$ est nul.

\end{proof}

\begin{lemma}\label{coho1}Soit $V$ une repr\'esentation analytique de $\rP_m$.
\begin{itemize}
\item[(1)]Quel que soit $c\in V$, la fonction $((u,v),(x,y))\ra c_{(u,v),(x,y)}=c uy$ est un $2$-cocycle.
  R\'eciproquement, tout \'el\'ement de la cohomologie $\rH^2(\cS^{\bullet})$ est pr\'esent\'e par un $2$-cocycle de cette forme.
 \item[(2)]Si $c_{(u,v),(x,y)}=\sum_{i+j+k+l\geq 2}c_{i,j,k,l}u^iv^jx^ky^l$ est un $2$-cocycle analytique, alors il existe une suite de polyn\^ome homog\`enes $\{Q_n\in\cP_1^n,n\geq 2\}$ telle que
     \[c_{(u,v),(x,y)}=(c_{1,0,0,1}-c_{0,1,1,0})uy-\sum_{i=2}^{+\infty}dQ_i,\] o\`u $Q_n$ est defini par r\'ecurrence: $Q_2(X,Y)=\frac{1}{2}(c_{1,0,1,0}X^2+c_{0,1,0,1}Y^2)+c_{0,1,1,0}XY$ et si $Q_n(X,Y)=\sum\limits_{k=0}^{n}b_k^{(n)}X^{n-k}Y^{k}$, alors
     \begin{equation}\label{recurence}
      \begin{split}
    b_{n}^{(n)}&=\frac{1}{n}c_{0,n-1,0,1},
     b_{0}^{(n)}=\frac{1}{n}(c_{1,0,n-1,0}-\partial_1b_0^{(n-1)}), \\
     b_{k}^{(n)}&=\frac{1}{k}(c_{n-k,k-1,0,1}-
     (n-k)b_{k-1}^{(n-1)}), \text{ si } 1\leq k\leq n-1 .
     \end{split}
   \end{equation}

\end{itemize}
\end{lemma}
\begin{proof}[D\'emonstration] La premiere assertion est une cons\'equence directe de la deuxi\`eme. Il suffit de montrer la $(2)$. Soit
$c_{((u,v),(x,y))}=\sum_{i+j+k+l\geq 1} c_{i,j,k,l}u^iv^jx^ky^l$ un $2$-cocycle du complexe $\cS^{\bullet}$.

D'apr\`es le lemme \ref{coho}, par r\'ecurrence,   il existe une suite de polyn\^omes homog\`enes $\{Q_n\in\cP_n^1\}_{n\geq 3}$  telle que l'on peut \'ecrire $c$ sous la forme
$(c_{1,0,0,1}-c_{0,1,1,0})uy-\sum\limits_{i=2}^{n-1}dQ_i+P_n,$
o\`u $P_n=\sum_{i+j+k+l\geq n} c_{i,j,k,l}^{(n)}u^iv^jx^ky^l$ est une fonction analytique sur $\rP_m^2$ de valuation $\geq n$.
 On s'int\'eresse aux coefficients $c_{n-k-1,k,0,1}^{(n)}$ et $c_{1,0,n-1,0}^{(n)}$ dans $P_n$ pour $0\leq k\leq n-1$. 
 
 Par ailleurs, les termes de deg\'re $n$ dans $P_n$ provennant du  cobord $dQ_e$ avec $ 2\leq e\leq n-1$ sont dans les termes $\sum_{i+j+e\geq n}u^iv^j\frac{\partial_1^i\partial_2^jQ_e(x,y)}{i!j!}$ et $\sum_{l=1}^{+\infty}\frac{1}{l!}((e^y-1)u)^l\partial_X^lQ_e(X,Y)|_{u+x,v+y}$
 par la relation de $2$-cobord
$(dQ_e)_{(u,v),(x,y)}=Q_e(x,y)*(u,v)-Q_e(e^yu+x, v+y)+Q_e(u,v)$, l'action de $(u,v)$ sur $Q_e(x,y)$  et les \'egalit\'es
\begin{equation*}
Q_e(e^yu+x, v+y)=Q_e(u+x,v+y)+\sum_{l=1}^{+\infty}\frac{1}{l!}((e^y-1)u)^l\partial_X^lQ_e(X,Y)|_{u+x,v+y}.
\end{equation*}

Si $Q_n(X,Y)=\sum\limits_{k=0}^nb_k^{(n)}X^{n-k}Y^{k}$,  on a la relation de r\'ecurrence pour $c_{n-k-1,k,0,1}^{(n)}$ et $c_{1,0,n-1,0}^{(n)}$:
\begin{equation*}
\begin{split}
&c_{0,n-1,0,1}^{(n)}=c_{0,n-1,0,1}, c_{1,0,n-1,0}^{(n)}=c_{1,0,n-1,0}-\partial_1b_{0}^{(n-1)}\\
&c_{n-k-1,k,0,1}^{(n)}=c_{n-k-1,k,0,1}-(n-k-1)b_{k}^{(n-1)}, \text{ si } 0\leq k\leq n-2 \\
\end{split}
\end{equation*}
D'apr\`es le lemme (\ref{coho}), on a donc $Q_n(X,Y)=\frac{1}{n}c_{1,0,n-1,0}^{(n)}X^n+ \sum\limits_{k=0}^{n-1}\frac{1}{k+1}c_{n-k-1,k,0,1}^{(n)}X^{n-k-1}Y^{k+1}$ et on en d\'eduit que
 \begin{equation*}
     \begin{split}
     b_{n}^{(n)}&=\frac{1}{n}c^{(n)}_{0,n-1,0,1}=\frac{1}{n}c_{0,n-1,0,1},
     b_{0}^{(n)}=\frac{1}{n}c_{1,0,n-1,0}^{(n)}=\frac{1}{n}(c_{1,0,n-1,0}-\partial_1b_0^{(n-1)}), \\
     b_{k}^{(n)}&=\frac{1}{k}c_{n-k,k-1,0,1}^{(n)}=\frac{1}{k}(c_{n-k,k-1,0,1}-
     (n-k)b_{k-1}^{(n-1)}), \text{ si } 1\leq k\leq n-1 .
     \end{split}
   \end{equation*}
\end{proof}

Soit $V$ une repr\'esentation analytique de $\rP_m$ munie d'un $\Z_p$-r\'eseau $T$ qui est stable sous l'action de $\rP_m$. On d\'efinit une valuation $v_T$ sur $V$ par rapport \`a $T$:
\[\text{ si } x\in V,  v_T(x)=\min\{n: x\in p^{n}T\}.\]
Soit $x\in T$ et soit $(u,0)\in \rP_m$; on a $x*(u,0)=x+u\partial_1x+O(u^2)$ et on a donc $\partial_1x\in p^{-m}T$.

\begin{lemma}\label{controle}Soit $V$ une repr\'esentation analytique de $\rP_m$ munie d'un $\Z_p$-r\'eseau $T$ qui est stable sous l'action de $\rP_m$. Soit $c=\sum_{i+j+k+l\geq2}c_{ijkl}u^iv^jx^ky^l$ un $2$-cocycle analytique du complexe $\cC^{\an,\bullet}$ avec $v_T(c_{i,j,k,l})\geq -m(i+j+k+l)$. Si $\{Q_n(x,y)=\sum\limits_{k=0}^{n}b_{k}^{(n)}x^{n-k}y^{k}\in\cP_n^1\}$ est la suite des polyn\^omes homog\`enes construit dans le lemme pr\'ec\'edent, pour tous $n\geq 2$ et $0\leq k\leq n$, on a
\[ v_T(b_{k}^{(n)}) \geq -mn-v_p(n!).\]
\end{lemma}
\begin{proof}[D\'emonstration]
Si $n=2$, le lemme est vrai par la formule explicite de $Q_2(x,y)$.
La relation de r\'ecurrence $(\ref{recurence})$ de $b_k^{(n)}$ nous donne la relation  suivante:
\begin{equation}
      \begin{split}
    b_{n}^{(n)}&=\frac{1}{n}c_{0,n-1,0,1};
     b_{0}^{(n)}=\frac{1}{n}(c_{1,0,n-1,0}-\partial_1b_0^{(n-1)}); \\
     b_{n-1}^{(n)}&=\sum\limits_{i=0}^{n-3}\frac{(-1)^i}{(n-1)\cdots (n-1-i)}c_{1,n-2-i,0,1}+\frac{(-1)^{n-2}}{(n-1)!}b_1^{(2)};\\
     b_{k}^{(n)}&=\sum\limits_{i=0}^{k-1}\frac{(k-n)^i}{k\cdots(k-i)}c_{n-k,k-1-i,0,1}+
     \frac{
     (k-n)^{k}}{k!}b_{0}^{(n-k)}, \text{ si } 1\leq k\leq n-2 .
     \end{split}
   \end{equation}
On en d\'eduit que le lemme est vrai pour $b_n^{(n)}$.

On montrera par r\'ecurrence que $v_T(b_0^{(n)})\geq -nm-v_p(n!)$ pour tous $n\geq 2$.
   Supposons $v_T(b_0^{(e)})\geq -em-v_p(e!)$ pour $ 2\leq e\leq n-1$.  On a $v_T(\partial_1b_{0}^{(n-1)})\geq -nm-v_p((n-1)!)$ car $v_T(\partial_1x)\geq v_T(x)-m$. On d\'eduit de la relation de r\'ecurrence
$b_{0}^{(n)}=\frac{1}{n}(c_{1,0,n-1,0}-\partial_1b_0^{(n-1)}),$
que $v_T(b_0^{(n)})\geq -nm-v_p(n!)$ pour tous $n\geq 2$.

On a donc, si  $1\leq k\leq n-2$,
\[v_T(\frac{
     (k-n)^{k}}{k!}b_{0}^{(n-k)})\geq -m(n-k)-v_p((n-k)!)-v_p(k!)\geq -mn-v_p(n!).\]
Par ailleurs, on a $v_p(\frac{(k-n)^i}{k\cdots(k-i)}c_{n-k,k-1-i,0,1})\geq -m(n-i)-v_p(k!)$ pour tous $1\leq i\leq k-1$. On en d\'eduit que, pour $1\leq k\leq n-2$, on a
$v_T(b_k^{(n)})\geq -mn-v_p(n!).$
Enfin, on a
\[
v_T(b_{n-1}^{(n)})\geq \inf\limits_{1\leq i\leq n-3}\{ v_p(\frac{(-1)^i}{(n-1)\cdots (n-1-i)}c_{1,n-2-i,0,1}), v_T(\frac{(-1)^{n-2}}{(n-1)!}b_1^{(2)})\}\geq -mn-v_p(n!).\]
 Ceci permet de conclure le lemme.

\end{proof}

On revient \`a la d\'emonstration de la proposition (\ref{analytic}):

Soit $c_{(u,v),(x,y)}=\sum_{i+j+k+l\geq 2}c_{i,j,k,l}u^iv^jx^ky^l$ un $2$-cocycle analytique du complexe $\cC^{\an,\bullet}$. Il existe une constante $N$ assez grande telle que $p^Nc_{i,j,k,l}\in p^{-(m-1)(i+j+k+l)}T$ pour tout $i,j,k,l$. 
Alors, on peut remplacer $c_{(u,v),(x,y)}$ par un $2$-cocycle v\'erifiant la condition du lemme $\ref{controle}$. 
Donc il existe  une s\'erie de polyn\^omes homog\`enes 
$\{Q_n(x,y)=\sum\limits_{i=0}^{n}b_{i}^{(n)}x^iy^{n-i}\in\cP_n^1\}_{n\geq 2}$ telle que 
$ v_T(b_{i}^{(n)}) \geq  -mn-v_p(n!)>-(m+v_p(2p))n$ et
$c_{(u,v),(x,y)}=(c_{1,0,0,1}-c_{0,1,1,0})uy-\sum\limits_{i=2}^{+\infty}dQ_i.$
Par cons\'equent, la s\'erie $\sum\limits_{n=2}^{+\infty}Q_{n}(u,v)$ converge vers une fonction analytique $Q(u,v)$ sur $\rP_{m+v_p(2p)}$.

Par la suite exacte d'inflation-restriction, on a
\[\xymatrix{0\ar[r]&\rH^2(\rP_m/\rP_{m+v_p(2p)},V^{\rP_{m+v_p(2p)}})
\ar[r]&\rH^2(\rP_m,V)\ar[r]&\rH^2(\rP_{m+v_p(2p)},V)}.\]
L'argument ci-dessus donne que
$c_{((u,v),(x,y))}-(c_{1,0,0,1}-c_{0,1,1,0})uy$ est nul sous
l'application de restriction. Par ailleurs, comme $V^{\rP_{m+v_p(2p)}}$ est un espace vectoriel sur $\Q_p$ de dimension finie et
$\rP_m/\rP_{m+v_p(2p)}$ est un groupe fini, on obtient
$\rH^2(\rP_m/\rP_{m+v_p(2p)},V^{\rP_{m+v_p(2p)}})=0$. Autrement
dit, l'application de restriction est injective. Donc
$c_{((u,v),(x,y))}=(c_{1,0,0,1}-c_{0,1,1,0})uy\in\rH^2(\rP_m,V)$.

Enfin,
comme $uy$ est un $2$-cocycle, on peut prendre n'importe quel
coefficient $c_{1,0,0,1}-c_{0,1,1,0}\in V$. Donc tout \'el\'ement de
$V/(\partial_1,\partial_2-1)$ est repr\'esentable par un $2$-cocycle
analytique.

\subsection{Cohomologie des $\B_{\dR}^+(\fK^+_{Mp^{\infty}})$-repr\'esentations du groupe $P_{\fK_M}$ }

\begin{prop}\label{Ta}Soit $M\geq 1$ un entier tel que $v_p(M)=m\geq v_p(2p)$; soit $V$ une $\Q_p$-repr\'esentation de $P_{\fK_M}$ munie d'un $\Z_p$-r\'eseau $T$ tel que $P_{\fK_M}$ agit trivialement sur $T/2pT$, alors, pour $i\in\N$, $\rT_M$ induit un isomorphisme:
\[\rT_M: \rH^i(P_{\fK_M},\C(\fK_{Mp^{\infty}}^+)\otimes V)\cong \rH^i(P_{\fK_M},\fK_M^+\otimes V).\]
\end{prop}
Pour d\'emontrer cette proposition, on a besoin de lemmes pr\'eparatoires:

On a une d\'ecomposition $\C(\fK_{Mp^{\infty}}^+)=\fK_M^+\oplus X_M$ comme $\fK_M^+$-module de Banach, o\`u $X_M=\{x\in \C(\fK_{Mp^{\infty}}^+)| \tr_M(x)=0\}$. On note $J'=\cup_n\frac{1}{p^n}\N$. D'apr\`es le corollaire \ref{C}, un \'el\'ement $x$ de $X_M$ s'\'ecrit de mani\`ere unique sous la forme
\begin{equation}\label{XM}x=\sum_{(i,j)\notin (0,\N), (i,j)\in J\times J'}a_{ij}(x)\z_M^iq_M^{j},
\end{equation}
o\`u $a_{ij}(x)$ est une suite d'\'el\'ement de $F_M$ v\'erifi\'ee que $v_p(a_{ij}(x))+\frac{j}{M}$ tend vers $+\infty$ \`a l'infini (i.e. $\forall N>0$, l'ensemble de $(i,j)\in J\times J'$ tel que $v_p(a_{ij}(x))+\frac{j}{M}\leq N$ est un ensemble fini).

\begin{lemma}\label{X}
\begin{itemize}
\item[(1)] Un \'el\'ement $x\in X_M^{U_m}$ si et seulement si $a_{ij}(x)=0$ pour tout $j\notin \N$;
\item[(2)]On a une d\'ecomposition $X_M=X_M^{U_m}\oplus Y_M$ comme $\fK_M^+$-module de Banach, o\`u
    \[Y_M=\{x=\sum_{(i,j)\in J\times (J'-\N)}a_{ij}(x)\z_M^iq_M^{j}\}\subset X_M.\]
\end{itemize}
\end{lemma}
\begin{proof}[D\'emonstration]
On d\'eduit les resultats du corollaire $\ref{C}$.
\end{proof}
 Soit $e_1,\cdots, e_d$ une $\Z_p$-base de $T$.  Tout \'el\'ement $x\in X_M\otimes T$ s'\'ecrit uniquement sous la forme $\sum_{i=1}^dx_i\otimes e_i$ avec $x_i\in X_M$. On d\'efinit une valuation $v$ sur $X_M\otimes T$ par la formule: $v(x)=\inf_{1\leq i\leq d }v_p(x_i).$

\begin{lemma}\label{sur}
Il existe des isomorphismes de $\fK_M^{+}$-modules:
\begin{itemize}
\item[(1)]$(X_M\otimes T)^{U_m}\cong(X_M^{U_m}\otimes T)^{U_m}\cong X_M^{U_m}\otimes T^{U_m}$;
\item[(2)]$Z_M:=\bigl(X_M^{U_m}\otimes T\bigr)/(u_m-1) \cong\rH^1(U_m, X_M\otimes T ).$
\end{itemize}
\end{lemma}
\begin{proof}[D\'emonstration]
On se ram\`ene \`a montrer que l'op\'erateur $u_m-1$ est inversible sur le $\fK_M^{+}$-module $ Y_M\otimes T$.
On a la formule explicite de $u_m-1$ sur $x=\sum\limits_{(i,j)\in J\times (J'-\N)}a_{ij}\z_{M}^{i}q_M^{j}\in  Y_M$:
\[(u_m-1)(x)=\sum_{(i,j)\in J\times (J'-\N)}a_{ij}\z_M^i(\z_{p^m}^{jp^m}-1)q_M^{j}.\]
Comme $0\leq v_p(\z_{p^m}^{jp^m}-1)< \frac{1}{p-1}$,  $\sum_{(i,j)\in J\times (J'-\N)}a_{ij}\z_{M}^{i}(\z_{p^m}^{jp^m}-1)^{-1}q_M^{j}$ est encore un \'el\'ement de $Y_M$. Donc l'op\'erateur $u_m-1$ est inversible sur $Y_M$ et on a $v((u_m-1)^{-1}x)\geq v(x)-\frac{1}{p-1}$.

L'op\'erateur $u_m-1$ sur $Y_M\otimes T$ se d\'ecompose comme suit:
\begin{equation*}
u_m-1=(u_m\otimes1-1)+ (u_m\otimes1)(1\otimes u_m-1)=(u_m\otimes1-1)(1+u'),
\end{equation*}
o\`u $u'=(u_m\otimes1-1)^{-1}(u_m\otimes1)(1\otimes u_m-1)$.
Si $x=\sum_{i}x_i\otimes e_i\in Y_M\otimes T$, alors on a  $v((1\otimes u_m-1)(x_i\otimes e_i))\geq v(x_i\otimes e_i)+v_p(2p)$ car $\rP_m$ agit trivialement sur $T/2pT$. Donc on a
\[v((1\otimes u_m-1)x)=\inf_iv((1\otimes u_m-1)x_i\otimes e_i)\geq \inf_iv(x_i\otimes e_i)+v_p(2p)=v(x)+v_p(2p).\]
Par cons\'equent, on a
$v(u'(x))\geq v(x)+v_p(2p)-\frac{1}{p-1}$.
Ceci permet de montrer que la s\'erie $\sum_{n=0}^{+\infty}(-u')^n$ converge et sa somme est l'inverse de $(1+u')$ sur $Y_M\otimes T$. Donc $u_m-1$ est inversible sur $Y_M\otimes T$.
\end{proof}

\begin{lemma}\label{inver}Si $a\in\Z_p^*$, l'op\'erateur $a\gamma_m-1$ est inversible sur $X_M^{U_m}$.
\end{lemma}
\begin{proof}[D\'emonstration]Comme un \'el\'ement $x$ dans $X_M^{U_m}$ s'\'ecrit sous la forme
 $\sum_{i\in J, i\neq 0}\sum_{k\in\N}a_{ik}(x)\z_M^iq_M^{k} $ avec $a_{ij}(x)\in F_M$, l'action de $a\gamma_m-1$ sur un \'el\'ement $x\in X_M^{U_m}$ est donn\'ee par la formule:
$(a\gamma_m-1)x=\sum_{i\in J, i\neq 0}\sum_{k\in\N}a_{ik}(x)(a\z_{p^m}^{i(e^{p^m}-1)}-1)\z_M^iq_M^{k}.$

Comme $i\neq 0$, on a $0\leq v_p(a\z_{p^m}^{i(e^{p^m}-1)}-1)\leq \frac{1}{p-1}$. On en d\'eduit que $a\gamma_m-1$ est inversible sur $X_M^{U_m}$ avec la formule de l'action de $(a\gamma_m-1)^{-1}$ sur $X_M^{U_m}$ donn\'ee par
\[(a\gamma_m-1)^{-1}x=\sum_{i\in J, i\neq 0}\sum_{k\in\N}a_{ik}(a\z_{p^m}^{i(e^{p^m}-1)}-1)^{-1}\z_M^iq_M^{k}.\]
\end{proof}

Revenons \`a la d\'emonstration de la proposition \ref{Ta}:
\begin{proof}[D\'emonstration] On se ram\`ene  \`a prouver $\rH^i(\rP_m, X_M\otimes T )=0.$

Notons $Z_M=\bigl(X_M^{U_m}\otimes T\bigr)/(u_m-1)$. D'apr\`es le lemme \ref{sur}, on a le diagramme commutatif:
\[\xymatrix{0\ar[r]&Z_M\ar[r]\ar[d]^{e^{-p^m}\gamma_m-1}& \rH^1(U_m,X_M\otimes T)\ar[r]\ar[d]^{\gamma_m-1}&0\\ 0\ar[r]&Z_M\ar[r]&\rH^1(U_m, X_M\otimes T)\ar[r]& 0 .}\] 
 Comme $e^{-p^m}\gamma_m\otimes 1-1$ est inversible sur $X_M^{U_m}\otimes T$ d'apr\`es le lemme \ref{inver}, on peut factoriser l'op\'erateur $e^{-p^m}\gamma_m-1$ sur $X_M^{U_m}\otimes T$ comme suit:
\begin{equation*}
e^{-p^m}\gamma_m-1=(e^{-p^m}\gamma_m\otimes 1-1)+(e^{-p^m}\gamma_m\otimes 1)(1\otimes e^{-p^m}\gamma_m-1)=(e^{-p^m}\gamma_m\otimes 1-1)(1+\gamma^{'}),
\end{equation*}
o\`u $\gamma^{'}=(e^{-p^m}\gamma_m\otimes 1-1)^{-1}(e^{-p^m}\gamma_m\otimes 1)(1\otimes e^{-p^m}\gamma_m-1)$.

Par ailleurs, si $x=\sum_ix_i\otimes e_i\in X_M^{U_m}\otimes T$, on a $v((1\otimes e^{-p^m}\gamma_m-1)(x))\geq  v_p(2p)+v(x)$ car $\rP_m$ agit trivialement sur $T/2pT$. Ceci permet de montrer que la s\'erie $\sum_{n=0}^{+\infty}(-\gamma^{'})^n$ converge et sa somme est l'inverse de $(1+\gamma^{'})$ sur $X_M^{U_m}\otimes T$. Donc $e^{-p^m}\gamma_m-1$ est inversible sur $X_M^{U_m}\otimes T$.
On en d\'eduit que l'op\'erateur  $e^{-p^m}\gamma_m-1$ est inversible sur $Z_M$. Donc on a 
\[\rH^1(U_m, X_M\otimes T)^{\Gamma_m}=0 \text{ et } \rH^1(U_m,X_M\otimes T)/(\gamma_m-1)=0.\]
\begin{itemize}
\item[(1)]si $i=1$, par la suite exacte d'inflation-restriction, on a:
\[0\ra\rH^1(\Gamma_m,(X_M\otimes T)^{U_m})\ra\rH^1(\rP_m,X_M\otimes T)\ra\rH^1(U_m,X_M\otimes T)^{\Gamma_m}=0\ra 0.\]
Par ailleurs, on d\'eduit $\rH^1(\Gamma_m,(X_M\otimes T)^{U_m})=0$ des lemmes \ref{sur} et $\ref{inver}$. Ceci permet de conclure.
\item[(2)]si $i=2$, d'apr\`es la suite spectrale de Hochschild-Serre, on a l'isomorphisme:
\[\rH^2(\rP_m, X_M\otimes T)\cong\rH^1(\Gamma_m,\rH^1(U_m,X_M\otimes T))=\rH^1(U_m,X_M\otimes T)/(\gamma_m-1)=0.\]
\end{itemize}
\end{proof}

\begin{prop}\label{Tatenorm}Soit $v_p(M)\geq v_p(2p)$; si $V$ est une
$\Q_p$-repr\'esentation de $P_{\fK_M}$ munie d'un $\Z_p$-r\'eseau
$T$ tel que $P_{\fK_{M}}$ agit trivialement sur $T/2pT$, alors pour
$i\in\N$, $\bR_M$ induit un isomorphisme:
\[\bR_M:\rH^{i}(P_{\fK_M},\B_{\dR}^{+}(\fK^+_{Mp^{\infty}})\otimes V)\cong\rH^{i}(P_{\fK_M},\tilde{\fK}^+_M\otimes V).\]
\end{prop}
\begin{proof}[D\'emonstration]
On pose $\tilde{X}_M=\{x\in\B_{\dR}^{+}(\fK_{Mp^{\infty}}^+)| \bR_M(x)=0\}$. Quel que soit $n\in\N$, on a une d\'ecomposition de $\B_{\dR}^+(\fK_{Mp^{\infty}}^+)/t^n$ comme $\tilde{\fK}^+/t^n$-module:
$\B_{\dR}^+(\fK_{Mp^{\infty}}^+)/t^n= \tilde{\fK}_M^+/t^n\oplus \tilde{X}_M/t^n.$
Cela induit une d\'ecomposition de $\B_{\dR}^+(\fK_{Mp^{\infty}}^+)$ comme $\tilde{\fK}^+$-module:
$\B_{\dR}^{+}(\fK_{Mp^{\infty}}^+)=\tilde{\fK}_M^+\oplus\tilde{X}_M .$
On se ram\`ene donc \`a montrer: $\rH^{i}(P_{\fK_M},\tilde{X}_M\otimes T)=0.$

Soit $c_0$ un $i$-cocycle qui pr\'esente un \'el\'ement dans $\rH^i(P_{\fK_M}, \tilde{X}_M\otimes T)$. D'apr\`es la proposition \ref{Ta}, $\theta (c_0)$ est un cobord dans $\rH^i(P_{\fK_M}, X_M\otimes T)$ et donc $\theta c_0=d b_0$. L'\'el\'ement $c_0- d(\iota_{\dR}(b_0))$ est encore un $i$-cocycle dans $\rH^i(P_{\fK_M}, \tilde{X}_M\otimes T)$, qui est \`a valeurs dans $t\tilde{X}_M\otimes T$. On d\'efinit une suites $(d b_n)_{n\in\N}$ de $i$-cobords dans $\rH^i(P_{\fK_M}, X_M\otimes T(-n))$ et $(c_n)_{n\in\N}$ de $i$-cocycles dans $\rH^i(P_{\fK_M}, \tilde{X}_M\otimes T(-n))$ par r\'ecurrence:
 \[db_{n-1}=\theta(c_{n-1}) \text{ et } c_n= t^{-1}(c_{n-1}-d(\iota_{\dR}b_{n-1}))),   \text{ si } n\geq 1.\]
Donc, on a $ c_0=\sum_{n=0}^{+\infty}t^nd(\iota_{\dR}(b_{n})),$
qui est une somme des cobords qui converge dans $\rH^i(P_{\fK_M}, \tilde{X}_M\otimes T)$. Ceci permet de conclure la proposition.
\end{proof} 

%% file: Loi.tex
\section{La loi de r\'eciprocit\'e explicite de Kato}

\subsection{Construction de l'application  exponentielle duale de Kato}\label{constructiondeloi}
On note $\cK^+=\Q_p[[q]]$ le compl\'et\'e $q$-adique de $\fK^+$ et $\cK_M^+$ le compl\'et\'e $q$-adique de $\fK^+$-module $\fK_M^+=\fK^+[\z_M,q_M]$. 
On a donc $\cK^+_M=F_M[[q_M]]$. On note $\tilde{\cK}^+=\Q_p[[\tilde{q},t]]$ le compl\'et\'e $\tilde{q}$-adique de $\tilde{\fK}^+=\iota_{\dR}(\fK^+)[[t]]$, 
o\`u l'application $\iota_{\dR}$ identifie $\fK^+$ \`a un sous-anneau de $\B_{\dR}^+(\ol{\fK}^+)$.  
On note $\tilde{\cK}_M^+$ le compl\'et\'e $\tilde{q}$-adique de $\tilde{\fK}^+_M$ comme $\tilde{\fK}^+$-module. 
D'apr\`es le lemme \ref{K}, on a bien $\tilde{\cK}^+_M=\tilde{\cK}^+[\tilde{\z}_M,\tilde{q}_M]=\tilde{\cK}^+[\z_M,\tilde{q}_M]=F_M[[t,\tilde{q}_M]].$
On d\'efinit une application $\theta: \tilde{\cK}^+_M\ra F_M[[q_M]]$ par r\'eduction modulo $t$. Elle co\"incide avec l'application $\theta$ sur $\tilde{\fK}_M^+$.

On note $\cK_{M,n}^+=\cK_M^+/(q_M)^{n}$, $\tilde{\cK}_{M,n}^+=\tilde{\cK}_M^+/(t,\tilde{q}_M)^{n}$. La repr\'esentation $\tilde{\cK}_M^+\otimes V_{k,j}$ de $\rP_m$ n'est pas une repr\'esentation analytique; mais c'est la limite projective des repr\'esentations analytiques $\{\tilde{\cK}_{M,n}^+\otimes V_{k,j}\}_{n\in \N}$ de $\rP_m$. 
\begin{lemma}Les actions de $\partial_1$ et $\partial_2-1$ sur $t$ et $\tilde{q}_M$ sont donn\'ees par les formules suivantes:

\begin{equation*}
\partial_1(t)=0,
\partial_1(\tilde{q}_M)=\frac{t}{M}\tilde{q}_M;
\partial_2(t)=t,\partial_2(\tilde{q}_M)=0.
\end{equation*}
\end{lemma}

\begin{prop}\label{reskj}Si $v_p(M)=m\geq v_p(2p)$ et $1\leq j\leq k-1$, alors l'application
\[f(q_M)\mapsto te_1^{k-2}f(\tilde{q}_M) \] induit un isomorphisme de $\cK_{M,n}^+$ sur $(\tilde{\cK}_{M,n+j-1}^+\otimes V_{k,j})/(\partial_1,\partial_2-1)$, pour tout $n\in\N$.
\end{prop}
\begin{proof}[D\'emonstration]Les $\{e_1^ie_2^{k-2-i}t^{l+2-j}q_M^v: 0\leq i\leq k-2, v,l\geq 0, v+l\leq n-1\}$ forment une $F_M$-base de $\tilde{\cK}_{M,n}^+\otimes V_{k,j}$.
 Rappelons que l'action de $\rP_m$ sur $V_{k,j}$ est donn\'ee par la formule: $e_1*(\begin{smallmatrix}a&b\\ c&d\end{smallmatrix})=ae_1+be_2$ et $e_2*(\begin{smallmatrix}a&b\\ c&d\end{smallmatrix})=ce_1+de_2$, si $(\begin{smallmatrix}a&b\\ c&d\end{smallmatrix})\in \rP_m$.
 On a donc les formules suivantes pour l'op\'erateur $\partial_1$:
\begin{equation*}
\partial_1(e_1)=p^{-m}\lim_{n\ra+\infty}e_1*\frac{u_m^{p^n}-1}{p^{n}}=e_2;
\partial_1(e_2)=p^{-m}\lim_{n\ra+\infty}e_2*\frac{u_m^{p^n}-1}{p^{n}}=0.
\end{equation*}
On en d\'eduit que
\begin{equation}\label{partial1}
\partial_1(e_1^ie_2^{k-2-i}t^l\tilde{q}^v_M))=ie_1^{i-1}e_2^{k-1-i}t^l\tilde{q}^v_M+e_1^ie_2^{k-2-i}t^lv\tilde{q}^v_M\frac{t}{M}.
\end{equation}
Ceci nous fournit la relation suivante dans $(\tilde{\cK}_{M,n}^+\otimes V_{k,j})/\partial_1$:
\[e_1^ie_2^{k-2-i}t^l\tilde{q}_M^v=\frac{-1}{i+1}e_1^{i+1}e_2^{k-3-i}t^{l+1}\tilde{q}^v_M\frac{v}{M}=\frac{(-1)^{k-2-i}i!}{(k-2)!}e_1^{k-2}t^{l+k-2-i}\tilde{q}^v_M(\frac{v}{M})^{k-2-i}. \]
Par cons\'equent, chaque \'el\'ement de $(\tilde{\cK}_{M,n}^+\otimes V_{k,j})/\partial_1 $ peut \^etre repr\'esent\'e par un \'el\'ement de $\tilde{\cK}_{M,n}^+(1\otimes (e_1^{k-2}t^{2-j}))$.
Ceci implique que le morphisme naturel de $F_M$-espaces
\[\phi:\tilde{\cK}_{M,n}^+(1\otimes (e_1^{k-2}t^{2-j}))\ra (\tilde{\cK}_{M,n}^+\otimes V_{k,j})/\partial_1 \]
est surjectif.  La formule $(\ref{partial1})$ montre qu'il n'existe pas d'\'el\'ement $x$ dans $(\tilde{\cK}_{M,n}^+\otimes V_{k,j})$ tel que $\partial_1x$ appartient \`a $\tilde{\cK}_{M,n}^+(1\otimes (e_1^{k-2}t^{2-j}))$. En cons\'equence, $\phi$ est un isomorphisme.

Par ailleurs, on a les formules suivantes pour l'op\'erateur $\partial_2-1$ sur $\tilde{\cK}_{M,n}^+\otimes V_{k,j}$:
\begin{equation*}
\partial_2(e_1)=p^{-m}\lim_{n\ra+\infty}e_1*\frac{\gamma_m^{p^n}-1}{p^n}=0,
\partial_2(e_2)=p^{-m}\lim_{n\ra+\infty}e_2*\frac{\gamma_m^{p^n}-1}{p^n}=e_2.
\end{equation*}

On en d\'eduit que:
\begin{equation}\label{partial2}
(\partial_2-1)(e_1^ie_2^{k-2-i}t^l\tilde{q}_M^v)=(k-3-i+l)e_1^ie_2^{k-2-i}t^l\tilde{q}_M^v.
\end{equation}

On tire la commutativit\'e de $\partial_1$ et $\partial_2-1$ sur $(\tilde{\cK}_M^+\otimes V_{k,j})_n$ des formules $(\ref{partial1})$ et $(\ref{partial2})$. 
On a donc le diagramme commutatif suivant:
\[\xymatrix{ (e_1^{k-2}t^{2-j}\otimes\tilde{\cK}_{M,n}^+)\ar[r]^-{\cong}\ar[d]^{\partial_2-1} &(\tilde{\cK}_{M,n}^+\otimes V_{k,j})/\partial_1\ar[d]^{\partial_2-1} \\ (e_1^{k-2}t^{2-j}\otimes\tilde{\cK}_{M,n}^+)\ar[r]^-{\cong}& (\tilde{\cK}_{M,n}^+\otimes V_{k,j})_n/\partial_1      }.       \]

On d\'eduit de la formule $(\ref{partial2})$ que, $e_1^ie_2^{k-2-i}t^l\tilde{q}_M^v$ est un vecteur propre de
$\partial_2-1$ avec la valeur propre $(k-3-i+l)$. Donc $\tilde{\cK}_{M,n}^+(1\otimes e_1^{k-2}t^{2-j})$ est dans l'image de $\partial_2-1$ si $l-1\neq 0$ et donc  on obtient
$(\tilde{\cK}_{M,n}^+\otimes
V_{k,j})/(\partial_1,\partial_2-1)\cong  \cK_{M,n+1-j}^+te_1^{k-2} ,$  ce qui
permet de conclure.
\end{proof}

En composant les applications obtenues dans les paragraphes pr\'ec\'edents , on obtient le diagramme suivant:
\[
\xymatrix{
\rH^2(\cG_{\fK_M},\B_{\dR}^{+}\otimes V_{k,j}) & \rH^2(P_{\fK_M},\B_{\dR}^{+}(\fK^+_{Mp^{\infty}})\otimes
V_{k,j})\ar[l]^-{(1)}\ar[r]^-{(2)}\ar@{..>}[dd]^{\exp^{*}_{\mathbf{Kato}}}&
\rH^2(P_{\fK_M},\tilde{\fK}^+_{M}\otimes V_{k,j})\ar[d]^-{(3)}\\ & & \plim_{n}\rH^2(P_{\fK_M},\tilde{\cK}_{M,n}^+\otimes V_{k,j} )\ar[d]^-{(4)}\\
& \cK_M^+ &\plim_n (\tilde{\cK}^+_{M,n}\otimes
V_{k,j})/(\partial_1,\partial_2-1)\ar[l]^-{\cong}_-{(5)},
}
\]
o\`u
$\bullet$ l'application $(1)$, d'inflation, est injective car $(\B_{\dR}^{+})^{\cG_{\fK_{Mp^{\infty}}}}=\B_{\dR}^{+}(\fK_{Mp^{\infty}}^+)$ et $\cG_{\fK_{Mp^{\infty}}}$ agit trivialement sur $V_{k,j}$; \\
$\bullet$ $(2)$ est l'isomorphisme induit par "la trace de
Tate normalis\'ee" $\bR_M$ (c.f prop. $\ref{Tatenorm}$);\\
$\bullet$ $(3)$ est l'application naturelle induit par la projection $\tilde{\fK}^+_{M}\otimes V_{k,j}\ra \tilde{\cK}_{M,n}^+\otimes V_{k,j}  $;\\
$\bullet$ $(4)$ est l'isomorphisme du corollaire $\ref{iso1}$ car $\tilde{\cK}_{M,n}^+\otimes V_{k,j} $ est analytique pour tout $n$;\\
$\bullet$ Comme $(\tilde{\cK}_M^+\otimes V_{k,j} )_n$ est une repr\'esentation analytique pour tout $n$, l'application $(4)$ se calcule gr\^ace \`a la prop $\ref{analytic}$  . Plus pr\'ecis\'ement, cela se fait comme suit:
   \begin{recette}\label{res} On d\'efinit une application $\res_{k,j}^{(n)}: \tilde{\cK}^+_{M,n+j-1}\otimes V_{k,j}\ra \cK_{M,n}^+$ en composant la
projection $\tilde{\cK}^+_{M,n+j-1}\otimes V_{k,j}\ra (\tilde{\cK}^+_{M,n+j-1}\otimes
V_{k,j})/(\partial_1,\partial_2-1)$ avec l'inverse de l'isomorphisme dans la proposition
$\ref{reskj}$. En prenant la limite projective, on obtient un morphisme $\res_{k,j}:\tilde{\cK}^+_M\otimes V_{k,j}\ra \cK_M^+$. Si $c=(c^{(n)})\in \plim \rH^2(P_{\fK_M},\tilde{\cK}_{M,n+j-1}^+\otimes V_{k,j} ) $ est repr\'esent\'e par une limite de $2$-cocycle analytique $(\sigma,\tau)\mapsto c^{(n)}_{\sigma,\tau}$ sur $P_{\fK_M}$ \`a valeurs dans $\tilde{\cK}_{M,n+j-1}^+\otimes V_{k,j} $, alors l'image de $c$ sous l'application $(5)$ est $\res_{k,j}(\delta^{(2)}(c))$, o\`u $\delta^{(2)}$ est l'application d\'efinie dans la proposition $\ref{analytic}$.
\end{recette}
On d\'efinit l'application $\exp^{*}_{\mathbf{Kato}}$  en composant les applications $(2),(3), (4), (5)$.

\subsection{Application au syst\`eme d'Euler de Kato}\label{constrcocycle}

\subsubsection{Esquise de la preuve du Th\'eor\`eme (\ref{theo})}

Soient $M\geq 1$ et $A=\bigl(\begin{smallmatrix}\alpha&\beta\\ \gamma&
\delta\end{smallmatrix}\bigr)$ avec
$\alpha,\beta,\gamma,\delta\in\{1,\cdots,M\}$ et $\det A\in\Z_p^{*}$. On note $\psi_{M,A}=1_{A+M\bM_2(\hat{\Z})}$ la fonction caract\'eristique de $A+M\bM_2(\hat{\Z})$. C'est une fonction invariante sous l'action de $\cG_{\fK_M}$. Par ailleurs, la distribution $z_{\mathbf{Kato},c,d}(k,j)$ appartient \`a $\rH^2(\cG_{\fK_M},\fD_{\alg}(\bM_2(\Q\otimes \hat{\Z})^{(p)},V_{k,j}))$. Alors, on a
\[\int\psi_{M,A}z_{\mathbf{Kato},c,d}(k,j)\in\rH^2(\cG_{\fK_M},V_{k,j})\]
et on note son image dans $\rH^2(\cG_{\fK_M},\B_{\dR}^{+}(\fK_{Mp^{\infty}}^+)\otimes
V_{k,j})$ par $z_{M,A}$. En outre, on note $z_{\mathbf{Eis},c,d}^{(p)}(k,j)$ la distribution sur $\bM_2(\Q\otimes\hat{\Z})^{(p)}$ \`a valeurs dans $\fK_{\infty}^+$ obtenue
en restreignant la distribution $z_{\mathbf{Eis},c,d}(k,j)\in\fD_{\alg}(\bM_2(\Q\otimes\hat{\Z}),\cM_k(\Q_p^{\cycl}))$ \`a
$\bM_2(\Q\otimes\hat{\Z})^{(p)}$, et en utilisant l'injection de $\cM_k(\Q_p^{\cycl})$
dans $\fK_{\infty}^+$.
\begin{lemma}Soit $k$ un entier $\geq 2$ et soit $j\in\N$ tel que $1\leq j\leq k-1$. La distribution $z_{\mathbf{Eis},c,d}^{(p)}(k,j)$ est invariante sous l'action de $\cG_{\fK}$.
\end{lemma}
\begin{proof}[D\'emonstration]Le groupe $\cG_{\fK}$ agit sur $\cM^{\con}(\Q_p^{\cycl})$ \`a travers son quotient $P^{\cycl}_{\Q_p}$ vu comme  sous-groupe de $\GL_2(\hat{\Z})$. D'apr\`es le th\'eor\`eme (\ref{eiskj}), on a:
\[\text{ si } \gamma\in\GL_2(\Q\otimes\hat{\Z}), z_{\mathbf{Eis},c,d}(k,j)_{|_k}\gamma=|\det\gamma|^{j-1}z_{\mathbf{Eis},c,d}(k,j).\]
En particulier, si $\gamma\in \GL_2(\hat{\Z})$, alors on a $|\det\gamma|=1$. Donc la distribution $z^{(p)}_{\mathbf{Eis},c,d}(k,j)$ est invariante sous l'action de $\GL_2(\hat{\Z})$.
\end{proof}
Pour montrer le th\'eor\`eme (\ref{theo}),
il suffit de prouver:
\begin{prop}\label{exam}
Pour toute paire $(M,A)$ ci-dessus et $v_p(M)\geq v_p(2p)$, $\det A
\in \Z_p^{*}$, on a
\begin{equation*}
\exp^{*}_{\mathbf{Kato}}(z_{M,A})=\frac{1}{(j-1)!}M^{k-2-2j}F_{c,\alpha/M,\beta/M}^{(k-j)}E_{d,\gamma/M,\delta/M}^{(j)}.
\end{equation*}
\end{prop}
Pour d\'emontrer ceci, nous aurons besoin d'\'ecrire un $2$-cocycle
explicite repr\'esentant $z_{M,A}$ et le suivre \`a travers les
\'etapes de la construction de l'application $\exp^{*}_{Kato}$.

\subsubsection{Construction d'un  $2$-cocycle}\label{constrcocycle1}

Soient $a_0,b_0,c_0,d_0\in\{1,\cdots,p^nM\}\text{ verifiant: }$
$a_0\equiv\alpha,b_0\equiv\beta,c_0\equiv\gamma,d_0\equiv\delta \mod M.$
On note les fonctions caract\'eristiques $
1_{(a_0+Mp^n\Z_p)\times(b_0+Mp^n\Z_p)} $,
$1_{(c_0+Mp^n\Z_p)\times(d_0+Mp^n\Z_p)}$ ,
 et \\ $1_{\bigl(\begin{smallmatrix}a_0&b_0\\c_0&d_0\end{smallmatrix}\bigr)+Mp^n\bM_2(\Z_p)}$
par  $\psi_{a_0,b_0}^{(n)}$, $\psi_{c_0,d_0}^{(n)}$, et
$\psi_{a_0,b_0,c_0,d_0}^{(n)}$ respectivement. Notons $U_1$ et $U_2$ respectivement les ouverts $(\alpha+M\Z_p)\times (\beta+M\Z_p)$ et $(\gamma+M\Z_p)\times (\delta+M\Z_p)$ de $\Z_p^2$, et $U=U_1\times U_2$ que l'on voit comme un ouvert de $\bM_2(\Z_p)$ et m\^eme de $\GL_2(\Z_p)$ puisque $\det A\in\Z_p^{*}$ et $p|M$.

Pour $i=1,2$, on d\'efinit les mesures alg\'ebriques
$\mu_i\in\rH^1(\cG_{\fK_M},\fD_0(U_i,\Z_p(1)))$ par les formules suivantes:
\begin{equation*}
\begin{split}
\int\psi_{a_0,b_0}^{(n)}\mu_1=\int_{(a_0+Mp^n\hat{\Z})\times(b_0+Mp^n\hat{\Z})}r_c(z_{siegel}^{(p)}), 
\int\psi_{c_0,d_0}^{(n)}\mu_2=\int_{(c_0+Mp^n\hat{\Z})\times(d_0+Mp^n\hat{\Z})}r_d(z_{siegel}^{(p)}).
\end{split}
\end{equation*}
On note $\nu=\mu_1\otimes\mu_2$ l'\'el\'ement de
$\rH^2(\cG_{\fK_M},\fD_0(U,\Z_p(2)))$. Comme on a $z_{kato,c,d}(k,j)=(e_1^{k-2}t^{-j})*x_pz_{kato,c,d}$, on a
\begin{equation}\label{ZMA}
z_{M,A}=\int\psi_{M,A}z_{kato,c,d}(k,j)=\int_{U} (e_1^{k-2}t^{-j})*x\nu.
\end{equation}

Consid\'erons le $q$-d\'eveloppement d'une unit\'e modulaire, on a une d\'ecomposition du groupe des unit\'es modulaires $\cU^{\con}(\Q^{\cycl})$:
$\cU^{\con}(\Q^{\cycl})=(\Q^{\cycl})^{*}\times q^{\Q}\times \cU_1^{\con}(\Q^{\cycl}),$
o\`u $\cU_1^{\con}(\Q^{\cycl})$ est l'ensemble d'\'el\'ements $x\in \cU^{\con}(\Q^{\cycl})$ v\'erifiant $v_q(x-1)> 0$. Alors on a une inclusion de $\cU^{\con}(\Q^{\cycl})$ dans $(\fK_{\infty}^{+}[q^{-1}])^{*}$.

Si $N$ un entier $\geq 1$, si $0\leq a,b\leq N-1$, soit $i$ un entier tel que $(i,6)=1$ et $(ia,ib)$ n'appartient pas \`a $N\Z^2$. Rappelons que l'expression explicite de $g_{i}(q,q^a_N\z^b_N)\in \cU(\Q^{\cycl})$ est comme suit ( on note $g_{i}(q,q^a_N\z^b_N)$, ce qui est not\'e $g_{i,\frac{a}{N},\frac{b}{N}}$ dans la section $\S\ref{section2}$, m\^eme chose pour la fonction $\theta$ et les s\'eries d'Eisenstein ):
\begin{equation}\label{explicite}  g_{i}(q,q^a_N\z^b_N)=q^{\frac{i^2-1}{12}}(-q_N^a\z_N^b)^{\frac{i-i^2}{2}}\frac{\prod_{n\geq 0}(1-q^nq_N^a\z_N^b)^{i^2}\prod_{n\geq 1}(1-q^n(q_N^a\z_N^b)^{-1})^{i^2}}{\prod_{n\geq 0}(1-q^n(q_N^a\z_N^b)^i)\prod_{n\geq 1}(1-q^n(q_N^{a}\z_N^b)^{-i})}   .
\end{equation}

Fixons $i$ un entier tel que $(i,6)=1$ et $i\equiv 1\mod N$. Rappelons que si $c\in\Z_p^{*}$, on a d\'efini un \'el\'ement $g_c(q,q_n^a\z^b_N)$ de $\Z_p\otimes \cU(\Q^{\cycl})\subset\Z_p\otimes (\ol{\fK}^+[q^{-1}])^*$ dans la section \ref{gcd}:
\[g_c(q,q_n^a\z^b_N)=r_c(g_{a/N,b/N})=\frac{c^2-c_1^2}{i^2-1}g_{i,a/N,b/N}+g_{c_1,a/N,b/N}\] avec $c_1\equiv c\mod pN$ et $v_p(c-c_1)$ assez grand.

Soit $g\in \ol{\fK}$; on note $\bar{g}=(g,g^{\frac{1}{p}},\cdots,)$ un rel\`evement de $g$ dans $\R(\ol{\fK})$ et $[\bar{g}]$ son repr\'esentant de Teichm\"uller  dans $\A_{\inf}(\ol{\fK})$. Alors l'\'el\'ement $\ol{g_c(q,q_n^a\z^b_N)}$ est bien d\'efini et appartient \`a $\Z_p\otimes (U_0(\ol{\fK}^+)[\bar{q}^{-1}])^{*}$, \`a multiplication d'un \'el\'ement de $(1,\z_p,\cdots,)^{\Q}$ pr\`es .

\begin{lemma}Soit $i$ un entier tel que $(i,6)=1$ et $i\equiv 1\mod N$. L'\'el\'ement
$\log [\bar{g}_{i,\frac{a}{N},\frac{b}{N}}]$ est un \'el\'ement bien d\'efini dans $\B_{\log}^+$, \`a addition d'un \'el\'ement de $\Q_p t$ pr\`es .
\end{lemma}
\begin{proof}[D\'emonstration]
Pour tout $n\geq 0$ ( resp. $n\geq 1$ ), $[\ol{1-q^nq_N^a\z_N^b}]$ ( resp. $[\ol{1-q^nq_N^{-a}\z_N^{-b}}]$ ) est un \'el\'ement inversible de $\A_{\inf}^{**}$.
Comme $q$ est de valuation $1$ pour la valuation $p$-adique sur $\ol{\fK}^+$, on a $\prod_{n\geq 0}\ol{(1-q^nq_N^a\z_N^b)}$ converge dans $U_0(\ol{\fK}^+)$. Ceci implique $\prod_{n\geq 0}[\ol{(1-q^nq_N^a\z_N^b)}]$ converge dans $\A_{\inf}^{**}$. De la m\^eme mani\`ere, on obtient que $O=\frac{\prod_{n\geq 0}[\ol{(1-q^nq_N^a\z_N^b)}]^{i^2}\prod_{n\geq 1}[\ol{(1-q^n(q_N^a\z_N^b))}^{-1}]^{i^2}}{\prod_{n\geq 0}[\ol{(1-q^n(q_N^a\z_N^b))}^i]\prod_{n\geq 1}[\ol{(1-q^n(q_N^a\z_N^b))}^{-i}]}$ converge dans $\A_{\inf}^{**}$ .
Donc l'\'el\'ement $[\ol{g_i(q,q^a_N,q^b_N)}]\in \A_{\inf}(\ol{\fK})$ est donn\'e par la formule explicite suivante:
\[ [\ol{g_i(q,q^a_N,q^b_N)}]=\tilde{q}^{\frac{i^2-1}{12}}(-\tilde{q}_N^a\tilde{\z}_N^b)^{\frac{i-i^2}{2}}\times O ;  \]
ceci implique qu'il est dans $\A_{\inf}^{**}\times\tilde{q}^{\Q}$. En appliquant l'application $\log $ ( c.f $\S \ref{anneaux3}$ ) \`a $[\bar{g}_{i,\frac{a}{N},\frac{b}{N}}]$,  on conclut le lemme.
\end{proof}

Si $c\in\Z_p^*$, on d\'efinit $\log [\bar{g}_{c,\frac{a}{N},\frac{b}{N}}]$ comme  un \'el\'ement de $\B_{\log}^+$, \`a addition d'un \'el\'ement de $\Q_p t$ pr\`es,  par la formule : $\log [\bar{g}_{c,\frac{a}{N},\frac{b}{N}}]=\frac{c^2-c^2_1}{i^2-1}\log [\bar{g}_{i\frac{a}{N},\frac{b}{N}}]+\log [\bar{g}_{c_1,\frac{a}{N},\frac{b}{N}}],$
o\`u $c_1$ est un entier tel que $v_p(\frac{c-c_1}{i^2-1})\geq 0$, $c\equiv c_1\mod N$ et $(c_1,6)=1$; ceci ne depend pas du choix de $i$ et $c_1$ \`a addition d'un \'el\'ement de $\Q_p t$ pr\`es .

Si $i=1,2$, soit $\Psi_i$ une base du
$\Z$-module des fonctions localement constantes sur $U_i$ constitu\'ee de fonctions du type $\psi_{a_0,b_0}^{(n)}$ ( resp. $\psi_{c_0,d_0}^{(n)}$ ), avec $n\in\N$ et $a_0,b_0$ (resp. $c_0,d_0$) comme ci-dessus. On d\'efinit une distribution alg\'ebrique $\mu_{1,\Psi_1}$
(resp. $\mu_{2,\Psi_2}$) sur $U_1$ (resp. $U_2$) \`a valeurs dans $\B_{\dR}^+(\overline{\fK}^+)[u_q]$ par la formule: si
$\psi_{a_0,b_0}^{(n)}\in\Psi_1$ (resp.
$\psi_{c_0,d_0}^{(n)}\in\Psi_2$), 
\[
\int\psi_{a_0,b_0}^{(n)}\mu_{1,\Psi_1}=\log[\ol{g_{c}(q, q^{a_0}_{Mp^n}\z^{b_0}_{Mp^n})}] 
( \text{ resp. }
\int\psi_{c_0,d_0}^{(n)}\mu_{2,\Psi_2}=\log[\ol{g_{d}(q, q^{c_0}_{Mp^n}\z^{d_0}_{Mp^n})}]).\]

On identifie $\Z_p(1)$ au sous-module de $\B_{\dR}^+$ via l'isomorphisme de $\Z_p[\cG_{\fK}]$-modules
\[\log\circ[\cdot]: \Z_p(1)\ra \Z_pt.\]
\begin{lemma}Si $i=1,2$, alors $\mu_i$ est l'image du $1$-cocycle
$\sigma\mapsto\mu_{i,\Psi_i}*(\sigma-1)$
dans $\rH^1(\cG_{\fK_M},\fD_0(U_i,\Z_p(1)))$ pour tout choix de $\Psi_i$ et des valeurs de l'int\'egration de $\mu_{i,\Psi_i}$.
\end{lemma}
\begin{proof}[D\'emonstration]La d\'emonstration est la m\^eme pour $i=1,2$. On peut donc supposer que $i=1$ et $\psi_{a_0,b_0}\in\Psi_1$. Rappelons que $\log[\bar{g}_{c,\frac{a_0}{Mp^n},\frac{b_0}{Mp^n}}]$ est bien d\'efini \`a $\Q_pt$ pr\`es. Comme $\sigma\mapsto t*(\sigma-1)$ est un cobord dans $\rH^1(\cG_{\fK_M},\Z_p(1))$,  $\mu_{i,\Psi_i}$ ne d\'epend du choix de $\log[\bar{g}_{c,\frac{a_0}{Mp^n},\frac{b_0}{Mp^n}}]$.

Si $\psi_{a_0,b_0}^{(n)}\in \Psi_1$,  on a $\int \psi_{a_0,b_0}^{(n)}\mu_1=g_{c,\frac{a_0}{Mp^n},\frac{b_0}{Mp^n}}$.
D'apr\`es la th\'eorie de Kummer $p$-adique (c.f. $\S \ref{section3}$ ), $\mu_1$ est repr\'esent\'e par le $1$-cocycle $\sigma\mapsto \mu*(\sigma-1)$ de $\rH^1(\Pi_{\Q}^{'},\fD_{\alg}(X_1,\Z_p(1)))$ avec $\mu\in\rH^1(\Pi_{\Q}^{'},\fD_{\alg}(X_1,Z^0))$ v\'erifiant:
$\int \psi_{a_0,b_0}^{(n)}\mu=\bar{g}_{c,\frac{a_0}{Mp^n},\frac{b_0}{Mp^n}}.$

Appliquons l'isomorphisme $\log\circ[\cdot]:\Z_p(1)\ra \Z_pt$ \`a $\mu*(\sigma-1)$; on d\'eduit le lemme de la formule $\log[\sigma x ]= \sigma\log[x]$.
\end{proof}

\subsubsection{Descente de $\overline{\fK}$ \`a $\fK_{Mp^{\infty}}$}

\begin{lemma}Si $pN|M$,  alors les produits infinis $\prod_{n\geq 0}(1-\tilde{q}^n\tilde{q}_N^a\tilde{\z}_N^b)$ et $\prod_{n\geq 1}(1-\tilde{q}^n\tilde{q}_N^{-a}\tilde{\z}_N^{-b})$ convergent dans $\A_{\inf}^{**}\cap\B_{\dR}^+(\fK^+_{Mp^{\infty}})$   .
\end{lemma}
\begin{proof}[D\'emonstration]La d\'emonstration pour  $\prod_{n\geq 0}(1-\tilde{q}^n\tilde{q}_N^a\tilde{\z}_N^b)$ et $\prod_{n\geq 1}(1-\tilde{q}^n\tilde{q}_N^{-a}\tilde{\z}_N^{-b})$ est la m\^eme. Donc on montre le lemme pour le produit $\prod_{n\geq 0}(1-\tilde{q}^n\tilde{q}_N^a\tilde{\z}_N^b)$.

Comme $q$ est de valuation $1$ pour la valuation $p$-adique dans $\ol{\fK}^+$, on a $\tilde{q}=p\alpha(1+\omega\beta)\in\A_{\inf}$, o\`u $\alpha, \beta\in\A_{\inf} $,  et $v_p(\tilde{q})\geq 1$.
Donc on a
$\prod_{n\geq 0}(1-\tilde{q}^n\tilde{q}_N^a\tilde{\z}_N^b)=1+\sum_{k=1}^{+\infty}
(-1)^k\tilde{q}_N^{ak}\z_N^{bk}c_k,$
o\`u $c_k=\sum_{0\leq i_1<i_2<\cdots<i_k}\tilde{q}^{\sum_{1\leq j\leq k}i_j}$ converge dans $\A_{\inf}$ pour la topologie $p$-adique. D'autre part, on a $v_p(c_k)\geq k(k-1)/2$ et $v_p(\tilde{q}_N^{ak}\z_N^{bk}c_k)\geq \frac{ak}{N}+\frac{k(k-1)}{2}>0$; ceci implique le produit converge pour la topologie $p$-adique dans $\A_{\inf}^{**}$.
Il est \'evident que le produit est un \'el\'ement de $\B_{\dR}^+(\fK_{Mp^{\infty}}^+)$.
\end{proof}

Si $i\in\Z$ est tel que $(i,6)=1$, alors $g_i(\tilde{q},\tilde{q}_N^a\tilde{\z}_N^b)$ appartient \`a $(\A_{\inf}^{**}\cap\B_{\dR}^+(\fK_{Mp^{\infty}}^+))\times \tilde{q}^{\Q}$. On peut appliquer l'application $\log$ \`a $g_{i}(\tilde{q},\tilde{q}_N^a\tilde{\z}_N^b)$ et d\'efinir $\log g_{c}(q,\tilde{q}^a_N\tilde{\z}_N^b)=\frac{c^2-c_1^2}{i^2-1}\log g_i(\tilde{q},\tilde{q}_N^a\tilde{\z}_N^b)+\log g_{c_1}(\tilde{q},\tilde{q}_N^a\tilde{q}_N^b)$, \`a addition d'un \'el\'ement de $\Q_pt$ pr\`es , o\`u $i$ est un entier tel que $(i,6)=1$ et $i\equiv 1\mod N$, et $c_1$ est un entier tel que $v_p(c-c_1)$  soit assez grand.

\begin{lemma}\label{appl}
Si $n\geq 0($ resp. $n\geq 1)$, $1-\tilde{q}^n\tilde{q}_N^a\tilde{\z}_N^b($ resp. $1-\tilde{q}^n\tilde{q}_N^{-a}\tilde{\z}_N^{-b})$ est inversible dans $\A_{\inf}(\ol{\fK}^+)$.
De plus, on a que $\frac{[\ol{1-q^nq_N^a\z_N^b}]}{1-\tilde{q}^n\tilde{q}_N^a\tilde{\z}_N^b}$ (resp. $\frac{[\ol{1-q^nq_N^{-a}\z_N^{-b}}]}{1-\tilde{q}^n\tilde{q}_N^{-a}\tilde{\z}_N^{-b}}$ ) est un \'el\'ement dans $1+\omega\A_{\inf}(\ol{\fK}^+)$.
\end{lemma}
\begin{proof}[D\'emonstration]On montrera le lemme pour le cas $n\geq 0$ et l'autre cas se d\'eduit de la m\^eme mani\`ere.
Comme $q$ est de valuation $1$ pour la valuation $p$-adique dans $\ol{\fK}^+$, $\tilde{q}$ peut s'\'ecrire sous la forme $p\beta+\omega \alpha$ avec $\alpha,\beta\in\A_{\inf}(\ol{\fK}^+)$. Donc on a 
$
\sum_{m=0}^{+\infty}(\tilde{q}^n\tilde{q}_N^a\tilde{\z}_N^b)^m
=1+\sum_{k=0}^{+\infty}\omega^k\sum\limits_{mn\geq k, m\geq 1}\binom{mn}{k}\alpha^k(p\beta)^{mn-k}\tilde{q}_N^{ma}\tilde{\z}_N^{mb}.
$
La s\'erie $\sum\limits_{mn\geq k, m\geq 1}\binom{mn}{k}\alpha^k(p\beta)^{mn-k}\tilde{q}_N^{ma}\tilde{\z}_N^{mb}$ converge dans $\A_{\inf}(\ol{\fK}^+)$ et donc $1+\sum_{m=1}^{+\infty}(\tilde{q}^n\tilde{q}_N^a\tilde{\z}_N^b)^m$ converge dans $\A_{\inf}(\ol{\fK}^+)$ pour la topologie faible. Ceci donne l'inverse de $1-\tilde{q}^n\tilde{q}^a_N\tilde{\z}_N^b$ dans $\A_{\inf}(\ol{\fK}^+)$ et implique que $\frac{[\ol{1-q^nq_N^a\z_N^b}]}{1-\tilde{q}^n\tilde{q}_N^a\tilde{\z}_N^b}$ est un \'el\'ement dans $\A_{\inf}(\ol{\fK}^{+})$. Par ailleurs, $\theta(\frac{[\ol{1-q^nq_N^a\z_N^b}]}{1-\tilde{q}^n\tilde{q}_N^a\tilde{\z}_N^b})$ a pour image $1$ dans $\C(\ol{\fK}^+)$, ce qui permet de conclure.
\end{proof}

Si $i=1,2$ et si $c, d\in \Z_p^{*}$, on d\'efinit une distribution alg\'ebrique $\tilde{\mu}_{i,\Psi_i}$ sur $U_i$ par la formule: pour
$\psi_{a_0,b_0}^{(n)}\in\Psi_1$ (resp.
$\psi_{c_0,d_0}^{(n)}\in\Psi_2$ ),
\begin{equation*}
\begin{split}
\int\psi_{a_0,b_0}^{(n)}\tilde{\mu}_{1,\Psi_1}=\log g_c(\tilde{q},\tilde{q}_{Mp^n}^{a_0}\tilde{\z}_{Mp^n}^{b_0}) 
(\text{resp.}\int\psi_{c_0,d_0}^{(n)}\tilde{\mu}_{2,\Psi_2}=\log
g_d(\tilde{q},\tilde{q}_{Mp^n}^{c_0}\tilde{\z}_{Mp^n}^{d_0})).
\end{split}
\end{equation*}
De la formule explicite $(\ref{explicite})$ de $g_{c}(q,q_N^a\z_N^b)$, on obtient  que $g_c(\tilde{q},\tilde{q}_{Mp^n}^{a_0}\tilde{\z}_{Mp^n}^{b_0})$ appartient \`a $\B_{\dR}^+(\fK_{Mp^{\infty}}^+)\oplus \Q_p\tilde{q}$, et son image par $\theta$ dans $\C(\fK_{Mp^{\infty}}^+)\oplus \Q_p q$ est $g_{c}(q,q_{Mp^n}^{a_0}\z_{Mp^n}^{b_0})$.
On en tire que $\log g_c(\tilde{q},\tilde{q}_{Mp^n}^{a_0}\tilde{\z}_{Mp^n}^{b_0})$ et $\log g_d(\tilde{q},\tilde{q}_{Mp^n}^{c_0}\tilde{\z}_{Mp^n}^{d_0})$  appartiennent \`a $\B_{\dR}^+(\fK_{Mp^{\infty}}^+)\oplus\Q_pu_q$.

\begin{lemma}Si $i=1,2$, alors:
\begin{enumerate}
\item[$(1)$] $\tilde{\mu}_{i,\Psi_i}-\mu_{i,\Psi_i}$ est une mesure \`a valeurs dans $t\B_{\dR}^{+}$;
\item[$(2)$] Consid\'erons l'application: $\rH^1(\cG_{\fK_M},\fD_0(U_i,\Z_p(1)))\ra\rH^1(\cG_{\fK_M},\fD_0(U_i,t\B_{\dR}^{+}))$. L'image de $\mu_i$ est repr\'esent\'e par
le cocycle
\[\sigma\mapsto\tilde{\mu}_{i,\Psi_i}*(\sigma-1),\]
qui est l'inflation d'un cocycle sur $P_{\fK_M}$ \`a valeurs dans $\fD_0(U_i,t\B_{\dR}^{+}(\fK_{Mp^{\infty}}^+))$.
\end{enumerate}
\end{lemma}
\begin{proof}[D\'emonstration]Le $(1)$ se d\'eduit du lemme $\ref{appl}$. Le $(2)$ est une cons\'equence imm\'ediate du $(1)$. En effet, comme $(\tilde{\mu}_{i,\Psi_i}-\mu_{i,\Psi_i})$ est une mesure \`a valeurs dans $t\B_{\dR}^{+}$, on obtient que $(\tilde{\mu}_{i,\Psi_i}-\mu_{i,\Psi_i})*(\sigma-1)$ est un cobord 
dans $\rH^1(\cG_{\fK_M}, \fD_0(U_i,t\B_{\dR}^{+}))$. D'autre part, on a $\tilde{\mu}_{i,\Psi_i}*(\sigma-1)=\mu_{i,\Psi_i}*(\sigma-1)+(\tilde{\mu}_{i,\Psi_i}-\mu_{i,\Psi_i})*(\sigma-1),$
 et donc $\sigma\mapsto \tilde{\mu}_{i,\Psi_i}*(\sigma-1)$
est un $1$-cocycle \`a valeurs dans $t\B_{\dR}^{+}$ qui repr\'esente $\mu_i\in \rH^1(\cG_{\fK_M},\fD_0(U_i,t\B_{\dR}^+))$.

\end{proof}

Soient $\Lambda_1, \Lambda_2$ deux $G$-modules \`a droite. Si
$x_1\in\Lambda_1, x_2\in\Lambda_2$ et $\sigma,\tau\in G$, on d\'efinit un
\'el\'ement $\{x_1\otimes
x_{2}\}_{\sigma,\tau}:=(x_1*(\tau\sigma-\sigma)\otimes(x_2*(\sigma-1)))\in\Lambda_1\otimes\Lambda_2$.
\begin{coro}
Consid\'erons l'application:
 \[\rH^2(\cG_{\fK_M},\fD_0(U,\Z_p(2)))\ra\rH^2(\cG_{\fK_M},\fD_0(U,t^2\B_{\dR}^{+}(\overline{\fK}^+))).\]
L'image de $\nu=\mu_1\otimes\mu_2$ peut \^etre repr\'esent\'e par le
$2$-cocycle
\[(\sigma,\tau)\mapsto\{\tilde{\mu}_{1,\Psi_1}\otimes\tilde{\mu}_{1,\Psi_2}\}_{\sigma,\tau},\] qui est l'inflation du $2$-cocycle sur $P_{\fK_M}$ \`a valeurs dans $\fD_0(U,t^2\B_{\dR}^{+}(\fK_{Mp^{\infty}}^+))$.
\end{coro}
\begin{proof}[D\'emonstration]Ce corollaire r\'esulte du lemme pr\'ec\'edent et de la formule du cup-produit. La formule du cup-produit de deux cocycles est:
soient $G$ un groupe, $A,B$ deux $G$-module \`a droite. Si
$f_1\in\rH^i(G,A)$ et $f_2\in\rH^j(G,B)$, on a $f_1\cup
f_2\in\rH^{i+j}(G,A\otimes B)$ o\`u
\[f_1\cup
f_2(\sigma_{i+j},\cdots,\sigma_{i+1},\sigma_{i},\cdots,\sigma_{1})=(f_1(\sigma_i,\cdots,\sigma_1)*\sigma_{i+1}*\cdots
*\sigma_{i+j})\otimes f_2(\sigma_{i+j},\cdots,\sigma_{i+1}).\]
On applique cette formule \`a $f_1=\tilde{\mu}_{1,\Psi_1}*(\sigma-1)$ et $f_2=\tilde{\mu}_{2,\Psi_2}*(\sigma-1)$ et on obtient un $2$-cocycle $(\sigma,\tau)\mapsto\{\tilde{\mu}_{1,\Psi_1}\otimes\tilde{\mu}_{1,\Psi_2}\}_{\sigma,\tau},$
qui repr\'esente $\nu$ et \`a valeurs dans $\fD_0(U,t^2\B_{\dR}^+(\fK_{Mp^{\infty}}^+))$. Donc il est l'inflation du $2$-cocycle sur $P_{\fK_M}$ \`a valeurs dans $\fD_0(U,t^2\B_{\dR}^+(\fK_{Mp^{\infty}}^+))$.
\end{proof}

\subsubsection{Descente de $\fK_{Mp^{\infty}}$ \`a $\fK_M$}
\begin{defn} Si $v_p(M)\geq v_p(2p)$, on pose $\B_{\log}^+(\fK_{Mp^{\infty}}^+)=\B_{\dR}^+(\fK_{Mp^{\infty}}^+)[u_q]$, o\`u $u_q=\log\tilde{q}$. On d\'efinit une application $\tilde{\fK}_M^+[u_q]$-lin\'eaire:
\begin{equation*}
\begin{split}
\bR_{M,\log}: \B_{\log}^+(\fK_{Mp^{\infty}}^+)\longrightarrow\tilde{\fK}^+_M[u_q]; 
\sum_{k}x_ku_q^k \mapsto\sum_{k}\bR_M(x_k)u_q^k,
\end{split}
\end{equation*}
o\`u $\bR_M$ est la trace de Tate normalis\'ee sur $\B_{\dR}^+(\fK_{Mp^{\infty}}^+)$ \`a valeurs dans $\tilde{\fK}_M^+$.
\end{defn}
Comme $\tilde{\fK}_M^+[u_q]$ est stable sous l'action de $P_{\fK_M}$ et $\bR_M$ commute avec l'action de $P_{\fK_M}$, on d\'eduit que $\bR_{M,\log}$ commute avec l'action de $P_{\fK_M}$.  Notons $\bR_{M,\log}$ encore par $\bR_M$.

 D'apr\`es ce qui pr\'ec\`ede et de la relation $(\ref{ZMA})$, $z_{M,A}$ est la classe du $2$-cocycle
\[(\sigma,\tau)\mapsto\int_{U}((e_1^{k-2}t^{-j})*g)\{\tilde{\mu}_{1,\Psi_1}\otimes\tilde{\mu}_{2,\Psi_2}\}_{\sigma,\tau},\]
qui est aussi la classe du $2$-cocycle par ``la  trace de Tate normalis\'ee "
\begin{equation*}
\begin{split}
(\sigma,\tau)\mapsto \bR_M(\int_U\bigl((e_1^{k-2}t^{-j})*g)\{\tilde{\mu}_{1,\Psi_1}\otimes\tilde{\mu}_{2,\Psi_2}\}_{\sigma,\tau}\bigr).
\end{split}
\end{equation*}

\begin{lemma}Si $ad-bc\in\Z_p^{*}$, alors
\begin{equation*}
\begin{split}
&\bR_M(\log\theta(\tilde{q},\tilde{q}_{Mp^n}^a\tilde{\z}_{Mp^n}^b)\log\theta(\tilde{q},\tilde{q}_{Mp^n}^c\tilde{\z}_{Mp^n}^d))\\
&=p^{-2n}\log\theta(\tilde{q}^{p^n},\tilde{q}_M^a\tilde{\z}_M^b)\log\theta(\tilde{q}^{p^n},\tilde{q}_M^c\tilde{\z}_M^d).
\end{split}
\end{equation*}
\end{lemma}
\begin{proof}[D\'emonstration]
On a la formule explicite pour $\log\theta(\tilde{q},\tilde{q}_{Mp^n}^a\tilde{\z}_{Mp^n}^b)$:
\begin{equation*}
\begin{split}
\log\theta(\tilde{q},\tilde{q}_{Mp^n}^a\tilde{\z}_{Mp^n}^b)=&(\frac{1}{12}+\frac{a}{2Mp^n})u_q+\frac{b}{2Mpn}t+\\
&+(\sum_{m=1}^{+\infty}\log(1-\tilde{q}^m\tilde{q}_{Mp^n}^a\tilde{\z}_{Mp^n}^{b})+\sum_{m=0}^{\infty}\log(1-\tilde{q}^m\tilde{q}_{Mp^n}^{-a}\tilde{\z}_{Mp^n}^{-b})).
\end{split}
\end{equation*}
Comme  $\bR_M(u_q)=u_q=p^{-n}u_q^{p^n}$ et $\bR_M(t)=t=p^{-n}\log \tilde{\z}^{p^n}$, il suffit de v\'erifier les relations suivantes:
\begin{equation*}
\begin{split}
\bR_M(\log (1-\tilde{q}^m\tilde{q}_{Mp^n}^a\tilde{\z}_{Mp^n}^{b}))&= p^{-n}\log(1-\tilde{q}^{p^nm}\tilde{q}_{M}^a\tilde{\z}_{M}^{b})\\
\bR_M(\log (1-\tilde{q}^{m_1}\tilde{q}_{Mp^n}^a\tilde{\z}_{Mp^n}^{b}) \log (1-\tilde{q}^{m_2}\tilde{q}_{Mp^n}^c\tilde{\z}_{Mp^n}^{d}))&= p^{-2n}\log(1-\tilde{q}^{p^nm_1}\tilde{q}_{M}^a\tilde{\z}_{M}^{b})\log(1-\tilde{q}^{p^nm_2}\tilde{q}_{M}^c\tilde{\z}_{M}^{d})
\end{split}
\end{equation*}

Comme $ad-bc\in\Z_p^{*}$, $a$ et $b$ ne sont pas divisibles par
$p^n$ \`a la fois. Donc pour $p^n\nmid i$, on obtient
$\bR_M((\tilde{q}^m\tilde{q}_{Mp^n}^a\tilde{\z}_{Mp^n}^b)^i)=0$.
On en d\'eduit que
\begin{equation*}
\bR_M(\log(1-\tilde{q}^m\tilde{q}_{Mp^n}^a\tilde{\z}_{Mp^n}^{b}))=\bR_M(-\sum\limits_{i=1}^{\infty}\frac{(\tilde{q}^m\tilde{q}_{Mp^n}^a\tilde{\z}_{Mp^n}^b)^i}{i})=p^{-n}\log(1-(\tilde{q}^{p^n})^m\tilde{q}_{M}^a\tilde{\z}_{M}^{b}).
\end{equation*}
 Pour le terme
\[(\log(1-\tilde{q}^{m_1}\tilde{q}_{Mp^n}^a\tilde{\z}_{Mp^n}^{b}))(\log(1-\tilde{q}^{m_2}\tilde{q}_{Mp^n}^c\tilde{\z}_{Mp^n}^{d})),\]
on d\'eveloppe le produit:
\begin{equation*}
\begin{split}
&(\log(1-\tilde{q}^{m_1}\tilde{q}_{Mp^n}^a\tilde{\z}_{Mp^n}^{b}))(\log(1-\tilde{q}^{m_2}\tilde{q}_{Mp^n}^c\tilde{\z}_{Mp^n}^{d}))\\
=&(\sum\limits_{i=1}^{+\infty}-\frac{(\tilde{q}^{m_1}\tilde{q}_{Mp^n}^a\tilde{\z}_{Mp^n}^{b})^i}{i})(\sum\limits_{j=1}^{+\infty}-\frac{(\tilde{q}^{m_2}\tilde{q}_{Mp^n}^c\tilde{\z}_{Mp^n}^{d})^j}{j})\\
=&\sum\limits_{i,j=1}^{+\infty}\frac{(\tilde{q}^{m_1}\tilde{q}_{Mp^n}^a\tilde{\z}_{Mp^n}^{b})^i(\tilde{q}^{m_2}\tilde{q}_{Mp^n}^c\tilde{\z}_{Mp^n}^{d})^j}{ij}
\end{split}
\end{equation*}
Comme $ad-bc\in\Z_p^{*}$, cela \'equivaut \`a dire que ``$p^n|ai+cj$
et $p^n|bi+dj$ sont \'equivalent \`a $p^n|i$ et $p^n|j$". Donc en
appliquant $\bR_M$ \`a la formule ci-dessus, on obtient:
\begin{equation*}
\begin{split}
&\bR_M\left((\log(1-\tilde{q}^{m_1}\tilde{q}_{Mp^n}^a\tilde{\z}_{Mp^n}^{b}))(\log(1-\tilde{q}^{m_2}\tilde{q}_{Mp^n}^c\tilde{\z}_{Mp^n}^{d}))\right)\\
=&p^{-2n}\sum\limits_{i,j=1}^{+\infty}\frac{(\tilde{q}^{p^nm_1}\tilde{q}_{M}^a\tilde{\z}_{M}^{b})^i(\tilde{q}^{p^nm_2}\tilde{q}_{M}^c\tilde{\z}_{M}^{d})^j}{ij}\\
=&p^{-2n}(\log(1-\tilde{q}^{p^nm_1}\tilde{q}_{M}^a\tilde{\z}_{M}^{b}))(\log(1-\tilde{q}^{p^nm_2}\tilde{q}_{M}^c\tilde{\z}_{M}^{d}))
\end{split}
\end{equation*}

En composant les r\'esultats ci-dessus, on obtient la formule
voulue dans le lemme.
\end{proof}

\begin{coro}\label{ZM}
Si on note
\[\log_{\left(\begin{smallmatrix}a_0&b_0\\c_0&d_0\end{smallmatrix}\right)}^{(\sigma,\tau)}=\left\{\log\left((c^2-<c>)\theta(\tilde{q}^{p^n},\tilde{q}_M^{a_0}\tilde{\z}_M^{b_0})\right)
\otimes\log\left((d^2-<d>)\theta(\tilde{q}^{p^n},\tilde{q}^{c_0}_M\tilde{\z}_M^{d_0})\right)\right\}_{\sigma,\tau},\]
 $z_{M,A}$ peut se repr\'esenter par le $2$-cocycle
\begin{equation*}
\begin{split}
(\sigma,\tau)\mapsto \lim_{n\ra+\infty}p^{-2n}
\sum\frac{(a_0e_1+b_0e_2)^{k-2}}{(a_0d_0-b_0c_0)^jt^j}\log_{\left(\begin{smallmatrix}a_0&b_0\\c_0&d_0\end{smallmatrix}\right)}^{(\sigma,\tau)},
\end{split}
\end{equation*}
 la somme portant sur l'ensemble
\[U^{(n)}:=\{(a_0,b_0,c_0,d_0)\in\{1,\cdots,Mp^n\}^4|a_0\equiv\alpha, b_0\equiv\beta, c_0\equiv\gamma, d_0\equiv\delta\mod M\}.\]
\end{coro}
\begin{proof}[D\'emonstration]
$z_{M,A}$ peut se repr\'esenter par le $2$-cocycle
\[(\sigma,\tau)\mapsto \bR_M( \int_U((e_1^{k-2}t^{-j})*g)\{\tilde{\mu}_{1,\Psi_1}\otimes\tilde{\mu}_{2,\Psi_2}\}_{\sigma,\tau}).\]
Par la d\'efinition de l'int\'egration sur $U$, on a
\begin{equation}\label{231}
\begin{split}&\bR_M(\int_U((e_1^{k-2}t^{-j})*g)\{\tilde{\mu}_{1,\Psi_1}\otimes\tilde{\mu}_{2,\Psi_2}\}_{\sigma,\tau})\\
=&\lim_{n\ra\infty}\sum_{g=\bigl(\begin{smallmatrix}a_0&b_0\\c_0&d_0\end{smallmatrix}\bigr)\in U^{(n)}}\bR_M(\int\psi_{a_0,b_0}^{(n)}\otimes\psi_{c_0,d_0}^{(n)}((e_1^{k-2}t^{-j})*g)\{\tilde{\mu}_{1,\Psi_1}\otimes\tilde{\mu}_{2,\Psi_2}\}_{\sigma,\tau})\\
=&\lim_{n\ra\infty}\sum_{g=\bigl(\begin{smallmatrix}a_0&b_0\\c_0&d_0\end{smallmatrix}\bigr)\in U^{(n)}}((e_1^{k-2}t^{-j})*g)\cdot r_{c,d}\left( \bR_M\left(\{\log\theta(\tilde{q},\tilde{q}_{Mp^n}^{a_0}\tilde{\z}_{Mp^n}^{b_0})\otimes\log\theta(\tilde{q},\tilde{q}_{Mp^n}^{c_0}\tilde{\z}_{Mp^n}^{d_0})\}_{\sigma,\tau}\right)\right)
\end{split}
\end{equation}

Comme la trace de Tate normalis\'ee commute avec l'action de $P_{\fK_M}$, on a
\begin{equation*}
\begin{split}
&\bR_M\{\log\theta(\tilde{q},\tilde{q}_{Mp^n}^{a_0}\tilde{\z}_{Mp^n}^{b_0})\otimes\log\theta(\tilde{q},\tilde{q}_{Mp^n}^{c_0}\tilde{\z}_{Mp^n}^{d_0})\}_{\sigma,\tau}\\
=&\{\bR_M(\log\theta(\tilde{q},\tilde{q}_{Mp^n}^{a_0}\tilde{\z}_{Mp^n}^{b_0})\otimes\log\theta(\tilde{q},\tilde{q}_{Mp^n}^{c_0}\tilde{\z}_{Mp^n}^{d_0}))\}_{\sigma,\tau}.
\end{split}
\end{equation*}

D'apr\`es le lemme pr\'ec\'edent, on a
\begin{equation*}
\begin{split}
(\ref{231})=&\lim_{n\ra\infty}\sum_{g=\bigl(\begin{smallmatrix}a_0&b_0\\c_0&d_0\end{smallmatrix}\bigr)\in U^{(n)}}((e_1^{k-2}t^{-j})*g)\cdot r_{c,d}\left(p^{-2n}\{\log\theta(\tilde{q}^{p^n},\tilde{q}_M^{a_0}\tilde{\z}_M^{b_0})\otimes\log\theta(\tilde{q}^{p^n},\tilde{q}_M^{c_0}\tilde{\z}_M^{d_0})\}_{\sigma,\tau}\right) \\
=&\lim_{n\ra\infty}p^{-2n}\sum_{g=\bigl(\begin{smallmatrix}a_0&b_0\\c_0&d_0\end{smallmatrix}\bigr)\in
U^{(n)}}\frac{(a_0e_1+b_0e_2)^{k-2}}{(a_0d_0-b_0c_0)^jt^j}\cdot \log_{\left(\begin{smallmatrix}a_0&b_0\\c_0&d_0\end{smallmatrix}\right)}^{(\sigma,\tau)}.
\end{split}
\end{equation*}
\end{proof}

\subsubsection{Passage \`a l'alg\`ebre de Lie }

Comme le $2$-cocycle $(\sigma,\tau)\mapsto \lim_{n\ra+\infty}p^{-2n}
\sum\frac{(a_0e_1+b_0e_2)^{k-2}}{(a_0d_0-b_0c_0)^jt^j}\log_{\left(\begin{smallmatrix}a_0&b_0\\c_0&d_0\end{smallmatrix}\right)}^{(\sigma,\tau)}$ obtenu dans la section pr\'ec\'edente est la
 limite de $2$-cocycles analytiques \`a valeurs dans $\cK^+_M\otimes V_{k,j}$,  on peut utiliser les techniques diff\'erentielles
dans la section (\ref{tech}) pour calculer son image dans $\cK^+_M$ par l'application $\exp^*_{Kato}$.

Si $f(x_1,x_2)$ est une fonction en deux variables,
on note $D_1$ ( resp. $D_2$ ) l'op\'erateur $x_1\frac{d}{dx_1}$ (
resp. $x_2\frac{d}{dx_2}$ ). Si $n\in\N$ et $a,b\in\Z$, on pose
$f_{a,b}^{(n)}=f(\tilde{q}^{p^n},\tilde{q}_M^a\tilde{\z}_M^b)$.
\begin{lemma}\label{2co}
On note $\delta^{(2)}_{a_0,b_0,c_0,d_0}=\delta^{(2)}\left(
\left\{\log \left(r_c\theta_{a_0,b_0}^{(n)}\right)\otimes\log\left(r_d\theta_{c_0,d_0}^{(n)}\right)\right\}_{\sigma,\tau}\right)$.
 On a
\begin{equation*}
\begin{split}
\delta^{(2)}_{a_0,b_0,c_0,d_0}=\frac{(a_0d_0-b_0c_0)t^2}{M^2}\cdot
D_2\log\left(r_c\theta_{a_0,b_0}^{(n)}\right)\cdot
D_2\log\left(r_d\theta_{c_0,d_0}^{(n)}\right).
\end{split}
\end{equation*}
\end{lemma}
\begin{proof}[D\'emonstration]Soit $f_{a,b}^{(n)}$ d\'efinie comme ci-dessus. Alors
 l'action habituelle de $(u,v)\in P_m$ sur $f_{a,b}^{(n)}$ est:
\[(u,v)f_{a,b}^{(n)}=f(\tilde{q}^{p^n},\tilde{q}_M^a\tilde{\z}_M^{au+be^v}).\]
Du d\'eveloppement limit\'e du terme de droite en $u$ et $v$, on a
\begin{equation*}
f(\tilde{q}^{p^n},\tilde{q}_M^a\tilde{\z}_M^{au+be^v})=f(\tilde{q}^{p^n},\tilde{q}_M^a\tilde{\z}_M^{b})+u(\frac{at}{M}D_2f_{a,b}^{(n)})+v(\frac{bt}{M}D_2f_{a,b}^{(n)})+O((u,v)^2)
\end{equation*}
et donc l'action de
$(\begin{smallmatrix}1&u\\0&e^v\end{smallmatrix})-1$ sur $f_{a,b}^{(n)}$ est donn\'ee par la formule :
\begin{equation}\label{action}
f_{a,b}^{(n)}*((\begin{smallmatrix}1&u\\0&e^v\end{smallmatrix})-1)=\frac{au+bv}{M}tD_2f_{a,b}^{(n)}+O((u,v)^2).
\end{equation}

On note $\Delta= \left\{\log \left(r_c\theta_{a_0,b_0}^{(n)}\right)
\otimes\log\left(r_d\theta_{c_0,d_0}^{(n)}\right)\right\}_{\sigma,\tau}
$. Soient $\sigma=(u,v)$,$\tau=(x,y)\in \rP_m$. Par un calcul explicite de
\[ f^{(n)}_{a_0,b_0}*\left((\begin{smallmatrix}1& ue^y+x\\0& e^{v+y}\end{smallmatrix})-(\begin{smallmatrix}1& x\\ 0& e^y\end{smallmatrix})\right) \otimes g^{(n)}_{c_0,d_0}*\left((\begin{smallmatrix}1&x\\ 0&e^y\end{smallmatrix})-1\right)\]
\`a $O((u,v,x,y)^3)$ pr\`es, en utilisant la formule que l'on vient juste d'\'etablir, on a:
\begin{equation*}
\begin{split}
\Delta=\frac{a_0x(c_0u+d_0v)+c_0y (b_0u+d_0v)}{M^2}t^2
D_2\log\left(r_c\theta_{a_0,b_0}^{(n)}\right)\cdot
D_2\log\left(r_d\theta_{c_0,d_0}^{(n)}\right).
\end{split}
\end{equation*}
Alors, par la d\'efinition de l'application $\delta^{(2)}$, on a
\begin{equation*}
\begin{split}
\delta^{(2)}_{a_0,b_0,c_0,d_0}
=\frac{(b_0c_0-a_0d_0)t^2}{M^2}\cdot
D_2\log\left(r_c\theta_{a_0,b_0}^{(n)}\right)\cdot
D_2\log\left(r_d\theta_{c_0,d_0}^{(n)}\right)
\end{split}
\end{equation*}
\end{proof}

\begin{lemma}\label{negligible} Si $s\geq 2-j$ et $a,b,c,d\in\Z$, alors, dans $(\tilde{\cK}^+_M\otimes V_{k,j})/(\partial_1,1-\partial_2)$, on a 
\[(ae_1+be_2)^{k-2}t^sf_{a,b}^{(n)}\cdot g_{c,d}^{(n)}=\frac{(ae_1+be_2)^{k-2}}{a(1-s)}\cdot (ad-bc)\cdot\frac{t}{M}f_{a,b}^{(n)}\cdot D_2g_{c,d}^{(n)}.\]
\end{lemma}
\begin{proof}[D\'emonstration]
Pour $(u,v)\in \rP_m$, on a
$: f_{a,b}^{(n)}*((u,v)-1)=u\partial_1f_{a,b}^{(n)}+v\partial_2f_{a,b}^{(n)}+O((u,v)^2). $
En comparant avec la formule (\ref{action}), on obtient que $\partial_1f_{a,b}^{(n)}=\frac{at}{M}D_2f_{a,b}^{(n)}$ et $\partial_2f_{a,b}^{(n)}=\frac{bt}{M}D_2f_{a,b}^{(n)}$. En composant les formules dans la d\'emonstration de la proposition (\ref{reskj}), on obtient:
\begin{equation*}
\begin{split}
&\left(a\left(1-\partial_2\right)+b\partial_1\right)\left(\left(ae_1+be_2\right)^{k-2}t^sf_{a,b}^{(n)}g_{c,d}^{(n)}\right)\\
=&a(1-s)(ae_1+be_2)^{k-2}t^sf_{a,b}^{(n)}g_{c,d}^{(n)}+\frac{(bc-ad)t}{M}(ae_1+be_2)^{k-2}t^sf_{a,b}^{(n)}D_2g_{c,d}^{(n)},
\end{split}
\end{equation*}
qui est nul dans $(\tilde{\fK}\otimes
V_{k,j})/(\partial_1,1-\partial_2)$ pour $s\geq 2-j$. Donc:
\[(ae_1+be_2)^{k-2}t^sf_{a,b}^{(n)}g_{c,d}^{(n)}=\frac{(ad-bc)}{a(1-s)}\frac{t}{M}(ae_1+be_2)^{k-2}t^sf_{a,b}^{(n)}D_2g_{c,d}^{(n)}.\]
\end{proof}

On a d\'efini une application $\res_{k,j}: \tilde{\cK}_M^+\otimes V_{k,j}\ra\cK_M^+$ dans la recette \ref{res}. Elle permet de calculer l'image de $z_{M,A}$ dans $\cK_M^+$.
\begin{coro}On a
\begin{equation*}
\begin{split}
\res_{k,j}
\left(\frac{(a_0e_1+b_0e_2)^{k-2}}{(a_0d_0-b_0c_0)^jt^j}\cdot\delta^{(2)}_{a_0,b_0,c_0,d_0}\right)=M^{-1-j}\cdot
\frac{a_0^{k-1-j}}{(j-1)!}\cdot D_2\log(r_c\theta_{a_0,b_0}^{(n)})\cdot
D_2^j\log(r_d\theta_{c_0,d_0}^{(n)}).
\end{split}
\end{equation*}
\end{coro}
\begin{proof}[D\'emonstration]
D'apr\`es les lemmes (\ref{2co}) et (\ref{negligible}), on obtient la formule suivante, 
\begin{equation*}\label{equation}
\begin{split}
\frac{(a_0e_1+b_0e_2)^{k-2}}{(a_0d_0-b_0c_0)^jt^j}\cdot\delta^{(2)}_{a_0,b_0,c_0,d_0}
&=\frac{(a_0e_1+b_0e_2)^{k-2}t^{2-j}}{(a_0d_0-b_0c_0)^{j-1}M^2}\cdot
D_2\log\left(r_c\theta_{a_0,b_0}^{(n)}\right)\cdot
D_2\log\left(r_d\theta_{c_0,d_0}^{(n)}\right)\\
&=M^{-1-j}\frac{(a_0e_1+b_0e_2)^{k-2}t}{a_0^{j-1}(j-1)!}\cdot D_2\log\left(r_c\theta_{a_0,b_0}^{(n)}\right)\cdot
D_2^j\log\left(r_d\theta_{c_0,d_0}^{(n)}\right).
\end{split}
\end{equation*}

Par la definition de $\res_{k,j}$, on a:
\begin{equation*}
\begin{split}
&\res_{k,j}\left(\frac{(a_0e_1+b_0e_2)^{k-2}}{(a_0d_0-b_0c_0)^jt^j}\cdot\delta^{(2)}_{a_0,b_0,c_0,d_0}\right)
=\frac{M^{-1-j}a_0^{k-1-j}}{(j-1)!}\cdot
D_2\log\left(r_c\theta_{a_0,b_0}^{(n)}\right)\cdot
D_2^j\log\left(r_d\theta_{c_0,d_0}^{(n)}\right).
\end{split}
\end{equation*}
\end{proof}
Par les propri\'et\'es des fonctions $E^{(k)}_{c,\alpha,\beta}$  dans le lemme \ref{partialle},
on a \[D_2^r\log r_c\theta(x_1,x_2)=c^2E_r(x_1,x_2)-c^{r}E_r(x_1,x_2^c).\]

On note $E_{c,r}(x_1,x_2)=c^2E_r(x_1,x_2)-c^{r}E_r(x_1,x_2^c)$.
Si $b\equiv\beta\mod M$ et $d\equiv\delta\mod M$, on a
$\z_M^b=\z_M^\beta, \z_M^d=\z_M^\delta$. Donc par le corollaire
(\ref{ZM}) et le calcul que nous avons fait, on obtient:
\begin{equation*}
\begin{split}
\exp^{*}_{\mathbf{Kato}}(z_{M,A})=\frac{M^{-1-j}}{(j-1)!}\cdot\lim\limits_{n\ra\infty}\sum_{\substack{a_0\equiv\alpha [M]\\
c_0\equiv\gamma[M]\\ 1\leq a_0,c_0\leq Mp^n}}
a_0^{k-1-j}E_{c,1}(q^{p^n},q_M^{a_0}\z_M^{\beta})E_{d,j}(q^{p^n},q_M^{c_0}\z_M^{\delta}).
\end{split}
\end{equation*}

Enfin, on utilise le lemme suivant pour terminer le calcul:

\begin{lemma}\label{final}
Si $1\leq r\in\N$, et si $c,d\in\Z_p^{*}$, on a:
\begin{itemize}
\item[(1)] $\sum\limits_{\substack{c_0\equiv \gamma [M]\\ 1\leq c_0\leq Mp^n }}E_j(q^{p^n},q_M^{c_0}\z_M^{\delta})=E_{\gamma/M,\delta/M}^{(j)}$ et 
$\sum\limits_{\substack{c_0\equiv \gamma [M]\\ 1\leq c_0\leq Mp^n }}E_{d,j}(q^{p^n},q_M^{c_0}\z_M^{\delta})=E_{d,\gamma/M,\delta/M}^{(j)}.$
\item[(2)]$ \lim\limits_{n\ra\infty}\sum_{\substack{a_0\equiv\alpha[M]\\ 1\leq a_0\leq Mp^n}}a_0^rE_{c,1}(q^{p^n},q_M^{a_0}\z_M^\beta)=M^rF_{c,\alpha/M,\beta/M}^{(r+1)}.$
\end{itemize}
\end{lemma}

\begin{proof}[D\'emonstration]
$(1)$ La deuxi\`eme assertion suit de la premi\`ere, et la premi\`ere est un calcul direct par d\'efinition.\\
$(2)$ On d\'emontre l'assertion par prouver que ces deux formes modulaires $p$-adiques ont le m\^eme $q$-d\'eveloppement. 
D'apr\`es la proposition (\ref{q-deve}), on a:
 \begin{equation*}
\begin{split}
\lim\limits_{n\ra\infty}\sum_{\substack{a_0\equiv\alpha[M]\\ 1\leq
a_0\leq Mp^n}}a_0^rE_{c,1}(q^{p^n},q_M^{a_0}\z_M^\beta)=
\lim\limits_{n\ra\infty}\sum_{\substack{a_0\equiv\alpha[M]\\ 1\leq
a_0\leq Mp^n}}a_0^rF^{(1)}_{c,a_0/Mp^n,\beta/M}(q^{p^n}).
\end{split}
\end{equation*}
Notons la limite \`a gauche par $(\star)$.

D'abord, on montre l'\'egalit\'e pour le terme constant.
D'apr\`es la proposition (\ref{q-deve}), le terme constant de $(\star)$ et $M^rF_{c,\alpha/M,
\beta/M}^{(r+1)}$ sont $A_0=\lim\limits_{n\ra+\infty}
\sum_{\substack{a_0\equiv\alpha[M]\\ 1\leq
a_0\leq Mp^n}}a_0^r(c^2\z(\frac{a_0}{Mp^n},0)-c\z(\frac{\langle c\rangle a_0}{Mp^n},0))$ et $B_0=M^r(c^2\z(\frac{\alpha}{M},-r)-c^{1-r}\z(\frac{\langle c\rangle\alpha}{M},-r))$ respectivement.

On note $\mu_c$ la mesure sur $\Z_p$ dont sa transform\'ee d'Amice est $\frac{c^2}{T}-\frac{c}{(1+T)^{c^{-1}}-1}$. Un calcul simple montre que, si $\alpha$ 
\begin{equation}
\begin{split}
\int_{\Z_p}(\frac{\alpha}{M}+x)^k\mu_c(x)&=(c^2\z(\frac{\alpha}{M},-k)-c^{1-k}\z(\frac{\langle c\rangle \alpha}{M},-k));\\
\int_{p^n\Z_p}(\frac{\alpha}{M}+x)^k\mu_c(x)&=(c^2\z(\frac{\alpha}{p^nM}),-k)-c^{1-k}\z(\frac{\langle c\rangle \alpha}{p^nM},-k)).
\end{split}
\end{equation}
On constate que $A_0$ est la somme de Riemann
\[\lim_{n\ra+\infty}\sum_{\substack{a_0\equiv\alpha[M]\\ 1\leq
a_0\leq Mp^n}}a_0^r(c^2\z(\frac{a_0}{Mp^n},0)-c\z(\frac{\langle c\rangle a_0}{Mp^n},0)),\]
et donc il se traduit l'int\'egration $p$-adique $M^r\int_{\Z_p}(\alpha/M+x)^r\mu_c$, qui est exactement $B_0$.

Ensuite, on se ram\`ene \`a montrer l'\'egalit\'e suivant:
\begin{equation}\label{passage} \lim\limits_{n\ra\infty}\sum_{\substack{a_0\equiv\alpha[M]\\ 1\leq a_0\leq Mp^n}}a_0^rE_1(q^{p^n},q_M^{a_0}\z_M^\beta)=M^rF_{\alpha/M,\beta/M}^{(r+1)}.\end{equation}
On note la limite \`a gauche par $(\star')$. 
En effet, si $(\ref{passage})$ est vrai, on a 
 \[\lim\limits_{n\ra\infty}\sum_{\substack{a_0\equiv\alpha[M]\\ 1\leq a_0\leq Mp^n}}a_0^rE_{c,1}(q^{p^n},q_M^{a_0}\z_M^\beta)=M^rc^2F^{(r+1)}_{\alpha/M,\beta/M}- \lim\limits_{n\ra\infty}\sum_{\substack{a_0\equiv\alpha[M]\\ 1\leq a_0\leq Mp^n}}a_0^rcE_{1}(q^{p^n},q_M^{ca_0}\z_M^{c\beta})\]
Comme $c\in\Z_p^{*}$, on le conclut en appliquant le $(\ref{passage})$ \`a
    \[\lim\limits_{n\ra\infty}\sum_{\substack{a_0\equiv\alpha[M]\\ 1\leq a_0\leq Mp^n}}a_0^r cE_1(q^{p^n},q_M^{ca_0}\z_M^{c\beta})=c^{-r}\lim\limits_{n\ra\infty}\sum_{\substack{a_0\equiv\alpha[M]\\ 1\leq a_0\leq Mp^n}}(ca_0)^r cE_1(q^{p^n},q_M^{ca_0}\z_M^{c\beta}).\]

 Revenons \`a d\'emontrer l'\'egalit\'e $(\ref{passage})$.   Il ne reste qu'\`a comparer les s\'eries de Dirichlet formelles associ\'ees aux $q$-d\'eveloppements de $(\star')$ et de $M^rF_{\alpha/M,\beta/M}^{(r+1)}$
 
Soit $F^{(1)}_{a_0/Mp^n,\beta/M}(q)=\sum_{m\in\Q^{+}}b_m q^m$. D'apr\`es la proposition (\ref{q-deve}), la s\'erie de Dirichlet formelle \`a coefficients dans $\Q^{\cycl}$ associ\'ee \`a $F^{(1)}_{a_0/Mp^n,\beta/M}(q^{p^n})$
v\'erifie:
\begin{equation}\label{formelle}
\begin{split}
\sum_{m\in\Q^{+}}\frac{b_m}{(p^nm)^s}
=p^{-ns}(\z(a_0/Mp^n,s)\z^{*}(\beta/M,s)-\z(-a_0/Mp^n,s)\z^{*}(-\beta/M,s)).
\end{split}
\end{equation}

Soit $\sum_{m\in \Q^{+}}b_{m,n} m^{-s}$ la s\'erie de Dirichlet formelle \`a coefficients dans $\Q^{\cycl}$ associ\'ee au $q$-d\'eveloppement
de $\sum_{\substack{a_0\equiv\alpha[M]\\ 1\leq
a_0\leq Mp^n}}a_0^rF^{(1)}_{a_0/Mp^n,\beta/M}(q^{p^n})$. Alors, de la formule $(\ref{formelle})$, on a:

\begin{equation}\label{fin}
\begin{split}
\sum_{m\in\Q^{*}_{+}}\frac{b_{m,n}}{m^s}=&\sum_{\substack{a_0\equiv\alpha[M]\\ 1\leq a_0\leq Mp^n}}a_0^rp^{-ns}(\z(\frac{a_0}{Mp^n},s)\z^{*}(\frac{\beta}{M},s)-\z(-\frac{a_0}{Mp^n},s)\z^{*}(-\frac{\beta}{M},s))\\
=&\sum_{\substack{a_0\equiv\alpha[M]\\ 1\leq a_0\leq Mp^n}}a_0^rp^{-ns}(\sum_{k=0}^{+\infty}\frac{1}{(k+\frac{a_0}{Mp^n})^s}\z^{*}(\frac{\beta}{M},s)
-\sum_{k=1}^{+\infty}\frac{1}{(k-\frac{a_0}{Mp^n})^s}\z^{*}(-\frac{\beta}{M},s))\\
=&\sum_{i=0}^{p^n-1}(\sum_{k=0}^{+\infty}\frac{(\alpha+iM)^r}{(kp^n+\frac{\alpha}{M}+i)^s}
\z^{*}(\frac{\beta}{M},s)-\sum_{k=1}^{+\infty}\frac{(\alpha+iM)^r}{(kp^n-\frac{\alpha}{M}-k)^s}
\z^{*}(-\frac{\beta}{M},s))
\end{split}
\end{equation}

Consid\'erons l'injection de $\Dir(\Q^{\cycl})$ dans $\Dir(\Q_p^{\cycl})$.
Alors la s\'erie de Dirichlet formelle, \`a coefficients dans $\Q^{\cycl}_p$, associ\'ee au $q$-d\'eveloppement de $(\star)$,
est la limite $p$-adique $\lim\limits_{n\ra\infty}\sum_{m\in\Q^{*}_{+}}\frac{b_{m,n}}{m^s}$ et on a 
\begin{equation*}
\begin{split}
\lim\limits_{n\ra\infty}\sum_{m\in\Q^{*}_{+}}\frac{b_{m,n}}{m^s}
=&\lim\limits_{n\ra\infty}\sum_{i=0}^{p^n-1}(\sum_{k=0}^{+\infty}
\frac{(\alpha+iM)^r}{(kp^n+\frac{\alpha}{M}+i)^s}\z^{*}(\beta/M,s)
-\sum_{k=1}^{+\infty}\frac{(\alpha+iM)^r}{(kp^n-\frac{\alpha}{M}-k)^s}\z^{*}(-\beta/M,s))
\\
=&M^r\sum_{i=0}^{p^{\infty}}(\sum_{k=0}^{\infty}\frac{(\frac{\alpha}{M}+i)^r}
{(\frac{\alpha}{M}+i+kp^{\infty})^s}\z^{*}(\beta/M,s)-
\sum_{k=1}^{+\infty}\frac{(\frac{\alpha}{M}+i)^r}
{(kp^{\infty}-\frac{\alpha}{M}-i)^s}\z^{*}(-\beta/M,s))\\
=&M^r(\sum_{j\equiv \frac{\alpha}{M}\mod \Z }\frac{1}{j^{s-r}}\z^{*}(\beta/M,s)+(-1)^{r+1}\sum_{j\equiv-\frac{\alpha}{M}\mod\Z}\frac{1}{j^{s-r}}\z^{*}(-\beta/M,s))\\
=&M^r(\z(\alpha/M,s-r)\z^{*}(\beta/M,s)+(-1)^{r+1}\z(-\alpha/M,s-r)\z^{*}(-\beta/M,s)).
\end{split}
\end{equation*}
C'est la m\^eme s\'erie de Dirichlet \`a coefficients dans $\Q_p^{\cycl}$ associ\'ee au $q$-d\'eveloppement de $M^rF_{\alpha/M,
\beta/M}^{(r+1)}$.

\end{proof}